\documentclass[11pt]{article}

\usepackage{amsmath, amssymb, amscd, amsthm, amsfonts}
\usepackage{graphicx}
\usepackage{hyperref}
\usepackage{indentfirst}
\usepackage{mathtools}
\usepackage{enumerate}
\usepackage{geometry}
\usepackage{gensymb}

\title{\textsc{\Large{\textbf{Gradient Existence and Energy Finiteness of Local Minimizers
				in the Wasserstein $L^\infty$ Topology for Binary-Star Systems}}}}

\author{Hangsheng Chen\thanks{Department of Mathematics, Statistics, and Computer Science, M/C 249, University of Illinois at Chicago, 851 S. Morgan Street, Chicago, IL 60607, USA. Email: hchen261@uic.edu}}

\date{ }

\newtheorem{theorem}{Theorem}

\newtheorem  {corollary} [theorem]{Corollary}

\newtheorem  {lemma} [theorem]{Lemma}
\newtheorem  {proposition}[theorem]{Proposition}

\theoremstyle{definition}
\newtheorem{definition}[theorem]{Definition}

\newtheorem  {remark} [theorem]{Remark}

\numberwithin{theorem}{section}

\makeatletter
\newcommand{\mylabel}[2]{#2\def\@currentlabel{#2}\label{#1}}
\makeatother

\begin{document}

\maketitle

\begin{abstract}\label{abstract}	
	In this paper, we refine and complement McCann's results on binary-star systems \cite{McC06}, a compressible fluid model governed by the Euler-Poisson equations. We consider a general form of the equation of state that includes polytropic gaseous stars indexed by  $\gamma$ as a special case. Beyond revisiting McCann's framework and conclusions—where solutions to the Euler-Poisson equations are obtained as local energy minimizers via variational methods under the topology induced by the Wasserstein  $L^\infty$  distance—we focus on a detailed exploration of the properties of local energy minimizers in this topology, addressing three key aspects: (1) the feasibility of transitioning from the Euler-Lagrange equation to the Euler-Poisson equation by demonstrating gradient existence; (2) the existence of  $L^\infty$  functions within neighborhoods in this topology; and (3) the finiteness of the energy of local minimizers in this topology, contrasted with the non-existence of finite-energy local minimizers and the existence of infinite-energy weak local minimizers in the topology inherited from topological vector spaces.
	
\end{abstract}

 \tableofcontents
 
\section{Introduction}\label{section-introduction}

\setlength{\parindent}{1.5em}

We consider uniformly rotating star systems, where we simply consider only the orbital revolution around each other and ignore the rotations of objects around their own axes. It is described by the reduced Euler-Poisson equation:
\begin{equation} \label{EP'}
	-\omega^2 \tilde{\rho}(x) P_{12}(x) + \nabla P(\tilde{\rho}(x)) - \tilde{\rho}(x) \nabla V_{\tilde{\rho}}(x) = 0,
	\tag{EP'}
\end{equation}
Here $\tilde{\rho}(x)\geq 0$ is a compactly supported density function, $\omega\geq0$ is the angular velocity, $P_{12}$ is the projection map given by $P_{12}=(x_1,x_2,x_3)=(x_1,x_2,0)$, and $V_{\tilde{\rho}}(x) = \int_{\mathbb{R}^3} \frac{\tilde{\rho}(y)}{|y-x|}\,dy$. The pressure $P(\tilde{\rho})$ depends on the density only, with some appropriate assumptions which will be described later.

This reduction for uniformly rotating configurations arises from the full Euler–Poisson system:
\begin{equation} \label{EP}
	\begin{split}
		\partial_{t}\rho + \nabla \cdot (\rho v) &= 0, \\
		\rho\partial_{t}v + \rho(v \cdot \nabla)v + \nabla P(\rho) &= \rho\nabla V, \\
		\Delta V &= -4\pi\rho,
	\end{split} \tag{EP}
\end{equation}
which serves as a fundamental hydrodynamic model for Newtonian stellar dynamics, describing the evolution of an isolated self‑gravitating fluid. Here $\rho(x,t)\geq0$ is the density of fluids, $t\geq0$ is the time, $v(x,t)\in\mathbb{R}^3$ is the velocity, and $-V(x,t)$ is the gravitational potential. The existence of single-star and multi-star systems under different settings has been established via variational methods (e.g., \cite{AB71M, AB71, Auc91, CF80, CL94, Li91, McC06}) and perturbative approaches (e.g., \cite{JM17, JM19, Hei94, SW17, SW19, AKV23, Lic33}).

As discussed by Jang and Seok in \cite{JS22} and in Subsection \ref{subsection2.2-uniform rotation minimizers kinetic energy} below, if $\tilde{\rho}$ satisfies the reduced Euler-Poisson equation (\ref{EP'}), then $\rho\left( {t,y} \right) = \tilde{\rho}\left( R_{- \omega t}y \right)$, $v\left( {t,y} \right) = \omega\left( {- y_{2},y_{1},0} \right)^{T}$, and $V(t, y)=V_{\widetilde{\rho}}\left(R_{-\omega t} y\right)$ satisfy the full system (\ref{EP}), where $R_{- \omega t}y$ is the rotation map given in Definition \ref{notations}.

In \cite{McC06}, McCann used a variational method to construct uniformly rotating binary‑star solutions to \eqref{EP'}. His binary stars exhibit a dual characterization: as Hamiltonian and energy minimizers under conservative constraints, and as perturbations of simpler objects like non-rotating Lane-Emden stars or relative equilibria for point masses (e.g., \cite[Section 3.2]{JS22}). This characterization is essential for showing that the support of the energy constrained minimizer $\rho$ is contained within a ball of a certain radius, making $\rho$ a local minimizer with respect to the topology induced by Wasserstein $L^\infty$ (denoted by $W^\infty$ hereafter) distance. As discussed in Theorem \ref{Properties of LEM} (see also \cite[Theorem 2.1]{McC06}), this ensures the existence of $\nabla P(\rho)$ in $\mathbb{R}^3$ and validates (\ref{EP'}) in the whole space. However, McCann does not thoroughly address the existence of $\nabla P(\rho)$. While McCann provides valuable insights into the $W^\infty$ distance, further exploration is warranted—particularly why the local minimizer in this topology has finite energy, enabling the existence of the variational derivatives. These are what we aim to refine in this paper. 

To be precise, this paper makes the following contributions:
\begin{enumerate}[(a)]\label{contributions}
	\item \label{contribution (a)} Gradient Existence and Equation Derivation: We establish the existence of $\nabla P(\rho)$ for the local minimizer $\rho$, which enables a rigorous transition from the Euler-Lagrange equation to the Euler-Poisson equation. The result is formalized in Theorem \ref{Properties of LEM} and proved in Section \ref{section3-detailed analysis}. 
	\item \label{contribution (b)} Existence of $L^\infty$ Functions: We prove Lemma \ref{properties of Wasser.} (vi), showing any neighborhood in the topology induced by $W^\infty$ distance contains $L^\infty$ functions.
	\item \label{contribution (c)} Energy Finiteness and Comparison of Topologies: We furthermore establish the finiteness of energy and the existence of variational derivatives for local minimizers in Lemma \ref{diff. at mini.}. We also discuss results in the alternative topology inherited from the topological vector space, where we show finite‑energy local minimizers do not exist (Proposition \ref{no local minimizers with finite energy in a vector space topology}), yet weak local minimizers with infinite energy can be found (Proposition \ref{existence of weak local minima when energy blows up}).
\end{enumerate}

Part of the results presented in this paper are based on the author’s Master’s thesis \cite{CVD24}. Beyond their intrinsic mathematical interest and largely self-contained exposition, these results serve as preliminary findings and will be used in the author’s companion paper \cite{Che26E3} on star–planet systems, which also originates from the same thesis. In \cite{Che26E3}, they serve as an alternative to direct citations of McCann’s conclusions \cite{McC06}, providing a more precise and rigorous foundation for the arguments therein.

The structure of this paper is outlined as follows: in Section \ref{section2-McCann's construction}, we review McCann's constructions and results of binary stars \cite{McC06}, with some nontrivial modifications. In Section \ref{section3-detailed analysis}, we discuss the necessity of those nontrivial modifications, provide nontrivial supplementary details regarding the properties of local minimizers under the topology induced by $W^\infty$ distance, particularly focusing on contribution (\ref{contribution (a)}). In Section \ref{section4-Wasserstein metric}, we address contribution (\ref{contribution (b)}), establishing the existence of $L^\infty$ functions in any neighborhood. In Section \ref{section5-constant chemical potential}, we accomplish contribution (\ref{contribution (c)}): we apply results in Section \ref{section4-Wasserstein metric} to deduce the finiteness of energy for local minimizers and compare it with the results in alternative topology.
\section{McCann's Construction and Results of 2-Body Systems}\label{section2-McCann's construction}

In this section we review McCann's construction and results of 2-Body systems \cite{McC06}. In the first subsection, we introduce some basic settings, notations, assumptions, and problem settings such as the variational formulation. In the second subsection, we focus on explaining why the general energy form $E(\rho,v)$ can be replaced by the energy with uniform rotation $E_J (\rho)$. In the third subsection, we present some of McCann's nice conclusions.

\subsection{Notations and problem setting}\label{subsection2.1-notations and problem setting}

We introduce some notations and problem settings based on McCann's one \cite{McC06} and Jang and Seok's \cite{JS22}. 

\begin{definition}[Notations] \label{notations}
	We give the following definitions:
	\begin{enumerate}[(i)]
		\item The projection operator of $x$ onto the $x_1 x_2$ plane: $P_{12}(x)=P_{12}\left(x_1, x_2, x_3\right):=\left(x_1, x_2, 0\right)$.
		\item A bilinear form $\langle \cdot , \cdot \rangle_2: \mathbb{R}^3 \times \mathbb{R}^3 \rightarrow \mathbb{R}: \forall x, y \in \mathbb{R}^3,\langle x, y\rangle_2=P_{12}(x) \cdot P_{12}(y)= x_1 y_1+x_2 y_2$.
		\item $r(x):=\left(\langle x, x\rangle_2\right)^{\frac{1}{2}}=\sqrt{x_1^2+x_2^2}$.
		\item Rotation map:
		$$
		R_\theta:=\left(\begin{array}{ccc}
			\cos \theta & -\sin \theta & 0 \\
			\sin \theta & \cos \theta & 0 \\
			0 & 0 & 1
		\end{array}\right)
		$$
		\item $W^{\infty}$ metric: Wasserstein $L^{\infty}$ metric.
	\end{enumerate}
\end{definition}

The state of a fluid may be represented by its mass density $\rho(x) \geq 0$ and velocity vector field $v(x)$ on $\mathbb{R}^3$. The fluid interacts with itself through Newtonian gravity; hence, we need to consider gravitational interaction energy, which will be given later. Moreover, to define internal energy, we follow Auchmuty and Beals's assumptions \cite{AB71} and first consider a general form of the pressure $P(\rho)$ as the following:

\begin{enumerate}
	\item [\mylabel{F1}{(F1)}] $P:[0, \infty) \rightarrow[0, \infty)$ is continuous and strictly increasing;
	\item [\mylabel{F2}{(F2)}] $\lim\limits_{\rho \rightarrow 0} P(\rho) \rho^{-\frac{4}{3}}=0$;
	\item [\mylabel{F3}{(F3)}] $\lim\limits_{\rho \rightarrow \infty} P(\rho) \rho^{-\frac{4}{3}}=\infty$.
\end{enumerate}

With these assumptions, we also define $A(\rho)$ as the following:
\begin{equation}\label{A}
	A(\rho):=\int_1^{\infty} P\left(\frac{\rho}{v}\right) d v=\rho \int_0^\rho P(\tau) \tau^{-2} d \tau
\end{equation}

As shown in subsection \ref{subsection2.2-uniform rotation minimizers kinetic energy}, the integral definition of $A(\rho)$ helps us obtain (\ref{EL0'}). Due to \ref{F2}, we know $A(\rho)$ is well-defined. Easy to see $A(\rho)$ is strictly increasing. It turns out that $A(\rho)$ is also convex (see Lemma \ref{inc. of A'}), and it is related to the internal energy $U(\rho)$ given in (\ref{U}) below. Moreover, if the polytropic equation of state $P(\rho)=K\rho^\gamma$ holds, where the parameter $\gamma > \frac{4}{3}$,  then easy to check $P$ satisfies \ref{F1} \ref{F2} \ref{F3}, and
\begin{equation}\label{AwithP}
	A(\rho)=\frac{K}{\gamma-1} \rho^\gamma
\end{equation}

\begin{remark}\label{collapse}
	When showing the existence of non-rotating single stars, \ref{F3} can be replaced by the assumption
	\begin{equation}\label{F3'}
		\lim\limits_{\rho \rightarrow \infty} \inf P(\rho) \rho^{-\frac{4}{3}}>K \tag*{(F3')}
	\end{equation}
	for some $K>0$. Usually, $K$ is supposed to be large enough to prevent gravitational collapse, and thereby the Chandrasekhar mass for the model is assumed to be greater than $m=1$. Then we can obtain minimizers with total mass not larger than 1 via variational method. See \cite[Section 2]{McC06} or \cite[Section 1,6,8 and Appendix]{AB71}.
\end{remark}

\begin{remark} \label{assumptions from P to A}
	By L'Hôpital's rule \cite[Theorem 5.13]{Rud76}, and the expression of $A(s)=s \int_0^s P(t) t^{-2} d t$ in (\ref{A}), we can prove that $A$ also satisfies \ref{F2} and \ref{F3} (or \ref{F3'} if $P$ satisfies \ref{F3'} instead of \ref{F3}, though the constant $K$ can be different).
\end{remark}

By (\ref{A}), we can obtain $A'(s)$ satisfies
\begin{equation}\label{A'}
	A^{\prime}(s)=\left\{\begin{array}{r}\int_0^s P(t) t^{-2} d t+\frac{P(s)}{s}, s>0 \\ 0, s=0\end{array}\right.
\end{equation}
Note that $A'(s)$ is continuous and 
\begin{equation}\label{A', A and P relation}
	A^{\prime}(s) s-A(s)= P(s)
\end{equation}

\begin{remark} \label{asymptotic behaviour for A'}
	By (\ref{A'}), we can also prove that $A^{\prime}$ satisfies: $\lim\limits_{\rho \rightarrow 0} A^{\prime}(\rho) \rho^{-\frac{1}{3}}=0$ and $\lim\limits_{\rho \rightarrow \infty} A^{\prime}(\rho) \rho^{-\frac{1}{3}}=\infty$.
\end{remark}


\begin{lemma}[Strictly Increasing of $A^{\prime}$] \label{inc. of A'}
	If $P(\rho)$ satisfies assumptions \ref{F1} \ref{F2} \ref{F3}, then $A'$ is continuous and strictly increasing, where $A$ is given in (\ref{A}). Moreover, $A$ is convex.
\end{lemma}

\begin{proof}
	{\rm Due to (\ref{A'}), we can see $A'$ is continuous and $A^{\prime}(s)> A^{\prime}(0)=0$ when $s>0$. Given $s>0$ and $h>0$, by (\ref{A'}) we have 
	$$A^{\prime}(s+h)-A^{\prime}(s)= \int_s^{s+h} P(t) t^{-2} d t+\frac{P(s+h)}{s+h}-\frac{P(s)}{s}$$ Since $P$ is monotonic (strictly increasing), by the second mean value theorem for definite integrals \cite{Hob08}, we know $\exists \xi \in(0, h)$, 
%
	\begin{align*}
		\int_s^{s+h} P(t) t^{-2} \, dt &= P(s) \int_s^{s+\xi} t^{-2} \, dt + P(s+h) \int_{s+\xi}^{s+h} t^{-2} \, dt \\
		&= P(s)\left(\frac{1}{s}-\frac{1}{s+\xi}\right) + P(s+h)\left(\frac{1}{s+\xi}-\frac{1}{s+h}\right)
	\end{align*}
	Then $A^{\prime}(s+h)-A^{\prime}(s)=\frac{P(s+h)-P(s)}{s+\xi}>0$, which means $A^{\prime}$ is strictly increasing on $[0, \infty)$. The convexity of $A$ can refer to \cite[Chapter 5, Exercise 14]{Rud76}}.
\end{proof}

\begin{remark} \label{existence of inverse}
	Thanks to Lemma \ref{inc. of A'}, the inverse function of $A^{\prime}$, denoted by $\left(A^{\prime}\right)^{-1}$ or $\phi$, is well-defined on $[0, \infty)$. Moreover, $\phi=\left(A^{\prime}\right)^{-1}$ is continuous since $A^{\prime}$ is continuous.
\end{remark}


Note $\nabla P(\rho)$ is one term in the Euler-Poisson equation \eqref{EP} (or \eqref{EP'}). To show the existence of $\nabla P(\rho)$, we make an additional assumption of $P(\rho)$:

\begin{enumerate}
	\item[\mylabel{F4}{(F4)}] $P(\rho)$ is continuously differentiable on $[0, \infty)$, and $P^{\prime}(\rho)>0$ if $\rho>0$.
\end{enumerate}

\begin{remark} \label{diff. of inverse}
	
	By \ref{F4}, we know $A^{\prime \prime}(\rho)$ exists and $A^{\prime \prime}(\rho)=\frac{P^{\prime}(\rho)}{\rho} \neq 0$ is continuous if $\rho>0$, which implies $\phi=\left(A^{\prime}\right)^{-1} \in C^1((0, \infty))$ (one can refer to \cite[Theorem 10.4.2]{Tao16}). It turns out this can help to show the differentiability of $\sigma$ and then of $P(\sigma)$ where $\sigma>0$, as we can see in Theorem \ref{Properties of LEM} below.
\end{remark}

Since what we want is the differentiability of $P(\sigma)$ rather than of $\sigma$ to show (\ref{EP'}), another alternative assumption, which is more general but also somewhat technical, is: 
\begin{enumerate}
	\item[\mylabel{F4'}{(F4')}] $P(\rho)$ is continuously differentiable on $[0, \infty)$. If $\rho>0, P(\rho)$ has non-vanishing (first order or higher order) derivative at $\rho$. That is, $\exists n \geq 1$, such that $P^{(n)}(\rho)$ exists and is not 0.
\end{enumerate}
The applications of this more general assumption will be demonstrated later, such as in Lemma \ref{differentiability of composition} and Theorem \ref{Properties of LEM} (vii).

\bigskip
In the following sections, we assume $P(\rho)$ satisfies \ref{F1}\ref{F2}\ref{F3} unless otherwise specified. We will mention \ref{F4} or \ref{F4'} or other assumptions when we want to use them. As a special case, if we assume the polytropic equations of state with index $\gamma>\frac{4}{3}$ holds, one can check the assumptions \ref{F1}\ref{F2}\ref{F3}\ref{F4} or \ref{F1}\ref{F2}\ref{F3}\ref{F4'} above are satisfied automatically.

We consider ``admissible classes'' for $\rho$ and $v$ as the following:
\begin{align}
	R\left(\mathbb{R}^3\right) & :=\left\{\left.\rho \in L^{\frac{4}{3}}\left(\mathbb{R}^3\right) \right\rvert\, \rho \geq 0, \int_{\mathbb{R}^3} \rho\,dx=1\right\} \\
	V\left(\mathbb{R}^3\right) & :=\left\{v: \mathbb{R}^3 \rightarrow \mathbb{R}^3 \mid v \text { is measurable. }\right\}
\end{align}

We can choose units so that the total mass of the fluid is one and the gravitational constant $G=1$. Then, given $\rho$ and $v$ in the sets above, $V_\rho$ represents the \textit{gravitational potential} of the mass density $\rho(x)$:
\begin{equation}\label{V}
	V_\rho(x):=\int_{\mathbb{R}^3}  \frac{\rho(y)}{|y-x|} \,dy
\end{equation}
Moreover, McCann \cite{McC06} considers energy $E(\rho, v)$ consisting of three terms:

\begin{equation}\label{energy}
	E(\rho, v):=U(\rho)-\frac{G(\rho, \rho)}{2}+T(\rho, v) 
\end{equation}

\begin{equation} \label{U}
	U(\rho):=\int_{\mathbb{R}^3}  A(\rho(x)) \,dx
\end{equation}

\begin{equation} \label{G}
	G(\sigma, \rho):=\int_{\mathbb{R}^3}  V_\sigma \rho\,dx=\iint_{\mathbb{R}^3\times \mathbb{R}^3}  \frac{\sigma(x) \rho(y)}{|x-y|} dy dx 
\end{equation}

\begin{equation} \label{T}
	T(\rho, v):=\frac{1}{2} \int_{\mathbb{R}^3} |v|^2 \rho \,dx
\end{equation}

Here $A(\rho)$ is a convex function given in (\ref{A}), and $U(\rho)$ is the \textit{internal energy}. $G$ is viewed as a symmetric bilinear form since this observation can help us to compute the variational derivative in Lemma \ref{diff. of energy} in subsection \ref{subsection5.1-variational derivative}. Then the \textit{gravitational potential energy} (also called \textit{gravitational interaction energy}) is defined as $G(\rho, \rho)$. We call $T(\rho, v)$ the \textit{kinetic energy}.



We can choose a frame of reference in which the \emph{center of mass}
\begin{equation}\label{centerofmass}
	\bar{x}(\rho):=\frac{\int_{\mathbb{R}^3} x \rho(x) \,dx}{\int_{\mathbb{R}^3} \rho(x) \,dx}
\end{equation}
is at rest. We are interested in finding minimum energy configurations subject to constraints of fixed (but small) mass ratio and fixed angular momentum $\boldsymbol{J}$ with respect to the center of mass $\bar{x}(\rho)$. The fluid \textit{angular momentum} is given by $\boldsymbol{J}(\rho, v)$:
\begin{equation} \label{angular momentum}
	\boldsymbol{J}(\rho, v):=\int_{\mathbb{R}^3}(x-\bar{x}(\rho)) \times v \rho(x) \,dx
\end{equation}

We denote by $J_z$ the z-component of $\boldsymbol{J}$, that is, $J_z(\rho, v):=\hat{e}_z \cdot \boldsymbol{J}(\rho, v)$, where $\hat{e}_z=(0,0,1)^T$. For simplicity of notation, we will sometimes use  $J$  to represent $J_z$ and call $J$ the angular momentum of the system when no confusion arises.

Since the z-component of the angular momentum is specified, the moment of inertia $I(\rho)$ of $\rho$ in the direction of $\hat{e}_z$ will be relevant. That is, we define $I(\rho)$ to be \textit{moment of inertia} of $\rho$ in the direction of $\hat{e}_z$:
\begin{equation} \label{Moment of Inertia}
	I(\rho):=\int_{\mathbb{R}^3} \rho r^2(x-\bar{x}(\rho))\,dx=\int_{\mathbb{R}^3} \rho(x)\left( (x_1-\bar{x}(\rho)_1)^2+(x_2-\bar{x}(\rho)_2)^2\right)\,dx
\end{equation}
where $r$ is given in Definition \ref{notations}.

\begin{remark} \label{positive MoI}
	When $\rho$ has positive mass, we have $I(\rho)>0$. If not, $I(\rho)=\int_{\mathbb{R}^3} \rho(x) r^2(x-\bar{x}(\rho))\,dx=0$ implies $\rho(x) r^2(x-\bar{x}(\rho))=0$ almost everywhere in $\mathbb{R}^3$. Since $\mu\left(\left\{x \in \mathbb{R}^3 \mid r^2(x-\bar{x}(\rho))=0\right\}\right)=0$, where $\mu$ denotes the Lebesgue measure, we have $\rho(x)=0$ almost everywhere in $\mathbb{R}^3$, which contradicts the fact that $\rho$ has positive mass.
\end{remark}

A simple case is the non-rotating problem, i.e., $\boldsymbol{J}=0$, with energy $E_0(\rho)=U(\rho)-\frac{G(\rho, \rho)}{2}$.
Properties of non-rotating minimizers are introduced in \cite{Che26R2, McC06, AB71, LY87}.

To ensure that the definition of $E(\rho,v)$ (or $E_J (\rho)$ given later) does not result in a situation of $\infty-\infty$, we hope the gravitational potential energy $G(\rho,\rho)$ to be finite. Thanks to our setup, this is indeed the case, which leads to the following proposition:

\begin{proposition}[Finite Gravitational Interaction Energy]\label{finite gravitational interaction energy}
	Suppose $\rho \in L^{\frac{4}{3}}\left( \mathbb{R}^{3} \right) \cap L^{1}\left( \mathbb{R}^{3} \right)$, with $\left\| \rho \right\|_{L^{\frac{4}{3}}{(\mathbb{R}^{3})}} + \left\| \rho \right\|_{L^{1}{(\mathbb{R}^{3})}} < \lambda$. Then there is a constant $C(\lambda)$ which depends only on $\lambda$, such that $G(\rho,\rho)<C(\lambda)$.
\end{proposition}

\begin{proof}
	{
		\rm
		Since $\rho \in L^{\frac{4}{3}}\left(\mathbb{R}^3\right) \cap L^1\left(\mathbb{R}^3\right)$, by Hardy-Littlewood-Sobolev Inequality \cite[Theorem 1.7]{BCD11} we have $V_\rho \in L^{12}\left(\mathbb{R}^3\right)$, with $\left\|V_\rho\right\|_{L^{12}\left(\mathbb{R}^3\right)}<C_1(\lambda)$ for some constant $C_1(\lambda)>0$. By Interpolation Inequality \cite[Section 4.2]{Bre11} we have $\rho \in L^{\frac{12}{11}}\left(\mathbb{R}^3\right)$, with $\|\rho\|_{L^{\frac{12}{11}\left(\mathbb{R}^3\right)}}<C_2(\lambda)$ for some constant $C_2(\lambda)>0$. Finally, by Hölder's inequality we have 
		$$
		G(\rho, \rho)=\int_{\mathbb{R}^3} \rho V_\rho\,dx \leq\|\rho\|_{L^{\frac{12}{11}\left(\mathbb{R}^3\right)}}\left\|V_\rho\right\|_{L^{12}\left(\mathbb{R}^3\right)}<C_1(\lambda)\cdot C_2(\lambda) \coloneq C(\lambda)
		$$ 
	}
\end{proof}

\begin{remark}\label{finite gravitational interaction energy with different objects}
	With similar arguments, we can show a more general result: suppose $\rho \in L^{\frac{4}{3}}\left( \mathbb{R}^{3} \right) \cap L^{1}\left( \mathbb{R}^{3} \right)$ and $\sigma  \in L^{\frac{12}{11}}\left( \mathbb{R}^{3} \right)$, then $G(\rho, \sigma)<\infty$.
\end{remark}

\begin{remark}
	From Proposition \ref{finite gravitational interaction energy}, we see one reason why we require $\rho$ to belong not only to $L^1$ but also to $L^{\frac{4}{3}}$. Furthermore, the rationale behind selecting the specific exponent $\frac{4}{3}$ can also refer to McCann's construction, which is a natural assumption by interpolation inequality \cite[Section 4.2]{Bre11} if we hope $\inf\limits_{\rho \in {R}\left(\mathbb{R}^3\right)} E_0(m \rho)$ to be finite, see \cite[Section 1, 6, and 8]{AB71}, and \cite[Section 2]{McC06}.
\end{remark}

Consider the non-rotating minimizer $\sigma_m$ of $E_0(\rho)$ among configurations of mass $m$, the corresponding minimum energy is finite due to the remark above. For the sake of convenience, we denote them by
\begin{equation}\label{e_0}
	e_0(m):=E_0\left(\sigma_m\right)=\inf _{\rho \in {R}\left(\mathbb{R}^3\right)} E_0(m \rho).
\end{equation}
Sometimes $e_0(1)$ is denoted by $e_0$.

\bigskip
Physically speaking, if quantum mechanical effects are not considered, it is desirable for stars to be compactly supported. Therefore, we also introduce a subset $R_a\left(\mathbb{R}^3\right)$ of $R\left(\mathbb{R}^3\right)$:

\begin{equation}\label{R_a}
	{R}_a\left(\mathbb{R}^3\right):=\left\{\rho \in {R}\left(\mathbb{R}^3\right) \mid \bar{x}(\rho)=a ; \text { spt } \rho \text { is bounded.}\right\}
\end{equation}
Here, the support of $\rho$, denoted by spt $\rho$, is the smallest closed set carrying the full mass of $\rho$. In general, it is easier to consider the case that the density is centered at 0, which is in ${R}_0\left(\mathbb{R}^3\right)$.

We also consider uniform rotation with the angular momentum $\boldsymbol{J}=J\hat{e}_z=(0,0,J)^T$ specified a priori. In this case, the energy is given by $E_J(\rho)$ as follows: 
\begin{equation} \label{energy of UR}
	E_J(\rho)=U(\rho)-\frac{G(\rho, \rho)}{2}+T_J(\rho)
\end{equation}
$$T_J(\rho):=\frac{J^2}{2 I(\rho)}$$

Due to this uniform rotation observation, after fixing a time, we can assume that the components of the 2-body system fall within two disjoint regions $\Omega_m$ and $\Omega_{1-m}$, widely separated relative to $\frac{J^2}{\mu_r^{2}}$, where $\mu_r=m(1-m)$ is their reduced mass. For the planet's mass $m \in(0,1)$ and the star's mass $(1-m)$, we consider $E_J(\rho)$ to be minimized subject to the constraint
\begin{equation} \label{admissible class}
	W_{m,J}\coloneq \left\{\rho(m)=\rho_m+\rho_{1-m} \in R\left(\mathbb{R}^3\right) \mid \int_{\mathbb{R}^3} \rho_m\,dx=m, \text{spt } \rho_m \subset \Omega_m, \text{spt } \rho_{1-m} \subset \Omega_{1-m}\right\}
\end{equation}
where $\Omega_m$ and $\Omega_{1-m}$ are subsets of $\mathbb{R}^3$, which are given in the following. The subscripts $m$ and $J$ in $W_{m,J}$ indicate $\Omega_m$ and $\Omega_{1-m}$ are related to $m$ and $J$.

Fix two points $y_m$ and $y_{1-m}$ in $\mathbb{R}^3$ from the plane $z=0$, which are separated by $\eta=\frac{J^2}{\mu_r^2}$, i.e. $\eta=\left|y_m-y_{1-m}\right|$. The $\Omega_m$ and $\Omega_{1-m}$ are defined as closed balls in $\mathbb{R}^3$ centered at $y_m$ and $y_{1-m}$, whose size and separation scale with $\eta$ as the following:

\begin{equation}\label{domains}
	\begin{aligned}
		\Omega_m&:=\left\{x \in \mathbb{R}^3 \mid | x-y_m |\, \leq \frac{\eta}{4}\right\} \\
		\Omega_{1-m}&:=\left\{x \in \mathbb{R}^3\mid | x-y_{1-m} |\, \leq \frac{\eta}{4}\right\}
	\end{aligned}
\end{equation}

The distance separating $\Omega_m$ and $\Omega_{1-m}$, and the diameter of their union are given by:

\begin{align}
	{dist}\left(\Omega_m, \Omega_{1-m}\right) &=\frac{\eta}{2} \label{dist} \\
	{diam}\left(\Omega_m, \Omega_{1-m}\right) & =\frac{3 \eta}{2} \label{diam}
\end{align}

\begin{remark}
	The reason we set the separation $\eta=\frac{J^2}{\mu_r^2}$ in the definitions above is inspired by the Kepler problem. Given two point masses $m$ and $1-m$, rotating with angular momentum $J>0$ about their fixed center of mass, if we assume their separation is $d$, then the gravitational energy plus the kinetic energy is $-\frac{\mu_r}{d}+\frac{J^2}{2 \mu_r d^2}$, which reaches its minimum at separation $d=\eta$. When $\eta$ is large (which occurs when $\mu_r$ is small or $J$ is large), since the gravitational interaction between two bodies seems small, one can expect a stable, slowly rotating equilibrium to exist in which fluid components with masses $m$ and $1-m$ lie near $y_m$ and $y_{1-m}$. It is indeed true as McCann discussed in \cite[Section 6]{McC06}.
\end{remark}

\begin{remark}\label{supports are in the interior}
	We choose that the radii of $\Omega_m$ and $\Omega_{1-m}$ increase as $\eta$ increases in order to guarantee the supports of constrained minimizers (stars or planets) will fall in the interior of $\Omega_m \cup \Omega_{1-m}$. This is because McCann shows the size of stars or planets will not expand too much as $\eta$ increases. As discussed in \cite[Section 6]{McC06}, it lays the groundwork for discussion on the constrained minimizers being local minimizers (here and in what follows, unless otherwise specified, “local” will always refer to “local under the topology induced by the Wasserstein $L^{\infty}$ distance”). Therefore, the minimizers are solutions to (\ref{EP'}) due to Theorem \ref{Properties of LEM}.
\end{remark}

For prescribed angular momentum, we also know the energy $E(\rho, v)$ is bounded from below on ${R}_0\left(\mathbb{R}^3\right)$ by the non-rotating energy $e_0$. However, as in Morgan \cite{Mor02}, McCann \cite[Example 3.6]{McC06} demonstrates that this bound --- although approached --- will not be attained (see also \cite{McC94}). We clarify this phenomenon in the following remark:

\begin{remark}\label{global energy minimum approached but not attained}
	(i) Non-attainability: By the definition of $E_J(\rho)$, we know $\forall \rho \in {R}\left(\mathbb{R}^3\right), E_J(\rho) \geq E_0(\rho) \geq e_0(1)=e_0$. But $e_0$ cannot be attained by $E_J(\rho)$ if $J>0$. This is because $E_J(\rho)=E_0(\rho)+T_J(\rho)$. If $E_0(\rho)=e_0$, then $\rho$ is a minimizer of $E_0$ and hence has compact support \cite[Theorem 3.5 (v)]{McC06}. It makes $I(\rho)>0$ thus $T_J(\rho)>0$, $E_J(\rho)>e_0$; if $E_0(\rho)>e_0$, then, of course, $E_J(\rho)>e_0$.
	
	(ii) Approachability: It turns out $e_0$ is a strictly concave function of the mass $m$ \cite[Theorem 3.5 (ii)]{McC06}. Hence, one can construct $\rho \in R_0\left({R}^3\right)$ such that $E_J(\rho)-e_0<\epsilon$ where $\epsilon$ is arbitrarily small, and the test function can be chosen as $\rho(x)=\sigma_m(x)+\sigma_{1-m}(x-y)$ with $|y|$ sufficiently large, where $\sigma_m$ and $\sigma_{1-m}$ denote the non-rotating minimizers with mass $m$ and $1-m$ respectively. The idea comes from Lieb, Morgan, and McCann \cite{LM, McC94, Mor02, McC06}. 
%
\end{remark}
	%
Therefore, instead of searching for a global energy minimizer in binary-star system, one is forced to settle for local minimizers of $E(\rho, v)$ in an appropriate topology.

In fact, the velocity field $v$ may be topologized in any way which makes ${V}\left(\mathbb{R}^3\right)$ a topological vector space, since the velocity distribution minimizing the kinetic energy is always given by uniform rotation \cite[Section 3]{McC06} \cite [Lemma 2.2.2]{JS22} as we will describe in subsection \ref{subsection2.2-uniform rotation minimizers kinetic energy}. However, the choice of topology for ${R}_0\left(\mathbb{R}^3\right)$ is quite delicate: McCann pointed out local energy minimizer will not exist if the topology of $R_0(\mathbb{R}^3)$ is inherited from a topological vector space (\cite[Remark 3.7]{McC06}), i.e. ${R}_0\left(\mathbb{R}^3\right)$ is a subset of a topological vector space (see \cite[Chapter 1, Section 16]{Mun00}). While we partially acknowledge this perspective, we argue that the focus should be on $R(\mathbb{R}^3)$ rather than on $R_0(\mathbb{R}^3)$. For further details, one can refer to Proposition \ref{no local minimizers with finite energy in a vector space topology} and Remark \ref{R instead R_0} in subsection \ref{subsection5.2-locally constant chemical potential}.  We will further compare this case with the case that the topology induced by the Wasserstein $L^\infty$ distance, to highlight the advantages of considering the topology induced by the Wasserstein $L^\infty$ distance in such problems.

Since the star and planet are separated, it is convenient to describe the relations between the total moment of inertia $I\left(\rho_m+\rho_{1-m}\right)$ and $I\left(\rho_m\right), I\left(\rho_{1-m}\right)$.

\begin{lemma}[Expansion of Moment of Inertia]\label{expansion of MoI}
	Let $\rho\geq 0$, $\sigma \geq 0$ be the density functions in $\mathbb{R}^3$ with mass $\int_{\mathbb{R}^3} \rho\,dx=m_1<\infty, \int_{\mathbb{R}^3} \sigma\,dx=m_2<\infty$. $\bar{x}(\rho)$ and $\bar{x}(\sigma)$ denote the centers of mass, $I(\rho)$ and $I(\sigma)$ denote the moments of inertia, function $r$ is given in (\ref{notations}).
	
	\begin{enumerate}[(1)]
		\item If $m_1+m_2=0$, then $I(\rho+\sigma)=0$.
		\item If $m_1+m_2>0$, then we have the moment of inertia of $\rho+\sigma$ satisfies
		\begin{equation} \label{Expansion of MoI}
			I(\rho+\sigma)=I(\rho)+I(\sigma)+\frac{m_1 m_2}{m_1+m_2} r^2(\bar{x}(\rho)-\bar{x}(\sigma))
		\end{equation}
	\end{enumerate}
\end{lemma}

\begin{proof}		
	{\rm The case (1) is obvious, since $m_1+m_2=0$ implies $\rho+\sigma$ is 0 almost everywhere. For the case (2), by definition $\left(m_1+m_2\right) \bar{x}(\rho+\sigma)=\int_{\mathbb{R}^3}(\rho+\sigma) x\,dx=m_1 \bar{x}(\rho)+ m_2 \bar{x}(\sigma)$, then
		$$
		\begin{aligned}
			I(\rho+\sigma)&=\int_{\mathbb{R}^3}(\rho+\sigma) r^2(x-\bar{x}(\rho+\sigma))\,dx \\
			& =\int_{\mathbb{R}^3} \rho r^2(x-\bar{x}(\rho)+\bar{x}(\rho)-\bar{x}(\rho+\sigma))\,dx+\int_{\mathbb{R}^3} \sigma r^2(x-\bar{x}(\sigma)+\bar{x}(\sigma)-\bar{x}(\rho+\sigma))\,dx
		\end{aligned}
		$$
		Consider the first term on the right above, we have
		$$
		\begin{aligned}
			\int_{\mathbb{R}^3} \rho r^2(x-\bar{x}(\rho)+\bar{x}(\rho)-\bar{x}(\rho+\sigma))\,dx =&\int_{\mathbb{R}^3} \rho r^2(x-\bar{x}(\rho))\,dx+\int_{\mathbb{R}^3} \rho r^2(\bar{x}(\rho)-\bar{x}(\rho+\sigma))\,dx \\
			& +2 \int_{\mathbb{R}^3} \rho\langle x-\bar{x}(\rho), \bar{x}(\rho)-\bar{x}(\rho+\sigma)\rangle _2\,dx
		\end{aligned}
		$$
		Notice by the definition of the center of mass, we have 
		$$\int_{\mathbb{R}^3} \rho \langle  x, \bar{x}(\rho)-\bar{x}(\rho+\sigma)\rangle_2\,dx=m_1\langle\bar{x}(\rho), \bar{x}(\rho)-\bar{x}(\rho+\sigma)\rangle_2=\int_{\mathbb{R}^3} \rho\langle\bar{x}(\rho), \bar{x}(\rho)-\bar{x}(\rho+\sigma)\rangle_2\,dx$$
		Thus $2 \int_{\mathbb{R}^3} \rho\langle x-\bar{x}(\rho), \bar{x}(\rho)-\bar{x}(\rho+\sigma)\rangle_2\,dx=0$, therefore,
		$$
		\begin{aligned}
			\int_{\mathbb{R}^3} \rho r^2(x-\bar{x}(\rho)+\bar{x}(\rho)-\bar{x}(\rho+\sigma))\,dx & =I(\rho)+m_1 r^2\left(\frac{m_2}{m_1+m_2}(\bar{x}(\rho)-\bar{x}(\sigma))\right) \\
			& =I(\rho)+\frac{m_1 \cdot m_2^2}{\left(m_1+m_2\right)^2} r^2(\bar{x}(\rho)-\bar{x}(\sigma))
		\end{aligned}
		$$
		Similarly
		$$
		\int_{\mathbb{R}^3} \sigma r^2(x-\bar{x}(\sigma)+\bar{x}(\sigma)-\bar{x}(\rho+\sigma))\,dx=I(\sigma)+\frac{m_2 \cdot m_1^2}{\left(m_1+m_2\right)^2} r^2(\bar{x}(\rho)-\bar{x}(\sigma))
		$$
		Thus
		$$
		\begin{aligned}
			I(\rho+\sigma)&=I(\rho)+\frac{m_1 \cdot m_2^2}{\left(m_1+m_2\right)^2} r^2(\bar{x}(\rho)-\bar{x}(\sigma))+I(\sigma) +\frac{m_2 \cdot m_1^2}{\left(m_1+m_2\right)^2} r^2(\bar{x}(\rho)-\bar{x}(\sigma)) \\
			&=I(\rho)+I(\sigma)+\frac{m_1 m_2}{m_1+m_2} r^2(\bar{x}(\rho)-\bar{x}(\sigma))
		\end{aligned}
		$$
		This is what we want.}
\end{proof}

\begin{remark}
	Physically it's natural to assume those mass, center of mass, and moment of inertia are finite. However, one may notice as long as the mass $m_1$ and $m_2$ are finite and one of the centers of mass is finite, the result can still hold even though the values may be infinite.
\end{remark}

\subsection{Uniform rotation minimizes kinetic energy}\label{subsection2.2-uniform rotation minimizers kinetic energy}

This subsection reviews several results of McCann \cite{McC06}, Jang and Seok \cite{JS22}. First, we mention that the problem of minimizing $E(\rho, v)$ is equivalent to a minimization problem of the energy $E_J(\rho)$ with respect to uniform rotation, that is, the fluid rotates uniformly about its center of mass. Then we show that with uniform rotation setting, the conservation of mass $\partial_t \rho+\nabla(\rho v)=0$ automatically holds true, and then the Euler-Poisson equations (\ref{EP}) can be simplified to (\ref{EP'}) as Jang and Seok describe \cite{JS22}.

When we define the total energy $E(\rho, v)$, we can notice that the velocity field $v$ only appears in the kinetic energy $T(\rho, v)$. Therefore, we can first consider the velocity minimizer of $T(\rho, v)$ with a given density $\rho \in {R}_0\left(\mathbb{R}^3\right)$. McCann points out that $T(\rho,v)$ is uniquely minimized by a uniform rotation:

\begin{proposition}[Uniform Rotation around Center of Mass {\cite{LM} \cite[Proposition 3.1]{McC06}}] \label{uniform rotation}
	Fix a fluid density $\rho \in {R}_0\left(\mathbb{R}^3\right)$ and $J \geq 0$. Among all velocities $v \in {V}\left(\mathbb{R}^3\right)$ for which $T(\rho, v)<\infty$ and satisfying the constraint ${J}_z(\rho, v)=J$, the kinetic energy $T(\rho, v)$ is uniquely minimized by a uniform rotation $v(x):=\omega \hat{e}_z \times x$ with angular velocity $\omega:=\frac{J}{I(\rho)}$.
\end{proposition}

\begin{proof}
	{
		\rm
		Given $\rho$ in ${R}_0\left(\mathbb{R}^3\right)$ fixed, we define the Hilbert space $H \coloneq L^{2}\left( {\mathbb{R}^{3},d\rho(x)} \right) \subset V\left( \mathbb{R}^{3} \right)$, where $d\rho(x) = \rho(x)dx$, with inner product $< v,w >_{H} \coloneq \int_{\mathbb{R}^3}(v\cdot w) \rho\,dx$, in particular $< v,v >_{H} = 2T(\rho,v)$. One can check: (1) $\hat{e_{z}} \times x \in H$; (2) if $v \in {V}\left(\mathbb{R}^3\right)$ satisfies $T(\rho, v)<\infty$ and the constraint ${J}_z(\rho, v)=J$, then it is in the space $G = \left\{ v \in H \middle| < v,~\hat{e_{z}} \times x >_{H} = J \right\}$; (3) if $v$ is in $G$, then $v$ can be expressed as $v = v_{g} + v_{o}$, where $v_{g} =\frac{J}{I(\rho)} \left( \hat{e_{z}} \times x \right)= \omega\left( \hat{e_{z}} \times x \right)$ and $v_o$ is orthogonal to $\hat{e_{z}} \times x$ in $H$, that is $< v_{0},\hat{e_{z}} \times x >_{H} = 0$. Therefore, 
		
		$$
		\begin{aligned}
			< v,v >_{H} &= < v_{g},v_{g} >_{H} + < v_{0},v_{0} >_{H} \\
			&\geq < v_{g},v_{g} >_{H} \\
			&= 2T\left( \rho,v_{g} \right)
		\end{aligned}
		$$
		For the detailed proof, one can refer to \cite[Proposition 3.1]{McC06}. 
	}
\end{proof}

\vspace*{0.8 em}
If $\rho \in {R}_0\left(\mathbb{R}^3\right)$ rotates with velocity $v(x):=\frac{J}{I(\rho)} \hat{e}_z \times x$, then it is easy to compute that its kinetic energy $T(\rho, v)$ is given by
\begin{equation} \label{T_J}
	T_J(\rho):=\frac{J^2}{2 I(\rho)}
\end{equation}

Thanks to Remark \ref{positive MoI}, we know $T_J(\rho)=\frac{J^2}{2 I(\rho)}$ is well-defined and finite if the mass of $\rho$ is not zero.

In fact, McCann obtained the following corollary, which, to some extent, strengthens Proposition \ref{uniform rotation}. 

\begin{corollary}[Velocity-free Reformulation {\cite[Secion 3]{McC06}}] \label{velocity-free}
	Let $J \geq 0$ and $\rho \in {R}\left(\mathbb{R}^3\right)$, and define $\omega:=\frac{J}{I(\rho)}$. Let $V(\mathbb{R}^3)$ be a topological vector space, and let the topology on ${R}_0\left(\mathbb{R}^3\right)$ be such that the map taking $\rho$ to $I(\rho)$ is continuous. Then $(\rho, v)$ minimizes $E(\rho, v)$ locally on ${R}_0\left(\mathbb{R}^3\right) \times {V}\left(\mathbb{R}^3\right)$ subject to the constraint $J_z(\rho, v)=J$ if and only if $\rho$ minimizes $E_J(\rho)$ locally on ${R}_0\left(\mathbb{R}^3\right)$ and $v(x)=\omega \hat{e}_z \times x$. Here $E_J(\rho)=U(\rho)-\frac{G(\rho, \rho)}{2}+T_J(\rho)$.
\end{corollary}

\begin{remark}
	By Lemma \ref{properties of Wasser.} (v), we know if ${R}_0\left(\mathbb{R}^3\right)$ is endowed with the topology induced by the Wasserstein $L^\infty$ distance, then the topology satisfies the conditions required in the above corollary, namely, that the mapping $\rho \mapsto I(\rho)$ is continuous.
\end{remark}


\begin{remark}
	In the arguments above, we only consider minimization on ${R}_0\left(\mathbb{R}^3\right)$ or ${R}_0\left(\mathbb{R}^3\right) \times {V}\left(\mathbb{R}^3\right)$. By translation invariance of energy we can see the results also hold over ${R}_a\left(\mathbb{R}^3\right)$ or ${R}_a\left(\mathbb{R}^3\right) \times {V}\left(\mathbb{R}^3\right)$ for any $a\in \mathbb{R}^3$, and furthermore, it can actually be extended to minimization over ${R}\left(\mathbb{R}^3\right)$ or ${R}\left(\mathbb{R}^3\right) \times {V}\left(\mathbb{R}^3\right)$. 
	
\end{remark}

In summary, the local minimization in ${R}_0\left(\mathbb{R}^3\right) \times {V}\left(\mathbb{R}^3\right)$ of $E(\rho, v)$ is equivalent to the local minimization over ${R}_0\left(\mathbb{R}^3\right)$ of $E_J(\rho)$. Note the internal energy and gravitational potential energy are rotation‑invariant. Under the constraint $J_z(\rho,v)=J$, when we know $\widetilde{\rho}$ is a local minimizer over ${R}_0\left(\mathbb{R}^3\right)$ of $E_J(\rho)$, we expect—on physical grounds—that the celestial bodies should be rotating rather than stationary. Therefore, we can further define
$$
({\rho}(t, x), v(t, x)):=\left(\widetilde{\rho}\left(R_{-w t} x\right), \omega\left(-x_2, x_1, 0\right)^T\right)
$$
where $\omega=\frac{J}{I\left(\rho\right)}$, then $({\rho}(t, x), v(t, x))$ gives a uniform rotating binary star system (or star-planet system) such that $E(\rho,v)$ is the local minimum.

Now we have understood one motivation of defining $E_J$. Actually, $E_J$ can be derived through another approach, inspired by Jang and Seok \cite{JS22}.

To find rotating $N$-body solutions to the Euler-Poisson system, we may assume $\rho(t, y)$, $v(t, y)$, and $V(t, y)$ have the following forms:
$$
\rho(t, y)=\widetilde{\rho}\left(R_{-\omega t} y\right), v(t, y)=\omega\left(-y_2, y_1, 0\right)^T, V(t, y)=V_{\widetilde{\rho}}\left(R_{-\omega t} y\right)
$$
where $\omega>0$ is angular velocity, $\widetilde{\rho}$ is a nonnegative density function with compact support. Then we can see the Poisson equation is satisfied automatically. For the conservation of mass, we compute:
\begin{equation*}
	\begin{aligned}
		\partial_t \rho &=\partial_{y_1} \widetilde{\rho} \cdot\left(-y_1 \sin \omega t \omega+y_2 \cos \omega t \omega\right)+\partial_{y_2} \widetilde{\rho} \cdot\left(-y_1 \cos \omega t \omega-y_2 \sin \omega t \omega\right) \\
		\nabla(\rho v) &=-\omega y_2\left(\partial_{y_1} \widetilde{\rho} \cos \omega t-\partial_{y_2} \widetilde{\rho} \sin \omega t\right)+\omega y_1\left(\partial_{y_2} \widetilde{\rho} \cos \omega t+\partial_{y_1} \widetilde{\rho} \sin \omega t\right)
	\end{aligned}
\end{equation*}

Thus $\partial_t \rho+\nabla(\rho v)=0$, the conservation of mass holds true.

We recall $P(\rho)$ denotes the pressure where $\rho$ is a density function, and $P_{12}(y)$ denotes the projection operator, where $y \in \mathbb{R}^3$. For the Euler (momentum) equation, we know it becomes
\begin{equation}\label{rhotilde to rho}
	-\omega^2 \widetilde{\rho}\left(R_{-\omega t} y\right) P_{12}(y)+R_{-\omega t}^T\left(\nabla_{\widetilde{x}} P\left(\widetilde{\rho}\left(R_{-\omega t} y\right)\right)\right)-\widetilde{\rho}\left(R_{-\omega t} y\right) R_{-\omega t}^T\left(\nabla_y V_p\left(R_{-\omega t} y\right)\right)=0
\end{equation}
where $\nabla_y$ means the gradient operator w.r.t. $y$, and $\nabla_{\widetilde{x}}$ means the gradient operator w.r.t. $\widetilde{x}$, where $\widetilde{x}=R_{-\omega t} y$ above.

By the change of variable $x=R_{-\omega t} y$, we have $P_{12}(x)=R_{-\omega t} P_{12}(y), P_{12}(y)=R_{-\omega t}^{-1} P_{12}(x)=R_{-\omega t}^T P_{12}(x)$, then we get the equation introduced in the Introduction:

\begin{equation}
	-\omega^2 \widetilde{\rho}(x) P_{12}(x)+\nabla P(\widetilde{\rho}(x))-\widetilde{\rho}(x)\left(\nabla V_{\widetilde{p}}(x)\right)=0 \tag{EP'}
\end{equation}
where $\nabla$ means the gradient operator w.r.t. $x$ by default. We call (\ref{EP'}) the reduced Euler-Poisson equation. One can also reverse this process to demonstrate that (\ref{EP'}) can imply (\ref{EP}) as well.

If we further assume $P^{\prime}$ and $\nabla \widetilde{\rho}$ exist, then (\ref{EP'}) becomes

\begin{equation}\label{3"}
	-\omega^2 \widetilde{\rho}(x) P(x)+P^{\prime}(\widetilde{\rho}(x)) \nabla(\widetilde{\rho}(x))-\widetilde{\rho}(x)\left(\nabla V_{\widetilde{p}}(x)\right)=0
\end{equation}

Since we consider $N$-body solutions, we denote
$$
\widetilde{\rho}=\sum_{i=1}^N \widetilde{\rho}_i, \quad \widetilde{\rho}_i \geq 0
$$
such that $\text{spt}\left(\widetilde{\rho}_i\right)$ is connected for every $i \in 1, \ldots, N$ and the supports are mutually disjoint. Notice when $s>0$, $A^{\prime \prime}(s)$ exists and $A^{\prime \prime}(s) s=P^{\prime}(s)$ where $A$ is given in (\ref{A}). Then in the region $\{\widetilde{\rho}>0\}$, dividing (\ref{3"}) by $\widetilde{\rho}$, one has
$$
\nabla\left(-\frac{1}{2} \omega^2|P_{12}(x)|^2+A^{\prime}\left(\widetilde{\rho}_i\right)-V_{\widetilde{p}}\right)=0, \forall x \in\left\{\widetilde{\rho}_i>0\right\}
$$

Therefore, we have
\begin{equation}\label{EL0'}
	-\frac{1}{2} \omega^2|P_{12}(x)|^2+A^{\prime}\left(\widetilde{\rho}_i\right)-V_{\widetilde{p}}=C_i \text { in }\left\{\widetilde{\rho}_i>0\right\}\tag{EL0'}
\end{equation}
for some constants $C_i, i=1, \ldots, N$.\\

If the center of mass of $\widetilde{\rho}$ is $\left( {0,0,0} \right)^{T}$, later in subsection \ref{subsection5.1-variational derivative} we will see the variational derivative $E_J^{\prime}(\widetilde{\rho})$ is given by $E_J^{\prime}(\widetilde{\rho})(\sigma)=\int_{\mathbb{R}^3} E_J^{\prime}(\widetilde{\rho}) \sigma\,dx$, where $E_J^{\prime}(\widetilde{\rho})$ on the left-hand side is a linear functional, while $E_J^{\prime}(\widetilde{\rho})$ on the right-hand side is a function, which is actually the left-hand side of (\ref{EL0'}). And (\ref{EL0'}) holds true and is essentially the same as the Euler-Lagrange equation (\ref{EL}), which is stated in Theorem \ref{Properties of LEM}, given $\widetilde{\rho}$ is a $W^{\infty}$ local energy minimizer.


\begin{remark}\label{non-zero center}
%
	Sometimes, instead of (\ref{EL0'}), it is more convenient for us to first consider the following kind of Euler-Lagrange equation:
	\begin{equation}\label{EL'}
		-\frac{1}{2} \omega^2|P_{12}\left( {x - \bar{x}\left( \hat{\rho} \right)} \right)|^2+A^{\prime}\left(\hat{\rho}_i\right)-V_{\hat{\rho}}=C_i \text { in }\left\{\hat{\rho}_i>0\right\}\tag{EL'}
	\end{equation}
	for some constants $C_i, i=1, \ldots, N$.
	If we make a translation: $\widetilde{\rho}(x) = \hat{\rho}\left( x + \bar{x}\left( \hat{\rho} \right) \right)$ and $\hat{\rho}$ satisfies (\ref{EL'}), then $\widetilde{\rho}$ satisfies (\ref{EL0'}). One can refer to \cite[Formula (48)]{McC06} for more information.
\end{remark}


\subsection{McCann's results for 2-body systems}\label{subsection2.3-McCann's results for 2-body systems}

In this subsection, we revisit a slightly modified version of McCann's results for Wasserstein $L^\infty$ ($W^{\infty}$)  local energy minimizers and existence theorem of binary star solution (further discussion on modifications will be discussed in Section \ref{section3-detailed analysis}). We also make some remarks on them.

Although McCann assumes that pressure $P=P(\rho)$ satisfies \ref{F1}\ref{F2}\ref{F3'}, here we assume $P=P(\rho)$ satisfies \ref{F1}\ref{F2}\ref{F3}. Actually the arguments are mainly the same as we can see in Section \ref{section3-detailed analysis}.

\begin{theorem}[Properties of $W^{\infty}$-Local Energy Minimizers {\cite[Theorem 2.1]{McC06}}] \label{Properties of LEM}
	Let $J>0$. If $(\rho, v)$ minimizes $E(\rho, v)$ locally on ${R}_0\left(\mathbb{R}^3\right) \times {V}\left(\mathbb{R}^3\right)$ subject to the constraint $J_z(\rho, v)=J$ then:
	\begin{enumerate}[(i)]
		\item the z-axis is a principal axis of inertia for $\rho$, with a moment of inertia $I(\rho)$ which is maximal and non-degenerate;
		\item the rotation is uniform: $v(x):=\frac{J \hat{e}_z \times x}{I(\rho)}$, $E(\rho,v)=E_J(\rho)$ and $\rho$ locally minimizes $E_J(\rho)$ on ${R}\left(\mathbb{R}^3\right)$;
		\item $\rho$ is continuous on $\mathbb{R}^3$;
		\item on each connected component $\Omega_i$ of $\{\rho>0\}, \rho$ satisfies
		\begin{equation} \label{EL}
			A^{\prime}(\rho(x))=\left[\frac{J^2}{2 I^2(\rho)} r^2(x-\bar{x}(\rho))+V_\rho(x)+\lambda_i\right]_{+} \tag{EL}
		\end{equation}
		for some Lagrange multiplier $\lambda_i<0$ depending on the component. Here $[\cdot]_{+}$ is the nonnegative (positive) part function defined as $[\lambda]_{+}:=\max \{\lambda, 0\}$;
		\item the equations (\ref{EL}) continue to hold on a $\delta$-neighbourhood of the $\Omega_i$. Here the $\delta$-neighbourhood of $\Omega_i$ is defined as $\bigcup_{y \in \Omega_i}\left\{x \in \mathbb{R}^3\mid | x-y |<\delta\right\}$;
		\item if $P(\rho)$ satisfies assumption \ref{F4}, then $\rho\in C^1(\{\rho>0\})$;
		\item if $P(\rho)$ satisfies assumption \ref{F4'}, then $\rho$  satisfies the reduced Euler-Poisson equations (\ref{EP'});
		\item this solution is stable with respect to $L^{\infty}$-small perturbations of the Lagrangian fluid variables.
	\end{enumerate}
\end{theorem}

\begin{remark}\label{removing positive part}
	On each connected component $\Omega_i$ of $\{\rho>0\}$, let $W_\rho(x)=\frac{J^2}{2 I^2(\rho)} r^2\left(x\right)+V_\rho\left(x\right)+\lambda_i$, then since $A^\prime$ is strictly increasing and $A^\prime(0)=0$, one can see $A^\prime(\rho(x))=W_\rho(x)$ on $\Omega_i$ due to (\ref{EL}). (See also the proof of Proposition \ref{local constant multiplier}.)
\end{remark}

\begin{remark}
	Due to the discussion in subsection \ref{subsection2.2-uniform rotation minimizers kinetic energy}, if $\widetilde{\rho}$ is a local minimizer over ${R}\left(\mathbb{R}^3\right)$ of $E_J(\rho)$ (or by Corollary \ref{center/support free}, even just over ${R}_0(\mathbb{R}^3)$), then we may define:
	$$
	({\rho}(t, x), v(t, x)):=\left(\widetilde{\rho}\left(R_{-w t} x\right), \omega\left(-x_2, x_1, 0\right)^T\right)
	$$
	where $\omega=\frac{J}{I\left(\rho\right)}$, then $({\rho}(t, x), v(t, x))$ gives a uniform rotating binary star system (or star-planet system) such that $E(\rho,v)$ and $E_J(\rho)$ are local minima. Moreover, $({\rho}(t, x), v(t, x))$ solves (\ref{EP}) with $V(t, x)=V_{\widetilde{\rho}}\left(R_{-\omega t} x\right)$.
\end{remark}

\begin{remark}
	The ``Lagrange multiplier'' $\lambda_i$ in Parts (iv) and (v) here is also referred to as the ``chemical potential'' in McCann's paper \cite{McC06}. In statistical mechanics, the chemical potential corresponds to $\frac{\partial E}{\partial N}$, where $E$ is the (free) energy and $N$ is the particle number. See also  \cite{Bae22}.
	
\end{remark}

\begin{remark}
%
	We retain the term $\bar{x}(\rho)$ here to be consistent with the form of the Euler-Lagrange equations in McCann's paper \cite[Section 6]{McC06}. One may also extend Theorem \ref{Properties of LEM} to cases with nonzero center of mass, for instance, via the translation mentioned in Remark \ref{non-zero center}.
\end{remark}

\begin{remark}\label{stable under L^infty-small perturbation }
	Thanks to $E(\rho,v)=E_J(\rho)$, we can further understand and elaborate on the meaning of Part (viii):
	
	Let $\rho$ represents the original density profile, and $Y_{t}(y) \in \mathbb{R}^{3}$ represents the position of the fluid particles at time $t$ which originated at $
	Y_{0}(y) = y$, then the density profile $\rho_t$ after time $t$ satisfies ${Y_t}_{\#} \rho=\rho_t$, where ${Y_{t}}_{\#}$ is the push-forward operator given in Definition \ref{push forward}. From the definition of Wasserstein metric, we have
	\begin{equation}\label{fluid velocity}
		W^{\infty}\left(\rho_s, \rho_t\right) \leq\left\|Y_s-Y_t\right\|_{L^{\infty}, \rho}
	\end{equation}	
	Here $\|h\|_{L^{\infty}(S),v}$ denotes the supremum of $|h|$ over $S$, discarding sets of $v$-measure zero.
	 
	If we further assume the fluid particles move with bounded velocities, then $Y_t(y)$ will be a Lipschitz function of $t$ uniformly in $y$, and it is evident that (\ref{fluid velocity}) will be controlled by a multiple of $|s-t|$. Thus $\rho_t $ is a continuous function with respect to time $t$.
	
	 On the other hand, if $\left\| {Y_{t} - Y_{0}} \right\|_{L^{\infty},\rho}$ is sufficiently small, then $\forall\tilde{v} \in V\left( \mathbb{R}^{3} \right)$ for which $T({Y_{t}}_{\#}\rho, \tilde{v})<\infty$ and satisfying the constraint ${J}_z({Y_{t}}_{\#}\rho, \tilde{v})=J$, we have
	$$
	E\left( {\rho,v} \right) = E_{J}(\rho) \leq E_{J}\left( {{Y_{t}}_{\#}\rho} \right) \leq E\left( {{Y_{t}}_{\#}\rho,~\tilde{v}} \right)
	$$
	The last inequality comes from Proposition \ref{uniform rotation}.
	
	Those arguments suggest an $L^{\infty}$-small perturbation of the Lagrangian fluid variables produces only a $W^{\infty}$-small perturbation of the density and thereby a local energy minimum $\rho \in {R}_0\left(\mathbb{R}^3\right)$ must be physically stable.	We will introduce more related content in Section \ref{section4-Wasserstein metric}.
\end{remark}
\begin{remark}\label{vector angular momentum}
	Although Theorem \ref{Properties of LEM} applies to energy minimizers subject only to a constraint on the z-component $J_z(\rho, v):=\hat{e}_z \cdot \boldsymbol{J}(\rho, v)$ of the angular momentum, it can be extended to the case of physical interest. To be precise, if $(\rho, v)$ minimizes $E(\rho, v)$ locally subject to the constraint $J_z(\rho, v)=J$, thanks to Theorem \ref{Properties of LEM} (i)(ii), the angular momentum of $(\rho, v)$ is $\boldsymbol{J}(\rho, v)=J \hat{e}_z$, which means $(\rho, v)$ also minimizes $E(\rho, v)$ locally subject to the more restricted constraints $\boldsymbol{J}(\rho, v)=J \hat{e}_z$ \cite[Corollary 2.2]{McC06}. 
	
	On the other hand, if the topology on ${V}\left(\mathbb{R}^3\right)$ is such that the map taking $\boldsymbol{w} \in \mathbb{R}^3$ to $v(x):=\boldsymbol{w} \times x \in {V}\left(\mathbb{R}^3\right)$ is continuous, then a converse to the statement above is also true: any local minimizer subject to the vector constraint $\boldsymbol{J}(\rho, v)=J \hat{e}_z$ also minimizes locally among the larger class of competitors with prescribed $J_z(\rho, v)=J$ \cite[Remark 2.3]{McC06}. This condition that $w \in \mathbb{R}^3  \mapsto v(x):=w \times x \in {V}\left(\mathbb{R}^3\right)$ is continuous is used in the following arguments. Due to \cite[Remark 3.3]{McC06}, we can have a similar result to Proposition \ref{uniform rotation}: if $\boldsymbol{J}$ is prescribed, the velocity which minimizes the energy is given by $\boldsymbol{v}(\rho,\boldsymbol{J})(x)=\boldsymbol{w}(\rho,\boldsymbol{J})\times x$, where $\boldsymbol{w}(\rho,\boldsymbol{J}) \in \mathbb{R}$ is the unique angular velocity compatible with the given density $\rho$ and angular momentum $\boldsymbol{J}$. Notice the axis $\boldsymbol{w}$ of rotation may not coincide with the $z$-axis. One can show when $\rho$ is sufficiently close to $\rho'$ in the sense of Wasserstein $L^\infty$ distance, $w(\rho,\boldsymbol{J})$ is sufficiently close to $w'(\rho',\boldsymbol{J'})$. This condition that $w \in \mathbb{R}^3  \mapsto v(x):=w \times x \in {V}\left(\mathbb{R}^3\right)$ is continuous ensures that $\boldsymbol{v}(\rho,\boldsymbol{J})$ is sufficiently close to $\boldsymbol{v'}(\rho',\boldsymbol{J'})$ in $V(\mathbb{R}^3)$. In another word, local perturbations in ${R}_0\left(\mathbb{R}^3\right)$ can create local perturbations in ${R}_0\left(\mathbb{R}^3\right) \times {V}\left(\mathbb{R}^3\right)$, which can help us to prove that any local minimizer subject to the vector constraint $\boldsymbol{J}(\rho, v)=J \hat{e}_z$ also minimizes locally among the larger class of competitors with prescribed $J_z(\rho, v)=J$, using similar arguments to \cite[Remark 2.3 and Section 3]{McC06}.
\end{remark}

\begin{remark}[Equivalence between (\ref{EP'}) and (\ref{EL})]\label{EL to EP}
	Now we can obtain a solution to (\ref{EP'}) from the Euler-Lagrange equation (\ref{EL}) thanks to Theorem \ref{Properties of LEM} (iv). Conversely, for $\rho\in R_0 (\mathbb{R}^3)$ with N connected components to be a solution of (\ref{EP'}), similar arguments as in subsection \ref{subsection2.2-uniform rotation minimizers kinetic energy} shows that (\ref{EL}) in Theorem \ref{Properties of LEM} (iv) is necessary.
\end{remark}

By the arguments of Auchmuty and Beals \cite{AB71} or Li \cite{Li91} or McCann \cite{McC06}, we know a constrained energy minimizer on $W_{m,J}$ exists and its support lies in the interior of $\Omega_m \cup \Omega_{1-m}$. Then such constrained energy minimizer is indeed a local minimizer due to Lemma \ref{properties of Wasser.} (ii), and hence a solution to (\ref{EP'}) after translation.
\begin{theorem}[Existence of Binary Stars {\cite[Theorem 6.1, Corollary 6.2]{McC06}}] \label{McCann's existence}
	Given $m \in(0,1)$, choose the angular momentum $J$ to be sufficiently large depending on $m$. Then any constrained minimizer $\widetilde{\rho}=\rho^{-}+\rho^{+}$of $E_J(\rho)$ on $W_{m,J}$ will, after a rotation about the z-axis and a translation, have support contained in the interior of $\Omega:=\Omega_{-} \cup \Omega_{+}$, that is, dist$\left(\operatorname{spt} \widetilde{\rho}, \mathbb{R}^3 \backslash \Omega\right)>0$. It will also be symmetric about the plane $z=0$ and a decreasing function of $|z|$. 
	
	What's more, after another translation, the center of mass of $\widetilde{\rho}$ can be 0 and is a local minimizer of $E_J(\rho)$. Let $v(x):=\omega \hat{e}_z \times x =\omega\left(-x_2, x_1, 0\right)^T$, where $\omega = \frac{J}{I(\widetilde{\rho})}$, then the pair $(\widetilde{\rho}, v)$ minimizes $E(\rho, v)$ locally on ${R}\left(\mathbb{R}^3\right) \times {V}\left(\mathbb{R}^3\right)$ (thus on ${R}_0\left(\mathbb{R}^3\right) \times {V}\left(\mathbb{R}^3\right))$ subject to the constraint $J_z(\rho, v)=J$ or $\boldsymbol{J}(\rho, v)=J \hat{e}_z$. $\widetilde{\rho}$ satisfies reduced Euler-Poisson equations (\ref{EP'}). Moreover, the uniformly rotating fluid $({\rho}(t, x), v(t, x))$ solves (\ref{EP}) with $V(t, x)=V_{\widetilde{\rho}}\left(R_{-\omega t} x\right)$, here $({\rho}(t, x), v(t, x))=\left(\widetilde{\rho}\left(R_{-\omega t} x\right), \omega\left(-x_2, x_1, 0\right)^T\right)$.
\end{theorem}

\section{Detailed Analysis of Theorem \ref{Properties of LEM}} \label{section3-detailed analysis}
We point out that in McCann's proof of Theorem \ref{Properties of LEM}, the discussion on how to transition from (\ref{EL}) to (\ref{EP'}) is not particularly detailed. One can see intuitively from subsection \ref{subsection2.2-uniform rotation minimizers kinetic energy} that a refined discussion on differentiability is required. Therefore, in this section, we will go through the proof of Theorem \ref{Properties of LEM}, give new ideas for establishing differentiability (Part (vi) and Part (vii) of Theorem \ref{Properties of LEM}), and explain the differences between our proof and McCann's.

First we show the differentiability of $P \circ \phi=P \circ \left(A^{\prime}\right)^{-1}$, with the additional assumption \ref{F4'}.

\begin{lemma}[Differentiability of $P \circ \phi$] \label{differentiability of composition}
	Assume $P$ satisfies \ref{F1}\ref{F2}\ref{F3}\ref{F4'}, then $P \circ \phi$ is differentiable in $(0, \infty)$ and the derivative is $\phi(s)$. Moreover, $P \circ \phi$ is right differentiable at 0 with right derivatives $\phi(0)=0$. Here $\phi=\left(A^{\prime}\right)^{-1}$ exists by Remark \ref{existence of inverse}.
\end{lemma}

\begin{proof}
	{\rm Given $s>0, h \in \mathbb{R}$ with $0<|h| \leq s$, let $t=\phi(s)$ and $t+l=\phi(s+h)$, then $h \rightarrow 0$ is equivalent to $l \rightarrow 0$ due to the continuity and monotonicity of $\phi$, and we also have $A^{\prime}(t)=s, A^{\prime}(t+l)=s+h$ since $\phi=\left(A^{\prime}\right)^{-1}$. Recall by \eqref{A', A and P relation} we have $A^{\prime}(s) s-A(s)= P(s)$. Therefore, we have
		$$
		\begin{aligned}
			\frac{P \circ \phi(s+h)-P \circ \phi(s)}{h}&=\frac{P(t+l)-P(t)}{A^{\prime}(t+l)-A^{\prime}(t)} \\
			&=\frac{\left((t+l) A^{\prime}(t+l)-A(t+l)\right)-\left(t A^{\prime}(t)-A(t)\right)}{A^{\prime}(t+l)-A^{\prime}(t)} \\
			&=\frac{(t+l)\left(A^{\prime}(t+l)-A^{\prime}(t)\right)}{A^{\prime}(t+l)-A^{\prime}(t)}+\frac{l A^{\prime}(t)+A(t)-A(t+l)}{A^{\prime}(t+l)-A^{\prime}(t)} \\
			&=t+l+\frac{l A^{\prime}(t)+A(t)-A(t+l)}{A^{\prime}(t+l)-A^{\prime}(t)}
		\end{aligned}
		$$
		Let $B(l)=l A^{\prime}(t)-A(t+l)$, then $B$ is differentiable and $B^{\prime}(l)=A^{\prime}(t)-A^{\prime}(t+l) \neq$ 0 if $l \neq 0$, since $A^{\prime}$ is strictly increasing. Notice $l A^{\prime}(t)+A(t)-A(t+l)=B(l)-B(0)$, by Cauchy's mean value theorem \cite[Theorem 5.9]{Rud76}, $\exists \theta \in(0,1)$, such that 
		$$(B(l)-B(0)) A^{\prime \prime}(t+\theta l)=\left(A^{\prime}(t+l)-A^{\prime}(t)\right) B^{\prime}(\theta l)$$  
		Since $l \neq 0$, $\left(A^{\prime}(t+l)-A^{\prime}(t)\right)\neq 0$, $B^{\prime}(\theta l)\neq 0$, then both $(B(l)-B(0))$ and $A^{\prime \prime}(t+\theta l)$ are also not 0. Then
		$$
		\begin{aligned}
			\frac{l A^{\prime}(t)+A(t)-A(t+l)}{A^{\prime}(t+l)-A^{\prime}(t)}&=\frac{(B(l)-B(0)) B^{\prime}(\theta l)}{(B(l)-B(0)) A^{\prime \prime}(t+\theta l)}\\
			&=\frac{A^{\prime}(t)-A^{\prime}(t+\theta l)}{A^{\prime \prime}(t+\theta l)}
		\end{aligned}
		$$
		By Lagrange's mean value theorem \cite[Theorem 5.8]{Rud76}, $\exists \widetilde{\theta} \in(0,1)$, such that $A^{\prime}(t)-A^{\prime}(t+\theta l)=(-\theta l) A^{\prime \prime}(t+\widetilde{\theta} \theta l)$. Notice $A^{\prime \prime}(s)=\frac{P^{\prime}(s)}{s}$ when $s>0$. Thus
		$$
		\begin{aligned}
			\frac{A^{\prime}(t)-A^{\prime}(t+\theta l)}{A^{\prime \prime}(t+\theta l)}&=\frac{(-\theta l) A^{\prime \prime}(t+\widetilde{\theta} \theta l)}{A^{\prime \prime}(t+\theta l)}\\
			&=-\frac{\theta l(t+\theta l)}{t+\widetilde{\theta} \theta l} \cdot \frac{P^{\prime}(t+\widetilde{\theta} \theta l)}{P^{\prime}(t+\theta l)}
		\end{aligned}
		$$
		By Taylor's theorem \cite[Theorem 5.15]{Rud76}, together with \ref{F4'}, we know $\exists \theta_1 \in(0,1), \theta_2 \in(0,1)$,
		$$P^{\prime}(t+\theta l)=P^{\prime}(t)+\theta l P^{\prime \prime}(t)+\frac{(\theta l)^2}{2} P^{(3)}(t)+\cdots+\frac{(\theta l)^n}{n!} P^{(n+1)}\left(t+\theta_1 \theta l\right)=\frac{(\theta l)^n}{n!} P^{(n+1)}\left(t+\theta_1 \theta l\right)\neq 0$$
		Similarly, we have $$P^{\prime}(t+\widetilde{\theta} \theta l)=\frac{(\widetilde{\theta} \theta l)^n}{n!} P^{(n+1)}\left(t+\theta_2 \widetilde{\theta} \theta l\right)\neq 0$$ 
		Thus $\left|\frac{P^{\prime}(t+\widetilde{\theta} \theta l)}{P^{\prime}(t+\theta l)}\right| \leq 2$ when $|l|$ small enough. As $l \rightarrow 0$, we have $\frac{A^{\prime}(t)-A^{\prime}(t+\theta l)}{A^{\prime \prime}(t+\theta l)} \rightarrow 0$, and then $\frac{P \circ \phi(s+h)-P \circ \phi(s)}{h} \rightarrow t=\phi(s)$ as $h \rightarrow 0$. Thus, the derivative of $P \circ \phi$ is $\phi(s)$. By the same arguments we know $P \circ \phi$ is right differentiable at 0 with right derivatives $\phi(0)=$ 0.}
\end{proof}


\bigskip
We skip the proofs of \textbf{Part (i)} and \textbf{Part (viii)} of Theorem \ref{Properties of LEM}; one can refer to McCann's paper \cite{McC06}. Part (ii) and Part (iii) are closely related to the context, so we revisit McCann's proof here. The other parts of the proof are relatively new, either addressing details not mentioned by McCann (Parts (iv-v)) or providing new approaches to proving differentiability and explaining the necessity of modifying the original statements of the conclusions (Parts (vi-vii)).

\begin{proof}[Proof of Part (ii-vi) of Theorem \ref{Properties of LEM}]
	{
		\rm
		\textbf{Part (ii)} essentially follows from Corollary \ref{velocity-free} and Corollary \ref{center/support free}. We note that $E(\rho,v)=E_J(\rho)$ due to the definition of $v(x):=\frac{J \hat{e}_z \times x}{I(\rho)}$, $E$ (\ref{energy}), $E_J$ (\ref{energy of UR}), $T$ (\ref{T}), and $T_J$ (\ref{T_J}). 
		
		Due to Part (ii), it suffices to consider $\rho$ that locally minimizes $E_J(\rho)$ on $R(\mathbb{R}^3)$. We will give proofs of ``weak'' versions of Parts (iv-v) (Proposition \ref{local constant multiplier}, Proposition \ref{constant multiplier} and Remark \ref{connected components and delta neighborhood}) later in subsection \ref{subsection5.2-locally constant chemical potential}, that is, (\ref{EL}) holds for almost all points in the region we consider with negative $\lambda_i$. Based on this ``weak'' version of (v), we obtain the global continuity in Part (iii) as the following: first notice that almost everywhere in the region $\{\rho>0\}$, (\ref{EL}) holds true and $A^{\prime}(\rho)>0$ implies $\frac{J^2}{2 I^2(\rho)} r^2(x)+V_\rho(x)+\lambda_i>0$. Therefore, on $\{\rho > 0\}$, since $\rho \in {R}_0\left(\mathbb{R}^3\right)$ has compact support, we have $$A^{\prime}(\rho(x))=\frac{J^2}{2 I^2(\rho)} r^2(x)+V_\rho(x)+\lambda_i \leq V_\rho(x)+C$$ This result, together with Hardy-Littlewood-Sobolev Inequality \cite[Theorem 1.7]{BCD11} or Proposition \ref{bound of potential}, can help to show the bound of $\|\rho\|_{L^p}$ and $\left\|V_\rho\right\|_{L^{\widetilde{p}}}$ with $p$ and $\widetilde{p}$ increasing alternately and reaching infinity in finite steps (bootstrap method, similar as \cite[Lemma 3 and Theorem A]{AB71}). Then $V_\rho$ is continuously differentiable thanks to Proposition \ref{diff. of poten.}. Then on a $\delta$-neighbourhood of $\{\rho>0\}$, denoted by $\Omega_\delta$, which contains $\partial\{\rho>0\}$, we have $\rho=\phi \circ\left[W_\rho\right]_{+}$ a.e., where $W_\rho(x)=\frac{J^2}{2 I^2(\rho)} r^2(x)+V_\rho(x)+\lambda_i$ and $\phi=\left(A^{\prime}\right)^{-1}$ exists and continuous by Remark \ref{existence of inverse}. It means $\rho$ has a continuous representative there. Notice the points in $\overline{\mathbb{R}^3 \backslash \Omega_\delta}$ are the interior points of $\{\rho=0\}$, hence $\rho$ is also continuous in $\overline{\mathbb{R}^3 \backslash \Omega_\delta}$. Thus, we conclude \textbf{Part (iii)}.
		
		Due to the global continuity of $\rho$ and $V_\rho$ (that is, they are continuous on $\mathbb{R}^3$), we can come back to strengthen the ``weak'' versions of Parts (iv-v), replace ``almost all points'' by ``all points'' and get \textbf{Parts (iv)} and \textbf{Part (v)}.
		
		For \textbf{Part (vi)}, notice in the region $\{\rho>0\}$, we have $\rho=\phi \circ W_\rho$ by Remark \ref{removing positive part}. Thanks to Part (iii) and $\rho \in {R}_0\left(\mathbb{R}^3\right)$, we know $\rho \in C_c^0\left(\mathbb{R}^3\right)$, where $C_c^0\left(\mathbb{R}^3\right)$ denotes the family of continuous functions with compact support. In particular, $\rho \in L^{1}\left( \mathbb{R}^{3} \right) \cap L^{q}\left( \mathbb{R}^{3} \right)$ for any $q>3$. By Proposition \ref{diff. of poten.}, we know $V_\rho$ is continuous differentiable in $\mathbb{R}^3$. Since $P(\rho)$ satisfies assumption \ref{F4}, by Remark \ref{diff. of inverse} we know $\phi \in C^1 ((0,\infty))$. Apply chain rule (one can refer to \cite[Theorem 10.1.15]{Tao16}), we obtain $\rho\in C^1 (\{\rho>0\})$. 
	}
	\end{proof}

	\begin{remark}
	Theorem \ref{Properties of LEM} (vi) is slightly different from the one in McCann's theorem \cite[Theorem 2.1 (vi)]{McC06}, which states ``where $\rho$ is positive, it has as many derivatives as the inverse of $A'(\rho)$''. In McCann's proof, he said ``Because $V_\rho$ gains a derivative from $\rho$, smoothness of $\rho$ where positive follows from a bootstrap in (15)'', where ``(15)'' is (\ref{EL}) in this paper. It is true that once $\rho \in W^{k,q}\left( \mathbb{R}^{3} \right)$ with some $q>3$ then we can obtain $V_{\rho} \in W^{k + 1,\infty}\left( \mathbb{R}^{3} \right) \cap C^{k + 1}\left( \mathbb{R}^{3} \right)$ by the similar arguments in Proposition \ref{diff. of poten.}. But so far we only know $\rho \in C^{1}\left( \left\{ {\rho > 0} \right\} \right) \cap C_{c}^{0}\left( \mathbb{R}^{3} \right)$, one thing that is not clear is whether $\rho \in W^{1,q}\left( \{\rho>0\} \right)$, and hence whether $\rho \in W^{1,q}\left( \mathbb{R}^{3} \right)$. The difficulty lies in the boundary behaviour: although $\rho$ is globally continuous and vanishes on $\partial\{\rho>0\}$, we do not know the regularity of $\partial\{\rho>0\}$ and we cannot rule out the possibility that $\nabla\rho$ blows up near the boundary. It means we might not be able to apply the result of trace operator (e.g., \cite[Exercise 15.26]{Leo09}) to glue $\rho \cdot \mathbf{1}_{\{\rho > 0\}}$ and $\rho \cdot \mathbf{1}_{\{\rho = 0\}}$ (here $\mathbf{1}_E$ denotes the indicator function of $E$).
	
	If we do not know whether $\rho \in W^{1,q}\left( \mathbb{R}^{3} \right)$, we might also not know whether $V_{\rho} \in W^{2,\infty}\left( \mathbb{R}^{3} \right) \cap C^{2}\left( \mathbb{R}^{3} \right)$. 
	Even if we assume, for example, $\phi \in C^{2}\left( (0,\infty) \right)$, due to Remark \ref{removing positive part}, we know on $\Omega_i$,
	$$\rho(x) = \phi \circ W_{\rho}(x)=\phi \circ\left(\frac{J^2}{2 I^2(\rho)} r^2\left(x\right)+V_\rho(x)+\lambda_i\right)$$ 
	 but it might not be enough to show $\rho\in C^2 (\{\rho>0\})$ if we do not know $V_{\rho} \in W^{2,\infty}\left( \mathbb{R}^{3} \right) \cap C^{2}\left( \mathbb{R}^{3} \right)$. Therefore, it is not quite convincing if $\phi\in C^r ((0,\infty))$ implies $\rho\in C^r (\{\rho>0\})$ if $r\geq 2$.
	 
	 Additionally, although Theorem \ref{Properties of LEM} (vi) suggests $\phi=\left(A^{\prime}\right)^{-1} \in C^1((0, \infty))$ implies $\rho \in C^1(\{\rho>0\})$, whether the converse holds remains open.
	\end{remark}

Note that when we try to show $\rho$ satisfies (\ref{EP'}), we do not necessarily need to deal with the gradient of $\rho$; we only need to consider the gradient of $P(\rho)$. As a result, \ref{F4} can be relaxed to \ref{F4'}. Under the conditions of \ref{F4'}, we can still prove the gradient of $P(\rho)$ exists and then $\rho$ satisfies (\ref{EP'}), as shown in the proof of Part (vii) of Theorem \ref{Properties of LEM} below.
\begin{proof}[Proof of Part (vii) of Theorem \ref{Properties of LEM}]
{\rm $\rho \in C_c^0\left(\mathbb{R}^3\right)$ implies $\nabla V_\rho$ exists and is continuous in $\mathbb{R}^3$ thanks to Proposition \ref{diff. of poten.}, so is $\nabla W_\rho$, where $W_\rho(x)=\frac{J^2}{2 I^2(\rho)} r^2(x)+V_\rho(x)+\lambda_i$ as in the proof of Part (iii-vi). In the region $\{\rho>0\}$, (\ref{EL}) holds true and $A^{\prime}(\rho)>0$ implies $W_\rho>0$. $\nabla W_\rho$ exists means $\nabla A^{\prime}(\rho)$ exists. Thanks to Lemma \ref{differentiability of composition}, we know 
	$$
	\begin{aligned}
		\nabla(P \circ \rho)(x)&=\nabla\left(P \circ \phi \circ W_\rho\right)(x)\\
		&=(P \circ \phi)^{\prime}\left(W_\rho(x)\right) \nabla W_\rho(x)\\
		&=\phi\left(W_\rho(x)\right) \nabla W_\rho(x)\\
		&=\phi\left(W_\rho(x)\right) \nabla A^{\prime}(\rho(x))\\
		&=\rho(x) \nabla A^{\prime}(\rho(x))
	\end{aligned}
	$$ 
	Then we take the gradient of (\ref{EL}) in $\{\rho>0\}$. By the similar computations in subsection \ref{subsection2.2-uniform rotation minimizers kinetic energy}, we have (\ref{EP'}) holds true in $\{\rho>0\}$.
	
	In the interior of $\{\rho=0\}$, easy to check (\ref{EP'}) still holds true since $\nabla P(\rho(x))$ exists and $\nabla P(\rho(x))=0$, and $\nabla V_\rho(x)$ exists and finite. 
	
	The points that remain to be checked are those on the boundary $\partial\{\rho=0\}=\partial\{\rho>0\}$. Since $\rho$ is continuous, we know $\rho=0$ on $\partial\{\rho=0\}$ and then similarly as above we have $-\omega^2 \rho(x) P_{12}(x)=-\rho(x) \nabla V_\rho(x)=0$. But to show (\ref{EP'}) holds, we still need to check if $\nabla P(\rho(x))$ exists and is 0 on $\partial\{\rho=0\}$.
	
	Given $x_0 \in \partial\{\rho=0\}$ and $\left\{x_n\right\} \subset \mathbb{R}^3$, where $x_n=x_0+h_n \widehat{e_1}, \widehat{e_1}=(1,0,0)^T, h_n \neq 0$, and $\lim\limits_{n \rightarrow \infty} h_n=0$. If $\rho\left(x_n\right)=0$ for all $n$, then $P\left(\rho\left(x_n\right)\right)=P\left(\rho\left(x_0\right)\right)=0$, and $\lim\limits_{n \rightarrow \infty} \frac{P\left(\rho\left(x_n\right)\right)-P\left(\rho\left(x_0\right)\right)}{h_n}=0$. If $\rho\left(x_n\right)>0$ for all $n$, then $y_n:=W_\rho\left(x_n\right)>0$ and $\rho\left(x_n\right)=\phi\left(y_n\right)$ as above. 
	
	We first claim $x_0 \in \partial\{\rho>0\}$ implies $y_0:=W_\rho\left(x_0\right)=0$ by continuity. In fact, $y_0>0$ would imply $\rho\left(x_0\right)>0$, while $y_0<0$ would imply $x_0$ is in the interior of $\{\rho=0\}$, and both cases cannot happen since $x_0 \in \partial\{\rho>0\}$. In particular, $\rho(x_0)=\phi(y_0)=0$.
	
	Then, by Lemma \ref{differentiability of composition} and Lagrange's mean value theorem \cite[Theorem 5.8]{Rud76}, $\exists \theta_1 \in(0,1), \theta_2 \in(0,1)$, such that
	$$
	\begin{aligned}
		\frac{P\left(\rho\left(x_n\right)\right)-P\left(\rho\left(x_0\right)\right)}{h_n}&=\frac{P\left(\phi\left(y_n\right)\right)-P\left(\phi\left(y_0\right)\right)}{y_n-y_0} \cdot \frac{y_n-y_0}{h_n} \\
		&=\phi\left(y_0+\theta_2\left(y_n-y_0\right)\right) \cdot \partial_{x_1} W_\rho\left(x_0+\theta_1\left(x_n-x_0\right)\right)
	\end{aligned}
	$$
	
	$\lim\limits_{n \rightarrow \infty} x_n=x_0$ implies $\left\{x_n\right\}$ is bounded, so is $\left\{y_n\right\}$. Therefore, $\{\partial_{x_1} W_\rho(x_0+\theta_1(x_n-x_0))\}$ is a bounded sequence. $\lim\limits_{n \rightarrow \infty} \phi\left(y_0+\theta_2\left(y_n-y_0\right)\right)=\phi\left(y_0\right)=0$ implies $\lim\limits_{n \rightarrow \infty} \frac{P\left(\rho\left(x_n\right)\right)-P\left(\rho\left(x_0\right)\right)}{h_n}=0$.
	
	Now we have considered the cases $\rho\left(x_n\right)=0$ and $\rho\left(x_n\right)>0$ and show $\lim\limits_{n \rightarrow \infty} \frac{P\left(\rho\left(x_n\right)\right)-P\left(\rho\left(x_0\right)\right)}{h_n}=0$ in both cases. It is a sufficient condition to show $\partial_{x_1} P(\rho)$ exists and is 0 at $x_0$. Otherwise, we could have a sequence $\left\{\widetilde{h_n}\right\}$ converging to $0$, $\widetilde{x_n}=x_0+\widetilde{h_n} \widehat{e_1},\left\{\widetilde{x_n}\right\}$ (up to subsequence) satisfies either $\rho\left(\widetilde{x_n}\right)>0$ for all $n$ or $\rho\left(\widetilde{x_n}\right)=0$ for all $n$, $\frac{P\left(\rho\left(\widetilde{x_n}\right)\right)-P\left(\rho\left(\widetilde{x_0}\right)\right)}{h_n}$ would not converge to 0. But the result above shows $\lim\limits_{n \rightarrow \infty} \frac{P\left(\rho\left(\widetilde{x_n}\right)\right)-P\left(\rho\left(\widetilde{x_0}\right)\right)}{h_n}=0$, which leads to a contradiction. Similarly, $\partial_{x_2} P(\rho)$ and $\partial_{x_3} P(\rho)$ also exist and are 0 at $x_0$. Therefore, $\nabla P(\rho(x))=0$ on $\partial\{\rho>0\}$ and then (\ref{EP'}) holds true on $\partial\{\rho>0\}$, which gives us \textbf{Part (vii)}.}
	\end{proof} 
	
	\begin{remark}
	Theorem \ref{Properties of LEM} (vii) is also slightly different from the one in McCann's theorem \cite[Theorem 2.1 (vii)]{McC06}, which states ``If $P(\rho)$ is continuously differentiable on $[0,\infty)$, then $\rho$ satisfies (\ref{EP'})''. Although there is only two-line proof of Part (vii) in McCann's paper \cite[Theorem 2.1]{McC06}: ``If $P(\rho)$ is continuously differentiable, then $A^{\prime \prime}(\rho)=\frac{P^{\prime}(\rho)}{\rho}$ and (vii) follows by taking the gradient of (15)'', where ``(15)'' is (\ref{EL}) in this paper, it is not quite trivial since we do not know whether $\nabla P(\rho)$ truly exists in $\mathbb{R}^3$ ((\ref{EL}) holds only in a $\delta$-neighborhood). Moreover, even if $P$ is differentiable, the identity $A''(s)=\frac{P'(s)}{s}$ for a scalar $s$ does not immediately imply $\nabla A'(\rho)=\frac{\nabla P(\rho)}{\rho}$ for a function $\rho$. We need more detailed discussions as above, under the assumption \ref{F4'}.
	\end{remark}
%

\begin{remark}
Going through the proof, we can see Theorem \ref{Properties of LEM} also holds true when the number of components $n$ is not 2. However, it may not be clear so far if such $W^{\infty}$ local minimizers exist or not when $n \geq 3$. Jang and Seok \cite[1. Introduction]{JS22} say ``when $N \geq 3$, uniformly rotating N-body stellar objects do not retain a variational characterization analogous to the binary case and they are not expected to be stable in general''.
\end{remark}

\begin{remark}
	In Remark \ref{supports are in the interior} and Theorem \ref{McCann's existence}, we mention that the support of the constrained minimizer lies in the interior of $\Omega_m \cup \Omega_{1-m}$ implies the constrained minimizer is a local minimizer, hence satisfies Theorem \ref{Properties of LEM}. In fact, if we revisit the proof of Theorem \ref{Properties of LEM}, those results (excluding the equivalence result that $(\rho, v)$ minimizes $E(\rho, v)$ locally if and only if $\rho$ minimizes $E_J(\rho)$ locally) are also true without specifying it is a local minimizer. The crucial thing is deriving (\ref{EL}) in a neighborhood of the support of the constrained minimizer, which is true if the support lies in the interior of $\Omega_m \cup \Omega_{1-m}$. Then we can verify (\ref{EP'}) at the boundary of the support.
\end{remark}

\section{Wasserstein $L^{\infty}$ Metric on ${R}\left(\mathbb{R}^3\right)$}\label{section4-Wasserstein metric}

In this section we revisit some basic concepts and properties of Wasserstein $L^{\infty}$ metric ($W^\infty$ metric), which is important to show the existence of solutions to (\ref{EP}). Additionally, we present a new result stating that, under the topology induced by $W^\infty$ metric, given $\rho$, one can always find a $\sigma$ in the neighborhood of $\rho$ such that $\sigma$ lies in $L^\infty$. This result is helpful in Section \ref{section5-constant chemical potential}, where we establish the existence of variational derivatives for local minimizers.

Let us first discuss the motivation of introducing $W^\infty$ metric: as McCann says in \cite{McC06}, when we discuss what local minima are, we should first specify the topology. The choice will be delicate. In order to have local minima, we hope the topology to be strong enough to preclude tunneling of mass (counterexample can be seen in \cite[Example 3.6 and Remark 3.7]{McC06}). On the other hand, to make the local minima physically meaningful, we hope the topology to be weak enough so that the evolution of physical flows is continuous (see Remark \ref{stable under L^infty-small perturbation } or the text below). Thanks to some probability literature, we can find a topology with these properties, that is, the topology induced by Wasserstein $L^{\infty}$ metric. It is described, e.g., by McCann \cite{McC06}, Ambrosio, Brué and Semola \cite{ABS21} or Givens and Shortt \cite{GS84}.

In a metric space $(X, d), \mathcal{B}(X)$ denotes its Borel $\sigma$-algebra and $\mathcal{M}(X)$ the set of the $\sigma$-additive functions $\tilde{\mu}: \mathcal{B}(X) \rightarrow \mathbb{R}$. Furthermore, we define the \textit{nonnegative measure space} as:
$$
\mathcal{M}_{+}(X):=\{\tilde{\mu} \in \mathcal{M}(X): \tilde{\mu} \geq 0\}
$$

And define the \textit{probability (measure) space} as:
$$
\mathcal{P}(X)=\left\{\tilde{\mu} \in \mathcal{M}_{+}(X): \tilde{\mu}(X)=1\right\}
$$

\begin{definition}[Push Forward Measure]\label{push forward}
	Given a Borel function $f: X \rightarrow Y$, we define the \textit{push forward operator} $f_{\#}: M(X) \rightarrow M(Y)$ by
	$$
	f_{\#} \tilde{\mu}(B):=\tilde{\mu}\left(f^{-1}(B)\right) \quad  \text{ for all } \tilde{\mu}\in M(X) \text{ and all Borel sets } B \in \mathcal{B}(Y)
	$$
	
	And call $f_{\#} \tilde{\mu}$ \textit{push forward measure}.
\end{definition}

Let $(X, d)$ be a metric space and $(S, \Sigma, v)$ be a probability space, given $\rho, \kappa \in \mathcal{P}(X)$, the \textit{Wasserstein $L^{\infty}$ distance} between $\rho$ and $\kappa$ is defined as

\begin{equation} \label{Wasser.}
	W^{\infty}(\rho, \kappa):=\inf \left\{\begin{array}{l}
		\|d(f(x), g(x))\|_{L^{\infty},v} \mid f: S \rightarrow X  \text { Borel }, \\
		f_{\#} v=\rho \text { and } g: S \rightarrow X  \text { Borel, } g_{\#} v=\kappa
	\end{array}\right\}
\end{equation}

Here $\|h\|_{L^{\infty},v}$ denotes the supremum of $|h|$ over $S$, discarding sets of $v$-measure zero. We call $f, g$ the \textit{transport maps}. Thanks to Strassen's Theorem, one can check $W^{\infty}$ is truly a metric, as explained in Givens and Shortt \cite{GS84}.

\begin{remark}
	One might wonder whether we could find such transport maps. It turns out at least in some cases we can always find them. See for example \cite[Theorem 1.11, Theorem 5.2]{ABS21}.
\end{remark}

\begin{remark}\label{independent of measure space}
	It turns out  $W^{\infty}(\rho, \kappa)$ is actually independent of the probability space $(S, \Sigma, v)$ we choose \cite[Section 5]{McC06}. Due to this fact, we can choose $S=[0,1]$ with Lebesgue measure, or $(S, \Sigma, v)=(X, \mathcal{B}(X), \rho)$. If we consider the latter case, we have 
		
	\begin{equation}
		\begin{aligned}
			W^{\infty}(\rho, \kappa)&:=\inf \left\{\begin{array}{l}
				\|d(f(x), g(x))\|_{L^{\infty},\rho} \mid f: X \rightarrow X  \text { Borel}, \\
				f_{\#} \rho=\rho \text { and } g: X \rightarrow X  \text { Borel, } g_{\#} \rho=\kappa
			\end{array}\right\}\\
			&\leq \inf \left\{\|d(x, g(x))\|_{L^{\infty},\rho} \mid g: X \rightarrow X \text{ Borel}, g_{\#} \rho=\kappa\right\}
		\end{aligned}
	\end{equation}
	
	The right side of the inequality can be viewed as a (generalized) Monge's formulation of the optimal transport problem (see, e.g., \cite[Section 1.2 and Section 8.1]{ABS21}).
\end{remark}

\begin{remark}
	Notice in \cite[Section 8.1]{ABS21}, the authors define another Wasserstein $L^{\infty}$ distance (with an abuse of notation denoted by $\widetilde{W}^{\infty}$) via \textit{transport plans} and Kantorovich's formulation of the optimal transport problem. In general, we also have
	\[
	\widetilde{W}^{\infty}(\rho, \kappa) \leq \inf \left\{\|d(x, g(x))\|_{L^{\infty}(\rho)} \mid g: X \rightarrow X \text{ Borel}, g_{\#} \rho=\kappa\right\}
	\]
	see \cite[Section 2.2]{ABS21}. Moreover, when we consider Wasserstein $L^p$ distance given in \cite{ABS21}, $p<\infty$, under additional assumption(s) if needed, we have $$\widetilde{W}^p(\rho, \kappa)=\left(\inf \{\int_X d^p(x, g(x)) d \rho(x) \mid g: X \rightarrow X \text{ Borel}, g_{\#} \rho=\kappa\}\right)^{\frac{1}{p}}$$ 
	see \cite[Section 2.2 or Section 5.1]{ABS21}. However, the link between $\widetilde{W}^{\infty}$ and $W^{\infty}$ seems not clear. In the following, our arguments will be based on $W^{\infty}$'s definition.
\end{remark}



Since the functions in ${R}\left(\mathbb{R}^3\right)$ have unit mass, we can view them as measures. That is, given $\rho \in {R}\left(\mathbb{R}^3\right)$, define $\widetilde{\rho}$ as the following: $\forall B \in \mathcal{B}\left(\mathbb{R}^3\right), \widetilde{\rho}(B):=\int_B \rho\,dx$. One can check in this definition, $\widetilde{\rho} \in \mathcal{P}\left(\mathbb{R}^3\right), \widetilde{\rho}$ is absolutely continuous with respect to $\mu$ (denoted by $\widetilde{\rho} \ll \mu$), where $\mu$ is Lebesgue measure, and the derivative (density) of $\widetilde{\rho}$ with respect to $\mu$ is $\rho$ \cite[Section 1.6]{EG15}. With an abuse of notation, we denote $\widetilde{\rho}$ by $\rho$ as well. Therefore, we can define the Wasserstein $L^{\infty}$ distance between two functions in ${R}\left(\mathbb{R}^3\right)$. Note that although $W^{\infty}(\rho, \kappa)$ may be infinite on ${R}\left(\mathbb{R}^3\right)$, it is finite whenever $\rho$ and $\kappa$ are of bounded support by Lemma \ref{properties of Wasser.} (i) below.

As McCann explained \cite[Section 5]{McC06}, it turns out the Wasserstein $L^{\infty}$ metric is not unphysically strong: for instance, in the Lagrangian description of fluid mechanics \cite[Section 1.2]{BC18} \cite{Lag24} \cite[Chapter 1]{Lam32}, the state of a fluid system is specified by its original density profile $\rho \in {R}\left(\mathbb{R}^3\right)$, together with the positions of the fluid particles as a function of time. In Remark \ref{stable under L^infty-small perturbation }, we already see $\rho_t \in {R}\left(\mathbb{R}^3\right)$, as a function of time, evolves continuously with respect to the topology induced by Wasserstein $L^{\infty}$ metric, at least for bounded fluid velocities. Here ${Y_t}_{\#} \rho=\rho_t$, and $Y_t(y) \in {R}\left(\mathbb{R}^3\right)$ represents the position of the fluid at time $t$ which originated at $Y_0(y)=y$. Therefore, a local energy minimum $\rho\in {R}_{0}\left(\mathbb{R}^3\right)$ is physically stable.

We notice that the definition of Wasserstein distance can be generated to two functions in 
$$t {R}\left(\mathbb{R}^3\right)=\left\{\left.\rho \in L^{\frac{4}{3}}\left(\mathbb{R}^3\right) \right\rvert\, \rho \geq 0, \int_{\mathbb{R}^3} \rho\,dx=t\right\}$$
Here $t\neq0$. Moreover, one can observe that given $\rho, \sigma$ in ${R}\left(\mathbb{R}^3\right), W^{\infty}(t \rho, t \sigma)=W^{\infty}(\rho, \sigma)$ (or $\widetilde{W}^{\infty}(t \rho, t \sigma)=\widetilde{W}^{\infty}(\rho, \sigma))$, which is different from Wasserstein $L^p$ distance.

We give some elementary properties required of $W^{\infty}$. The first 5 properties essentially are picked from \cite[Lemma 5.1]{McC06}. Here $\operatorname{spt}(\rho-\kappa) \subset \mathbb{R}^3$ denotes the support of the signed measure $\rho-\kappa$, while a $\delta$-neighbourhood is defined as in \ref{notations} (viii).

\begin{lemma}[Simple Properties of the Wasserstein $L^{\infty}$ Metric]\label{properties of Wasser.}
	Let $\rho, \kappa$ in $t {R}\left(\mathbb{R}^3\right)$, then
	\begin{enumerate}[(i)]
		\item [\mylabel{i}{($i$)}] $W^{\infty}(\rho, \kappa)$ does not exceed the diameter of $\operatorname{spt}(\rho-\kappa)$;
		\setcounter{enumi}{1}
		\item if $W^{\infty}(\rho, \kappa)<\delta$, each connected component of the $\delta$-neighbourhood of spt $\rho$ has the same mass for $\kappa$ as for $\rho$;
		\item $W^{\infty}\left(\rho, f_{\#} \rho\right) \leq\|f-i d\|_{\infty, \rho}$ for $f: \mathbb{R}^3 \rightarrow \mathbb{R}^3$ measurable and $i d(x):=x$;
		\item the centers of mass satisfy $|\bar{x}(\rho)-\bar{x}(\kappa)| \leq W^{\infty}(\rho, \kappa)$;
		\item the moment of inertia $I(\rho)$ depends continuously on $\rho$;
		\item given $\rho \in t {R}\left(\mathbb{R}^3\right), \forall \epsilon>0, \exists \sigma \in t {R}\left(\mathbb{R}^3\right) \cap L^{\infty}\left(\mathbb{R}^3\right)$, such that $W^{\infty}(\rho, \sigma)<\epsilon$. Moreover, if spt $\rho$ is bounded, then spt $\sigma$ is also bounded.
	\end{enumerate}
\end{lemma}

\begin{proof}
	{\rm
		($i$-$v$) comes from \cite[Lemma 5.1]{McC06} with replacing ${R}\left(\mathbb{R}^3\right)$ by $t {R}\left(\mathbb{R}^3\right)$, which proofs are from the definition of $W^{\infty}$; the idea behind part (iii) can also be seen in Remark \ref{stable under L^infty-small perturbation }.
		
		To prove ($vi$), the intuitive strategy is to ``cut off'' the function values of  $\rho$ at the region where the value is relatively large, and then redistribute the value cut off to surrounding area where the values are relatively smaller, and then prove that the $W^\infty$ distance between the rearranged function and the original function is very small. To be precise, we carry out the following procedure. Without loss of generality we assume $t=1$, and we first divide the space $\mathbb{R}^3$ as 
		$$\mathbb{R}^3=\left(\cup_{n=1}^{\infty} C\left(x_n\right)\right) \cup D$$
		Here $\left\{C\left(x_n\right)\right\}$ is a family of (countably) many disjoint small open cubes, each with a diagonal length equal to $\frac{\epsilon}{2}$, volume equal to $V:=\left(\frac{\epsilon}{2 \sqrt{3}}\right)^3$, centered at $x_n$. $D$ is the union of the remaining boundaries, $0\in D$, and $\mu(D)=0$. Since 
		$$\int_{\mathbb{R}^3} \rho\,dx=1=\sum_{N=0}^{\infty} \int_{\{x \mid N \leq \rho(x)<N+1\}} \rho\,dx$$ 
		then $\exists R>1$, such that 
		$$\int_{\{x \mid \rho(x)>R\}} \rho\,dx<\frac{V}{4}$$
		Let $\rho_n=\rho \cdot \mathbf{1}_{C\left(x_n\right)}$, where $\mathbf{1}_{C\left(x_n\right)}$ is indicator function of set $C\left(x_n\right)$, and define
		$$
		\sigma_n(x)=\left\{\begin{array}{ll}
			2 R, & x \in\left\{x \mid \rho_n(x) \geq 2 R\right\} \cap C\left(x_n\right) \\
			\rho_n(x), & x \in\left\{x \mid R<\rho_n(x)<2 R\right\} \cap C\left(x_n\right) \\
			\rho_n(x)+\frac{\int_{\left\{x \mid \rho_n(x) \geq 2 R\right\}}\left(\rho_n-2 R\right)\,dx}{\mu\left(\left\{x \mid \rho_n(x) \leq R\right\} \cap C\left(x_n\right)\right)}, & x \in\left\{x \mid \rho_n(x) \leq R\right\} \cap C\left(x_n\right) \\
			0, & x \in \mathbb{R}^3 \backslash C\left(x_n\right)
		\end{array}\right.
		$$
		Since $\int_{\{x \mid \rho(x)>R\}} R\,dx<\int_{\{x \mid \rho(x)>R\}} \rho\,dx<\frac{V}{4}$, we know $\mu(\{x \mid \rho(x)>R\})<\frac{V}{4 R}<\frac{V}{4}$, then 
		$$
		\mu\left(\left\{x \mid \rho_n(x) \leq R\right\} \cap C\left(x_n\right)\right)=\mu\left(C\left(x_n\right)\right)-\mu\left(\left\{x \mid \rho_n(x)>R\right\} \cap C\left(x_n\right)\right)>\frac{3 V}{4}
		$$
		On the other hand, we have
		$$
		\int_{\left\{x \mid \rho_n(x) \geq 2 R\right\}}\left(\rho_n-2 R\right)\,dx \leq \int_{\left\{x \mid \rho_n(x) \geq 2 R\right\}} \rho_n \,dx\leq \int_{\{x \mid \rho(x)>R\}} \rho\,dx<\frac{V}{4}
		$$
		We obtain
		$$
		\frac{\int_{\left\{x \mid \rho_n(x) \geq 2 R\right\}}\left(\rho_n-2 R\right)\,dx}{\mu\left(\left\{x \mid \rho_n(x) \leq R\right\} \cap C\left(x_n\right)\right)}<\frac{1}{3}<\frac{3}{4}
		$$
		Therefore, if $x \in \left\{x \mid \rho_n(x) \leq R\right\} \cap C\left(x_n\right)$, we get $$\sigma_n(x) \leq \rho_n (x)+\frac{3}{4}\leq R+\frac{3}{4}<2 R$$
		Then we have $\sigma_n(x) \leq 2 R$ for all $x \in \mathbb{R}^3$. One can also check $\int_{\mathbb{R}^3} \rho_n\,dx=\int_{\mathbb{R}^3} \sigma_n\,dx$ (notice $\left\{x \mid \rho_n(x) \geq 2 R\right\}=\left\{x \mid \rho_n(x) \geq 2 R\right\} \cap C\left(x_n\right)$). Since both $\rho_n$ and $\sigma_n$ are supported in $C\left(x_n\right)$ and $C\left(x_n\right)$ has diameter $\frac{\epsilon}{2}$, then diameter of $\operatorname{spt}\left(\rho_n-\sigma_n\right)$ is not larger than $\frac{\epsilon}{2}$. Thanks to Property \ref{i}, if $\int_{\mathbb{R}^3} \rho_n\,dx=\int_{\mathbb{R}^3} \sigma_n\,dx>0$, we have $W^{\infty}\left(\rho_n, \sigma_n\right) \leq \frac{\epsilon}{2}$, where $W^{\infty}$ is the Wasserstein $L^{\infty}$ distance on $\mathcal{M}(\mathbb{R}^3)$. 
		
		Let $\sigma=\sum_{n=1}^{\infty} \sigma_n$, then $\sigma \in L^{\infty}\left(\mathbb{R}^3\right)$ and $\|\sigma\|_{L^{\infty}} \leq 2 R$. We know 
		$$\int_{\mathbb{R}^3}\sigma\,dx = \sum_{n}\int_{\mathbb{R}^3}\sigma_{n}\,dx = \sum_{n}\int_{\mathbb{R}^3}\rho_{n}\,dx = {\int_{\mathbb{R}^3}\rho\,dx} - {\int_{D}\rho\,dx}\stackrel{\mu\left(D\right)=0}{=}\int_{\mathbb{R}^3} \rho\,dx$$
		In particular, $\sigma$ has the same finite measure as $\rho$, and we can now estimate the Wasserstein $L^\infty$ metric between them. We claim: $W^{\infty}(\rho, \sigma)<\epsilon$. 
		
		In order to show the claim $W^{\infty}(\rho, \sigma)<\epsilon$, we first consider if $\int_{\mathbb{R}^3} \rho_n\,dx=\int_{\mathbb{R}^3} \sigma_n\,dx>0$, since $W^{\infty}\left(\rho_n, \sigma_n\right) \leq \frac{\epsilon}{2}$, we can set $\left( {S,\Sigma,\nu} \right) = \left( {C\left( x_{n} \right),M_{n},c_{n}\mu_{n}} \right)$, where $M_n$ is the collection of all Lebesgue measurable sets in $C\left( x_{n} \right)$, $c_n$ satisfies 
		$$c_{n}\mu_{n}\left( {C\left( x_{n} \right)} \right) = c_{n}V = \int_{\mathbb{R}^3}\rho_{n}\,dx = \int_{\mathbb{R}^3}\sigma_{n}\,dx>0$$
		and $\mu_n$ satisfies $\mu_n(E)=\mu(E\cap C(x_n))$ for all measurable sets $E$ in $\mathbb{R}^3$, recall $\mu$ is the Lebesgue measure on $\mathbb{R}^3$. In particular, we know $c_n>0$, and $\mu_n(C(x_n)^\mathsf{c})=0$, here $C(x_n)^\mathsf{c}=\mathbb{R}^3 \setminus C(x_n)$ is the complement of $C(x_n)$. Easy to verify that $\mu_n$ is a measure in  $\mathbb{R}^3$ (see for example \cite[Exercise 10 in Section 1.3]{Fol13}). Moreover, its restriction to $C(x_n)$  coincides with the Lebesgue measure on $C(x_n)$. Therefore, $\left( {C\left( x_{n} \right),M_{n},c_{n}\mu_{n}} \right)$ is a (finite) measure space, with $c_{n}u_{n}\left( {C\left( x_{n} \right)} \right) = \rho_{n}\left( \mathbb{R}^{3} \right) = {\int_{\mathbb{R}^3}\rho_{n}\,dx}$. Thanks to Remark \ref{independent of measure space}, we can choose $\left( {C\left( x_{n} \right),M_{n},c_{n}\mu_{n}} \right)$ as the measure space in the definition of $W^{\infty}\left(\rho_n, \sigma_n\right)$\footnote{The reason for choosing such a measure space is to find transport maps $f_n$ and $g_n$  defined on $C(x_n)$, which will facilitate the subsequent definition of  $\widetilde{f_n}$, $\widetilde{g_n}$, $f$ and $g$,  and the analysis of their properties.}, then we can find $f_n: C\left(x_n\right) \rightarrow \mathbb{R}^3$ and $g_n: C\left(x_n\right) \rightarrow \mathbb{R}^3$ with ${f_{n}}_{\#} (c_n\mu_n)=\rho_n, {g_{n}}_{\#} (c_n\mu_n)=\sigma_n$, and 
		$$\left\|f_n-g_n\right\|_{L^{\infty}\left(C\left(x_n\right)\right), c_n\mu_n}<W^{\infty}\left(\rho_n, \sigma_n\right)+\frac{\epsilon}{4}<\frac{3 \epsilon}{4}$$ 
		Notice since $c_n>0$, we have $c_n\mu_n \ll \mu_n$ and $\mu_n \ll c_n\mu_n$, and then $$\left\|f_n-g_n\right\|_{L^{\infty}\left(C\left(x_n\right)\right), \mu_n}=\left\|f_n-g_n\right\|_{L^{\infty}\left(C\left(x_n\right)\right), c_n\mu_n}<\frac{3 \epsilon}{4}$$ 
		Therefore, given a measurable set $B \subset \mathbb{R}^{3}$, we have $c_n\mu_{n}\left( {f_{n}^{- 1}(B)} \right) = \rho_{n}(B)$. Since we know  $\rho_n$ is supported in $C\left(x_n\right)$, then $\rho_{n}(B) = \rho_{n}\left( {B \cap C\left( x_{n} \right)} \right)$. Therefore, 
		$$c_n\mu_{n}\left( {f_{n}^{- 1}(B)} \right) = \rho_{n}(B) = \rho_{n}\left( {B \cap C\left( x_{n} \right)} \right) = c_n\mu_{n}\left( {f_{n}^{- 1}\left( {B \cap \left( {C\left( x_{n} \right)} \right)} \right)} \right)$$
		In particular, we know 
		$$c_n\mu_{n}\left( {f_{n}^{- 1}((C(x_n))^{\mathsf{c}})} \right) =c_n\mu_{n}\left( {f_{n}^{- 1}\left( \emptyset \right)} \right)=c_n\mu_{n}\left( {\left( \emptyset \right)} \right)=0$$
		and then $\mu_{n}\left( {f_{n}^{- 1}((C(x_n))^{\mathsf{c}})} \right)=0$ since $c_n>0$. Then we can define $\widetilde{f_n}:  C(x_n) \rightarrow \mathbb{R}^3$ as the following:
		$$\widetilde{f_{n}}(x) = \left\{ \begin{matrix}
			{f_{n}(x),~~~\text{if}~~ f_{n}(x) \in C\left( x_{n} \right)} \\
			{0,~~~\text{if}~~ f_{n}(x) \in (C\left( x_{n} \right))^\mathsf{c}}
		\end{matrix} \right. \footnote{The motivation behind this definition is to restrict the range to  $C(x_n) \cup \{0\}$. This setup helps to give the subsequent definition of $f$  and analyze properties of $f$.}
		$$
		Notice by construction we know $0\in D$, so $0\notin C(x_n)$ for any $n$. If  $0\notin B$, we can check ${\widetilde{f_{n}}}^{- 1}(B) = f_{n}^{- 1}\left( {B \cap C\left( x_{n} \right)} \right)$, then we have 
		$$c_{n}\mu_{n}\left( {{\widetilde{f_{n}}}^{- 1}(B)} \right) = c_{n}\mu_{n}\left( {f_{n}^{- 1}\left( {B \cap C\left( x_{n} \right)} \right)} \right) = \rho_{n}(B)$$
		If $0\in B$, we can check ${\widetilde{f_{n}}}^{- 1}(B) = f_{n}^{- 1}\left( {B \cap C\left( x_{n} \right)} \right) \cup f_{n}^{- 1}\left( {C\left( x_{n} \right)^\mathsf{c}} \right)$, and 
		$$
		\begin{aligned}
			&\quad \ c_{n}\mu_{n}\left( {f_{n}^{- 1}\left( {B \cap C\left( x_{n} \right)} \right)} \right) \\
			&\leq c_{n}\mu_{n}\left( {{\widetilde{f_{n}}}^{- 1}(B)} \right) \\
			&= c_{n}\mu_{n}\left( {f_{n}^{- 1}\left( {B \cap C\left( x_{n} \right)} \right) \cup f_{n}^{- 1}\left(  {C\left( x_{n} \right)^\mathsf{c}} \right) } \right) \\
			&\leq c_{n}\mu_{n}\left( {f_{n}^{- 1}\left( {B \cap C\left( x_{n} \right)} \right)} \right) + c_{n}\mu_{n}\left( {f_{n}^{- 1}\left( {C\left( x_{n} \right)^\mathsf{c}} \right)} \right) \\
			&=c_{n}\mu_{n}\left( {f_{n}^{- 1}\left( {B \cap C\left( x_{n} \right)} \right)} \right) \\
			&=\rho_n(B)
		\end{aligned}
		$$
		
		Therefore, we know $c_{n}\mu_{n}\left( {{\widetilde{f_{n}}}^{- 1}(B)} \right)= \rho_n (B)=c_{n}\mu_{n}\left( {f_{n}^{- 1}\left( {B \cap C\left( x_{n} \right)} \right)} \right)$ for any measurable set $B\subset \mathbb{R}^3$, and thus $\widetilde{f_{n}}_{\#} (c_n\mu_n)=\rho_n$. Similarly, we can construct $\widetilde{g_n}$ such that $\widetilde{g_{n}}_{\#} (c_n\mu_n)=\sigma_n$. Since $\left\|f_n-g_n\right\|_{L^{\infty}\left(C\left(x_n\right)\right), \mu_n}<\frac{3 \epsilon}{4}$, we can find $U_{n} \subset C\left( x_{n} \right)$ with $\mu_{n}\left( U_{n} \right) = 0$, and $${\sup\limits_{C{(x_{n})}\backslash U_{n}}\left| f_{n} - g_{n} \right|} = \left\| {f_{n} - g_{n}} \right\|_{L^{\infty}{({C{(x_{n})}})},\mu_{n}} < \frac{3\epsilon}{4}$$ 
		Take $\widetilde{U_{n}} \coloneq U_{n} \cup f_{n}^{- 1}\left( \left( {C\left( x_{n} \right)} \right)^\mathsf{c} \right) \cup g_{n}^{- 1}\left( \left( {C\left( x_{n} \right)} \right)^\mathsf{c} \right)$, since we know 
		$$\mu_{n}\left( {f_{n}^{- 1}((C(x_n))^{\mathsf{c}})} \right)=\mu_{n}\left( {g_{n}^{- 1}((C(x_n))^{\mathsf{c}})} \right)=0$$
		then $\mu_{n}\left( \widetilde{U_{n}} \right) = 0$, and we have
		$$
		{\sup\limits_{C(x_n)\backslash\widetilde{U_{n}}}\left| \widetilde{f_{n}} - \widetilde{g_{n}}~ \right|} ={\sup\limits_{C{(x_{n})}\backslash\widetilde{U_{n}}}\left| {f_{n}} - {g_{n}}~ \right|} \leq {\sup\limits_{C{(x_{n})}\backslash U_{n}}\left| f_{n} - g_{n} \right|} = \left\| {f_{n} - g_{n}} \right\|_{L^{\infty}{({C{(x_{n})}})},\mu_{n}} < \frac{3\epsilon}{4}
		$$
		Therefore, $\left\|\widetilde{f_n}-\widetilde{g_n}\right\|_{L^{\infty}\left(C(x_n)\right), \mu_n}<\frac{3 \epsilon}{4}$.
		
		We define $\nu = {\sum\limits_{n}{c_{n}\mu_{n}}}$, where $c_n$ satisfies $c_{n}\mu_{n}\left( {C\left( x_{n} \right)} \right) = c_{n}V = \int_{\mathbb{R}^3}\rho_{n}\,dx = \int_{\mathbb{R}^3}\sigma_{n}\,dx$. Notice here we allow $c_n=0$. We can see $\nu$ is a measure on $\mathbb{R}^3$ and 
		$$\nu\left( \mathbb{R}^{3} \right) = \sum_{n}c_{n}\mu_{n}\left( \mathbb{R}^{3} \right) = \sum_{n}c_{n}\mu_{n}\left( {C\left( x_{n} \right)} \right) = \sum_{n}\int_{\mathbb{R}^3}\rho_{n}\,dx = \int_{\mathbb{R}^3}\rho\,dx = \int_{\mathbb{R}^3}\sigma\,dx=1 <\infty$$
		Thus $\nu$ is a finite measure.
		
		Denote $N_{0} = \left\{ n \middle| c_{n} = 0 \right\}$, and $N_{1} = \left\{ n \middle| c_{n} > 0 \right\}$. We further define $f: \mathbb{R}^3 \rightarrow \mathbb{R}^3$ as 
		
		$$f(x)=\left\{\begin{array}{ll}\widetilde{f_n}(x),& x \in C\left(x_n\right) \text{with} ~n \in N_1 \\ x, &\text{otherwise}\end{array}\right.$$
		
		Similarly, let $g: \mathbb{R}^3 \rightarrow \mathbb{R}^3$ be 
		
		$$g(x)=\left\{\begin{array}{ll}\widetilde{g_n}(x), &x \in C\left(x_n\right) \text{with} ~n \in N_1\\  x, &\text{otherwise}\end{array}\right.$$
		
		Then 
		$$\|f-g\|_{L^{\infty}\left(\mathbb{R}^3\right), \nu} \leq \sup\limits_{\{n|c_n>0\} }\left\|\widetilde{f_n}-\widetilde{g_n}\right\|_{L^{\infty}\left(C\left(x_n\right)\right), \mu_n} \leq \frac{3 \epsilon}{4}$$
		Given a measurable set $B \subset \mathbb{R}^3$, let $B_n=B \cap C\left(x_n\right)$ be the measurable set in $C\left(x_n\right), B_D=B \cap D$, we know $\mu(D) = 0$ implies $\mu\left(B_D\right)=0$. We also have if $n\in N_0$, then 
		$$\rho\left( B_{n} \right) = \rho_{n}\left( B_{n} \right) \leq \rho_{n}\left( {C\left( x_{n} \right)} \right) = {\int_{\mathbb{R}^3}\rho_{n}\,dx} = 0$$
		and then 
		\[
		\begin{aligned}
			\rho(B) &= \int_B \rho\,dx \\
			&= \int_{\cup_{n=1}^{\infty} B_n} \rho\,dx + \int_{B_D} \rho\,dx \\
			&= \int_{\cup_{n=1}^{\infty} B_n} \rho\,dx \qquad (\text{since }\mu(B_D)=0)\\
			&= \int_{\cup_{n \in N_1} B_n} \rho\,dx \\
			&= \sum_{n \in N_1} \rho(B_n) \\
			&= \sum_{n \in N_1} \rho_n(B_n) \\
			&= \sum_{n \in N_1} c_n \mu_n\bigl( \widetilde{f_n}^{-1}(B_n) \bigr) \qquad (\text{since } \widetilde{f_n}_{\#}(c_n\mu_n)=\rho_n)
		\end{aligned}
		\]
		
		We can check $f^{- 1}\left( \left\{ 0 \right\} \right) = {\bigcup\limits_{n \in N_{1}}{{\widetilde{f}}_{n}}^{- 1}}\left( \left\{ 0 \right\} \right) \cup \left\{ 0 \right\}$. Then if $n\in N_1$, since ${\widetilde{f_{n}}}_{\#}\left( {c_{n}\mu_{n}} \right) = \rho_{n}$, we have $$c_{n}\mu_{n}\left( {f^{- 1}\left( \left\{ 0 \right\} \right)} \right) = c_{n}\mu_{n}\left( {{\widetilde{f_{n}}}^{- 1}\left( \left\{ 0 \right\} \right)} \right) = \rho_{n}\left( \left\{ 0 \right\} \right) = 0$$
		Easy to check if $n\in N_0$, $f^{- 1}\left( \left\{ 0 \right\} \right) \cap C\left( x_{n} \right) = \varnothing$, then $\mu_{n}\left( {f^{- 1}\left( \left\{ 0 \right\} \right)} \right) = 0$. Thus we know $v\left( {f^{- 1}\left\{ 0 \right\}} \right) = {\sum_{n}{c_{n}\mu_{n}\left( {f^{- 1}\left( \left\{ 0 \right\} \right)} \right)}} =0$ (alternatively, this can also be seen by observing that $c_n = 0$ for all $n \in N_0$).
		
		If $n\in N_1$, we can check
		\begin{enumerate}
			\item  ${\widetilde{f_{n}}}^{- 1}\left( B_n \right) = f^{- 1}\left( B_n \right)$
			\item $f^{- 1}\left( {B\backslash C\left( x_{n} \right)} \right) \cap C\left( x_{n} \right) \subset {\widetilde{f_{n}}}^{- 1}\left( \left\{ 0 \right\} \right) \subset f^{- 1}\left( \left\{ 0 \right\} \right)$, together with $\mu_n$ is supported in $C(x_n)$, we know 
			$$\mu_{n}\left( {f^{- 1}\left( {B\backslash C\left( x_{n} \right)} \right)} \right) = \mu_{n}\left( {f^{- 1}\left( {B\backslash C\left( x_{n} \right)} \right) \cap C\left( x_{n} \right)} \right) = 0$$
			\item 
			$$
			\begin{aligned}
				\mu_{n}\left( {f^{- 1}(B)} \right) &= \mu_{n}\left( {f^{- 1}\left( B_n \right) \cup f^{- 1}\left( {B\backslash C\left( x_{n} \right)} \right)} \right) \\
				&= \mu_{n}\left( {f^{- 1}\left( B_n \right)} \right) + \mu_{n}\left( {f^{- 1}\left( {B\backslash C\left( x_{n} \right)} \right)} \right) \\
				&= \mu_{n}\left( {f^{- 1}\left( B_n \right)} \right) \\
				&= \mu_{n}\left( {{\widetilde{f_{n}}}^{- 1}\left( B_n \right)} \right)
			\end{aligned}
			$$
		\end{enumerate}
		Therefore, we can compute
		$$
		\begin{aligned}
			\rho(B)&={\sum_{n \in N_{1}}{c_{n}\mu_{n}\left( {{\widetilde{f_{n}}}^{- 1}\left( B_{n} \right)} \right)}}\\
			&={\sum_{n \in N_{1}}{c_{n}\mu_{n}\left( {f^{- 1}\left( B\right)} \right)}}\\
			&={\sum_{n \in N_{1}}{c_{n}\mu_{n}\left( {f^{- 1}\left( B\right)} \right)}}+{\sum_{n \in N_{0}}{c_{n}\mu_{n}\left( {f^{- 1}\left( B\right)} \right)}} \qquad (\text{since } c_n=0 \text{ when } n\in N_0)\\
			&={\sum_{n=1}^{\infty}{c_{n}\mu_{n}\left( {f^{- 1}\left( B\right)} \right)}}\\
			&=\nu(f^{-1}(B))\\
			&=f_{\#}\nu (B)
		\end{aligned}	
		$$

		Similarly, $g_{\#} \nu(B)=\sigma(B)$. Recall $\sigma \in L^\infty(\mathbb{R}^3)$. Therefore, $f_{\#} \nu=\rho$ and $g_{\#} \nu=\sigma$. Moreover, $W^{\infty}(\rho, \sigma) \leq \|f-g\|_{L^{\infty}\left(\mathbb{R}^3\right), \nu} \leq \frac{3 \epsilon}{4}<\epsilon$. By the construction of $\sigma$ we know if spt $\rho$ is bounded, then $\operatorname{spt} \sigma$ is also bounded.
	}
\end{proof}


In Theorem \ref{Properties of LEM} we assume $(\rho, v)$ minimizes $E(\rho, v)$ locally on ${R}_0\left(\mathbb{R}^3\right) \times {V}\left(\mathbb{R}^3\right)$, which implies $\rho$ minimizes $E_J(\rho)$ locally on ${R}_0\left(\mathbb{R}^3\right)$ by Corollary \ref{velocity-free}. 
Although in the paper of Auchmuty and Beals \cite{AB71}, the center of mass after perturbation is not necessarily zero, it turns out that a local energy minimizer $\rho$ on ${R}_0\left(\mathbb{R}^3\right)$ is also stable under perturbations which shift its center of mass:


\begin{corollary}[Center-free and Support-free Minimizer {\cite[Corollary 5.2]{McC06}}]\label{center/support free}
	If $\rho$ minimizes $E_J(\rho)$ locally on ${R}_0\left(\mathbb{R}^3\right)$, then it minimizes $E_J(\rho)$ locally on ${R}\left(\mathbb{R}^3\right)$.
\end{corollary}

\begin{proof}
	{
		\rm There exists $\delta>0$ such that $E_J(\rho) \leq E_J(\kappa)$ whenever $\kappa \in {R}_0\left(\mathbb{R}^3\right)$ with $W^{\infty}(\rho, \kappa)<2 \delta$. Now, suppose $\kappa \in {R}\left(\mathbb{R}^3\right)$ with $W^{\infty}(\rho, \kappa)<\frac{\delta}{2}$. Thanks to Lemma \ref{properties of Wasser.} $(ii)$, we know $s p t ~ \kappa$ is also bounded. Lemma \ref{properties of Wasser.} $(iv)$ shows that $|x(\kappa)|<\frac{\delta}{2}$. Let $\widetilde{\kappa}(x)=\kappa(x+x(\kappa))$, then $\widetilde{\kappa}\in {R}_0\left(\mathbb{R}^3\right)$ and Lemma \ref{properties of Wasser.} $(iii)$ shows that $W^{\infty}(\widetilde{\kappa}, \kappa)<\frac{\delta}{2}$. By triangle inequality (recall $W^\infty$ is a metric), we have $W^{\infty}(\widetilde{\kappa}, \rho)\leq W^{\infty}(\widetilde{\kappa}, \kappa)+W^{\infty}(\kappa, \rho)<\delta$, i.e. $\widetilde{\kappa}$ lies within $\delta$-neighborhood of $\rho$ in ${R}_0\left(\mathbb{R}^3\right)$. By translation invariance, $E_J(\rho) \leq E_J(\widetilde{\kappa}) =E_J(\kappa)$.
	}
\end{proof}

We should point out: suppose $\rho$ minimizes $E_J(\rho)$ locally on ${R}_0\left(\mathbb{R}^3\right)$ and let $\sigma \in L^{\infty}\left(\mathbb{R}^3\right)$. Even if the perturbation satisfies $\rho+t \sigma \in {R}\left(\mathbb{R}^3\right)$ for $t \in[0,1]$, it may not be $W^{\infty}$-continuous as function of $t$, and we do not know if $E_J(\rho+t\sigma)\geq E_J(\rho)$ for $t \in[0,1]$; nevertheless, when $\sigma$ is supported on a set which diameter is small enough, with $\int_{\mathbb{R}^3} \sigma \,dx=0$, then $W^{\infty}(\rho, \rho+t\sigma)$ is small enough for all $t$ thanks to Lemma \ref{properties of Wasser.} $(i)$, hence $\rho$ minimizes $E_J(\rho+t \sigma)$ for all $t$. Therefore, $\sigma$ will then be a useful variation of $E_J(\rho)$.

To be precise, we define the set of admissible perturbations, which depend on $\rho$, as:
$$
P_{\infty}(\rho):=\bigcup_{R<\infty} P_R(\rho)
$$

Here, $P_R(\rho)$ is defined as:
$$P_R(\rho)=\left\{ \sigma \in L^{\infty}(\mathbb{R}^3)\mid \begin{array}{ll}
	\sigma(x)=0, & \text{where } x \text{ statisfies } \rho(x) >R \text{ or } |X|>R\\
	\sigma(x)\geq 0, & \text{where } x \text{ statisfies } \rho(x) <R^{-1}
\end{array} \right\}$$

Note that this definition does not a priori require the support of $\sigma$ to have small diameter; such a restriction will be imposed later when needed. The present formulation is convenient because certain papers that do not rely on the $W^\infty$ distance do not necessarily require such diameter condition.

One can see $P_{\infty}(\rho)$ is a convex cone.

\begin{remark}\label{test function with positive measure support}
	Given a set $E$ with positive measure, when $R$ is large enough, we can always find a non-zero function $\sigma \in P_R(\rho)$, with support in $E$. In fact, since $\rho$ is integrable implies $\rho$ is finite almost everywhere, we have 
	$$
	\begin{aligned}
		1&=\int_{\mathbb{R}^3} \rho\,dx\\
		&= \sum_{N=0}^{\infty} \int_{\{x \mid N \leq \rho(x)<N+1\}} \rho\,dx \\
		&\geq \sum_{N=0}^{\infty} \int_{\{x \mid N \leq \rho(x)<N+1\}} N\,dx\\
		&=\sum_{N=0}^{\infty} N \mu(\{x \mid N \leq \rho(x)<N+1\}) \\
		&\geq \sum_{N=1}^{\infty} \mu(\{x \mid N \leq \rho(x)<N+1\})
	\end{aligned}	
	$$ 
	Notice when $R \rightarrow \infty$, 
	$$0 \leq \mu\left(\{x \in E \mid \rho(x)>R\} \leq \mu(\{x \in \mathbb{R}^3 \mid \rho(x)>R\right\}\leq\sum_{N=[R]}^{\infty} \mu(\{x \mid N \leq \rho(x)<N+1\}) \rightarrow 0$$
	Hence 
	$$\mu(\{x \in E \mid \rho(x) \leq R\})=\mu(E)-\mu(\{x \in E \mid \rho(x)>R\} \rightarrow \mu(E)$$
	By similar arguments in the following paragraph about how to get $\widetilde{E}$ from $\widetilde{E}_k$, we can show when $R$ is large enough, $\mu(\{x \in E \mid \rho(x) \leq R\})> \frac{1}{2} \mu(E)>0$.

	Pick one of such $R$, and let $\widetilde{E}_k=\{x \in E\mid \rho(x) \leq R, | x| \leq k\}$, then $\left\{\widetilde{E}_k\right\}_{k=1}^{\infty}$ is an ascending collection of measurable sets, and $\cup_{k=1}^{\infty} \widetilde{E}_k=\{x \in E \mid \rho(x) \leq R\}$, therefore, $\mu(\{x \in E \mid \rho(x) \leq R\})=\lim\limits_{k \rightarrow \infty} \mu\left(\widetilde{E}_k\right)$. (see \cite[Section 2.5]{RF10}). When $k$ is large enough, 
	$$\mu(\{x \in E\mid \rho(x) \leq R, | x|\leq k\})>\frac{1}{2} \mu(E)>0$$
	Then replace $R$ by another larger $R$ if needed, we have $\mu(\{x \in E\mid \rho(x) \leq R, | x |\leq R\})>0$. We set $\widetilde{E}=\{x \in E\mid \rho(x) \leq R, | x |\leq R\}$, then $\sigma(x)=\left\{\begin{array}{ll}1, &x \in \widetilde{E} \\ 0, &\text {otherwise }\end{array}\right.$ has the properties mentioned at the beginning of this remark. In particular, one can see $0 \in P_{\infty}(\rho)$, but $P_{\infty}(\rho) \neq\{0\}$. Notice in $\widetilde{E}$, $\rho(x)$ is bounded and $0\leq \rho(x) \leq R$.
\end{remark}

\begin{remark}\label{test function with minimizer bounded}
	Similar as above, given a set $E$ with positive measure, and $\rho$ has positive mass in $E$ which implies $\mu\{x\in E \mid \rho(x)>0\}>0$, then when $R$ is large enough, let $E_R=\{x \in E\mid | x | \leq R, R^{-1}<\rho(x)<R\}$, we have $\mu\left(E_R\right)>0$, and $\mathbf{1}_{E_R} \in P_R(\rho) \subset P_0(\rho)$.
\end{remark}

\begin{remark}\label{nonnegative perturbation}
	Given $\rho \in R(\mathbb{R}^3)$, and $\sigma \in P_\infty(\rho)$, then by definition of $P_\infty(\rho)$, we know $\exists R>0$, such that $\sigma \in P_R(\rho)$, one can check when t is small enough, $\rho+t\sigma\geq0$, thus we can also understand from this perspective, that $\rho+t\sigma\geq0$ is a reasonable perturbation.
\end{remark}
\section{Analysis of Lagrange Multiplier and Properties of Local Minimizers}\label{section5-constant chemical potential}

This section is divided into three subsections. In the first subsection, we compute the variational derivative of $E_J$, and prove the finite energy property of local minimizers, which ensures the variational derivative makes sense. In the second subsection, we revisit and modify the proofs of ``weak'' versions of Theorem \ref{Properties of LEM} ($iv$-$v$). Here ``weak'' means (\ref{EL}) holds true for almost all points, rather than pointwise, in the region we consider. In the third subsection, we discuss the nonexistence of finite energy local minimizers under the topology inherited from a topological vector space, thereby suggesting the particularity of the topology induced by the $W^\infty$ metric.


\subsection{Variational derivative of $E_J$ and finite energy of local minimizer}\label{subsection5.1-variational derivative}
Note that an element $\rho \in R(\mathbb{R}^3)$ might have infinite total energy $E_J(\rho)$. A simple example with infinite total energy could be $\rho(x)\coloneq\rho_{N}(x) = C\frac{1}{|x|^{\alpha}}\mathbf{1}_{\{{{|x|} < N}\}}$ with suitable $C$, $\alpha$, and $N$. In order to consider the variational derivative $E_J^{\prime}(\rho)$ of the energy $E_J(\rho)$, we first determine the space $W \subsetneq R(\mathbb{R}^3)$ where we want to take derivatives:
$$X = L^{\frac{4}{3}}\left( \mathbb{R}^{3} \right)\cap L^{1}\left( \mathbb{R}^{3} \right)$$
$$U=\{\rho \in X \mid \rho \geq 0 \text{ , } U(\rho)< \left.\infty\right\}$$ 
$$W=\left\{\rho \in U \mid \int_{\mathbb{R}^3} \rho\,dx=1\right\}$$

\begin{remark}\label{equivalence of finite total energy and finite internal energy}
	Suppose $\rho \in X$ is non-negative and positive mass, then by Remark \ref{positive MoI} we know $I(\rho)>0$, thus $T_J(\rho)$ is finite. Thanks to Proposition \ref{finite gravitational interaction energy}, we have $G(\rho, \rho)<\infty$. Therefore, $$\left. \rho \in U\Longleftrightarrow U(\rho) < \infty\Longleftrightarrow E_{J}(\rho) < \infty \right.$$
\end{remark}

It is helpful if we have a criterion for determining when to obtain $U(\rho)<\infty$.

\begin{lemma}[Finite Internal Energy]\label{finite internal energy}
	If a nonnegative function $\sigma$ is in $L^1(\mathbb{R}^3) \cap L^\infty (\mathbb{R}^3)$, then $U(\sigma)<\infty$.
\end{lemma} 

\begin{proof}
	{
		\rm
		Thanks to Remark \ref{assumptions from P to A}, we know  $A(s)$ also satisfies \ref{F2} and \ref{F3}. In particular, $\exists \delta>0$, such that if $s \in [0,\delta)$, then $A(s)<s^{\frac{4}{3}}$. We know $$U(\sigma)=\int_{\mathbb{R}^3} A(\sigma(x)) \,dx=\int_{\{\sigma<\delta\}} A(\sigma(x)) \,dx+\int_{\{\delta \leq \sigma \leq C\}} A(\sigma(x)) \,dx=:U_1+U_2$$ where $C=\|\sigma\|_{L^\infty(\mathbb{R}^3)}$. By the choice of $\delta$ we know 
		$$U_1\leq \int_{\{\sigma<\delta\}} \sigma^{\frac{4}{3}}(x) \,dx\leq\int_{\mathbb{R}^3} \sigma^{\frac{4}{3}}(x) \,dx$$
		Since $\sigma \in L^{1}\left(\mathbb{R}^3\right) \cap L^{\infty}\left(\mathbb{R}^3\right)$, we know by Interpolation Inequality \cite[Section 4.2]{Bre11} that $U_1\leq\int_{\mathbb{R}^3} \sigma^{\frac{4}{3}}(x)\,dx<\infty$. Since $A(s)$ is increasing, we know 
		$$
		\begin{aligned}
			U_2&\leq \int_{\{\delta \leq \sigma \leq C\}} A(C)\,dx\\
			&\leq A(C) \int_{\{\delta \leq \sigma \leq C\}} \frac{\sigma(x)}{\delta}\,dx\\
			&\leq \frac{A(C)}{\delta}\int_{\{\delta \leq \sigma \leq C\}} {\sigma(x)}\,dx\\
			&\leq\frac{A(C)}{\delta}\|\sigma\|_{L^{1}(\mathbb{R}^3)}\\
			&\leq\frac{A(C)}{\delta}<\infty
		\end{aligned}
		$$
		Therefore, we know $U(\sigma)<\infty$.
	}
\end{proof}

To discuss the differentiability of $E_J$, we modify the arguments of Auchmuty and Beals in \cite[Section 4]{AB71} and have the following lemma:

\begin{lemma}[Differentiability of Energy $E_J(\rho)$]\label{diff. of energy}
	Given $\rho \in W$, we have $P_{\infty}(\rho) \subset P(\rho)$, and $E_J(\rho)$ is $P_{\infty}(\rho)$-differentiable at $\rho$. Here $P(\rho)$ and $P_{\infty}(\rho)$-differentiability (i.e. differentiability at $\rho$ in the direction of $P_\infty (\rho)$) are provided in Appendix \ref{sectionA-calculus of variations in vector spaces}. Moreover, the derivative at $\rho$ is $E_J^{\prime}(\rho)$ in the sense that $\forall \sigma\in P_{\infty}(\rho)$, $E_J^{\prime}(\rho)(\sigma)=\int_{\mathbb{R}^3} E_J^{\prime}(\rho) \sigma\,dx$\footnote{To remain consistent with the notation in \cite[Section 4]{AB71}, we use $E_J^{\prime}$ as the symbol for both linear functional and function,  provided it does not cause confusion.}, where on the right hand side the function $E_J^{\prime}(\rho)$ is given by
	\begin{equation}\label{variational derivative}
		E_J^{\prime}(\rho)(x):=A^{\prime}(\rho(x))-V_\rho(x)-\frac{J^2}{2 I^2(\rho)} r^2(x-\bar{x}(\rho))
	\end{equation}
\end{lemma}

\begin{proof}
	{
		\rm
		Given $\sigma \in P_{\infty}(\rho)$, then $\exists R>0$ such that $\sigma \in P_R(\rho)$. By construction we know $\sigma \in L^{\infty}\left(\mathbb{R}^3\right)$ with compact support, thus by Hölder's inequality we know $\sigma \in L^p\left(\mathbb{R}^3\right)$ for $p>0$. Since $\rho \in W$, $\int_{\mathbb{R}^3} \rho\,dx =1$, then when $t$ is small enough, we have $\int_{\mathbb{R}^3} (\rho+t \sigma)\,dx>0$. Thanks to Remark \ref{nonnegative perturbation} and Remark \ref{equivalence of finite total energy and finite internal energy}, we know $\rho+t\sigma$ is nonnegative, $T_J(\rho+t\sigma)<\infty$ and $G(\rho+t\sigma, \rho+t\sigma)<\infty$. Moreover, to show $\rho+t\sigma \in U$, it suffices to show $U(\rho+t\sigma)<\infty$. 
		
		For almost every $x \in \mathbb{R}^3, \rho(x)$ is finite, and $\sigma(x) \leq\|\sigma\|_{L^\infty}$. For such $x$ fixed and $0<t<1$, by mean value theorem we know $\exists \theta \in(0, t)$, such that
		\begin{equation}\label{difference of A by mean value theorem}
			A(\rho(x)+t \sigma(x))-A(\rho(x))=A^{\prime}(\rho(x)+\theta \sigma(x)) t \sigma(x)
		\end{equation}
		
		Set $E_R=\left\{x \in \mathbb{R}^3 \mid \rho(x) \leq R\right.$ and $\left.|x| \leq R\right\}$. If $x \notin E_R$, by construction $\sigma(x)=0$, thus
		$$
		A(\rho(x)+t \sigma(x))-A(\rho(x)=0
		$$
		
		If $x \in E_R$, since $A^{\prime}$ is increasing and non-negative, we have
		$$
		|A(\rho(x)+t \sigma(x))-A(\rho(x))| \leq A^{\prime}\left(R+\|\sigma\|_{L^{\infty}}\right)\|\sigma\|_{L^{\infty}}\cdot t
		$$
		
		Therefore, we have 
		$$
		|A\left( {\rho + t\sigma} \right)| \leq |A(\rho)| + A^{\prime}\left( {R + \left\| \sigma \right\|_{L^{\infty}}} \right)\left\| \sigma \right\|_{L^{\infty}} \cdot \mathbf{1}_{E_R}\cdot t
		$$
		The right hand side is an integrable function. Then we know $U(\rho+t\sigma)=\int_{\mathbb{R}^3} A(\rho+t\sigma)\,dx<\infty$. Then we know $\rho+t \sigma \in U$ for $t$ small enough. Therefore, $\sigma \in P(\rho)$ and then $P_{\infty}(\rho) \subset P(\rho)$.  
		
		The observations above are also helpful to calculate the derivative. We notice $|\frac{A(\rho(x)+t \sigma(x))-A(\rho(x))}{t}|$ is also bounded by an integrable function $A^{\prime}(R+\|\sigma\|_{L^{\infty}})\|\sigma\|_{L^{\infty}} \cdot \mathbf{1}_{E_R}$. Moreover, by (\ref{difference of A by mean value theorem}) we know $\frac{A(\rho(x)+t \sigma(x))-A(\rho(x))}{t}$ converges a.e. to $A^{\prime}(\rho(x)) \sigma(x)$ as $t \rightarrow 0$. By the dominated convergence theorem, we have
		\begin{equation}\label{derivative of internal energy}
			\lim _{t \rightarrow \infty} \frac{\int_{\mathbb{R}^3} A(\rho+t \sigma)\,dx-\int_{\mathbb{R}^3} A(\rho)\,dx}{t}=\int_{\mathbb{R}^3} A^{\prime}(\rho) \cdot \sigma\,dx
		\end{equation}

		Now consider gravitational interaction energy. Due to Remark \ref{finite gravitational interaction energy with different objects}, we know $G(\rho, \rho)=\int_{\mathbb{R}^3} \rho V_\rho\,dx <\infty$ and $G(\rho+t\sigma, \rho+t\sigma) <\infty$, then we compute
		$$
		\frac{G(\rho+t \sigma, \rho+t \sigma)-G(\rho, \rho)}{t}=G(\rho, \sigma)+G(\sigma, \rho)-t G(\sigma, \sigma) \rightarrow G(\rho, \sigma)+G(\sigma, \rho)
		$$
		
		Since $G(,\cdot ,)$ is a symmetric bilinear form, thus 
		$$G(\rho, \sigma)+G(\sigma, \rho)=2 G(\sigma, \rho)=2 \int_{\mathbb{R}^3} \sigma V_\rho\,dx$$
		
		Consider $T_J(\rho)=\frac{J^2}{2 I(\rho)}$, when $t$ is small enough, we have $\int_{\mathbb{R}^3} (\rho+t \sigma)\,dx>0$, thus $I(\rho+t \sigma)>0$. and $T_J(\rho+t \sigma)$ is well-defined. Easy to check the mass of $t \sigma$ is $m(t \sigma)=t m(\sigma)$, the center of mass is $\bar{x}(t \sigma)=\bar{x}(\sigma)$, and the moment of inertia is $I(t \sigma)=t I(\sigma)$. Thanks to Lemma \ref{expansion of MoI}, we can expand $I(\rho+t \sigma)$ and get
		$$
		\begin{aligned}
			\frac{T_J(\rho+t \sigma)-T_J(\rho)}{t} & =\frac{J^2(I(\rho)-I(\rho+t \sigma))}{2 I(\rho) I(\rho+t \sigma) \cdot t} \\
			& =\frac{J^2\left(-I(t \sigma)-\left(\frac{m(\rho) m(t \sigma)}{m(\rho)+m(t \sigma)}\right) r^2(\bar{x}(\rho)-\bar{x}(t \sigma))\right)}{2 I(\rho) I(\rho+t \sigma) \cdot t} \\
			& =\frac{J^2\left(-I(\sigma)-\left(\frac{m(\rho) m(\sigma)}{m(\rho)+m(t \sigma)}\right) r^2(\bar{x}(\rho)-\bar{x}(\sigma))\right)}{2 I(\rho) I(\rho+t \sigma)} \\
			& \overset{t \rightarrow 0}{\rightarrow} \frac{J^2\left(-I(\sigma)-m(\sigma) \cdot r^2(\bar{x}(\rho)-\bar{x}(\sigma))\right)}{2 I^2(\rho)}
		\end{aligned}
		$$
		
		Consider
		$$
		\begin{aligned}
			I(\sigma)+m(\sigma) \cdot r^2(\bar{x}(\rho)-\bar{x}(\sigma)) =&\int_{\mathbb{R}^3} \sigma \cdot\left(r^2(x-\bar{x}(\sigma))+r^2(\bar{x}(\rho)-\bar{x}(\sigma))\right)\,dx \\
			=&\int_{\mathbb{R}^3} \sigma \cdot(\langle x, x \rangle_2-2\langle\bar{x}(\sigma), x\rangle_2+\langle\bar{x}(\sigma), \bar{x}(\sigma)\rangle_2\\
			&+\langle\bar{x}(\rho), \bar{x}(\rho)\rangle_2-2\langle\bar{x}(\sigma), \bar{x}(\rho)\rangle_2+\langle\bar{x}(\sigma), \bar{x}(\sigma)\rangle_2\\
			&-2\langle\bar{x}(\rho), x\rangle_2+2\langle\bar{x}(\rho), x\rangle_2)\,dx \\
			=&\int_{\mathbb{R}^3} \sigma \cdot(r^2(x-\bar{x}(\rho))+2\langle\bar{x}(\sigma), \bar{x}(\sigma)\rangle_2\\
			&-2\langle\bar{x}(\sigma), x\rangle_2-2\langle\bar{x}(\sigma), \bar{x}(\rho)\rangle_2+2 \langle\bar{x}(\rho), x\rangle_2)\,dx\\
			=&\int_{\mathbb{R}^3} \sigma \cdot r^2(x-\bar{x}(\rho))\,dx+2 m(\sigma)\langle\bar{x}(\sigma), \bar{x}(\sigma)\rangle_2-2 m(\sigma)\langle\bar{x}(\sigma), \bar{x}(\sigma)\rangle_2\\
			&-2 m(\sigma) \langle\bar{x}(\sigma), \bar{x}(\rho)\rangle_2+2 m(\sigma)\langle\bar{x}(\rho), \bar{x}(\sigma)\rangle_2\\
			=&\int_{\mathbb{R}^3} r^2(x-\bar{x}(\rho)) \cdot \sigma\,dx
		\end{aligned}
		$$
		Therefore, $\frac{T_J(\rho+t \sigma)-T_J(\rho)}{t} \rightarrow \int_{\mathbb{R}^3} \frac{-J^2}{2 I^2(\rho)} \cdot r^2(x-\bar{x}(\rho)) \cdot \sigma\,dx$ as $t \rightarrow 0$. Collecting the results above, we finish the proof of Lemma \ref{diff. of energy}.
	}
\end{proof}

\begin{remark}
	One can generate the results to any $\rho \in U$ with positive mass.
\end{remark}

\begin{remark}
	The admissible perturbations are a bit different from $P_0(\rho)$ mentioned in \cite[Section 4]{AB71}. There the energy related to momentum is not $T_J(\rho)$ but $\frac{1}{2} \int_{\mathbb{R}^3} \rho(x) L\left(m_\rho(r(x))\right) r^{-2}(x) \,dx$, in which the term $r^{-2}(x)$ will be singular when $r(x) \rightarrow 0$. Therefore, the authors require $\sigma(x)=0$ when $r(x)$ is small.
\end{remark}



We are now ready to prove the differentiability of $E_J $ at local minimizers:

\begin{lemma}[Differentiability at Minimizers]\label{diff. at mini.}
	If $\rho$ minimizes $E_J(\rho)$ locally on ${R}\left(\mathbb{R}^3\right)$, then $E_J(\rho)$ is finite, $\rho \in W$. In particular, by Lemma \ref{diff. of energy}, we know $E_J(\rho)$ is $P_{\infty}(\rho)$-differentiable at $\rho$, and the derivative at $\rho$ is $E_J^{\prime}(\rho)$ given in (\ref{variational derivative}).
\end{lemma}

\begin{proof}
	{
		\rm
		Since $\int_{\mathbb{R}^3} \rho\,dx=1>0$, similar to Remark \ref{equivalence of finite total energy and finite internal energy} we know $T_J(\rho)$  and $G(\rho, \rho)$ are finite. Since $\rho$ is a local minimizer, $\exists \delta>0$, such that $\forall \sigma \in {R}\left(\mathbb{R}^3\right)$ with $W^{\infty}(\rho, \sigma)<\delta$, one has $E_J(\rho) \leq E_J(\sigma)$. Thanks to Lemma \ref{properties of Wasser.} ($vi$), take $\delta$ small enough, we can find a $\sigma \in {R}\left(\mathbb{R}^3\right) \cap L^{\infty}\left(\mathbb{R}^3\right)$, with $W^{\infty}(\rho, \sigma)< \delta$ and $\int_{\mathbb{R}^3} \sigma\,dx>0$. Similar as above we have $G(\sigma, \sigma)<\infty, T_J(\sigma)<\infty$. Thanks to Lemma \ref{finite internal energy}, we know $U(\sigma)<\infty$, and then $E_J(\rho) \leq E_J(\sigma)<\infty$, which implies $U(\rho)<\infty$. Then we have $\rho \in W$.
	}
\end{proof}

\subsection{Locally constant Lagrange multiplier}\label{subsection5.2-locally constant chemical potential}

In this subsection, we essentially follow McCann's arguments \cite[Section 5]{McC06} but with some refinements to prove ``weak'' versions of Theorem \ref{Properties of LEM} Parts ($iv$-$v$). The ``weak'' version serves as the bridge linking the norm of $\rho$ and the norm of the potential $V_\rho$, which can help to show the global continuity of the $W^\infty$ local minimizer $\rho$ as we have seen in the proof of Theorem \ref{Properties of LEM}

\begin{proposition}[Locally Constant Lagrange Multiplier {\cite[Proposition 5.3]{McC06}}] \label{local constant multiplier}
	Let $\rho \in {R}_0\left(\mathbb{R}^3\right)$ minimize $E_J(\kappa)$ among $\kappa \in {R}\left(\mathbb{R}^3\right)$ for which $W^{\infty}(\rho, \kappa) \leq 2 \delta$. Let $M$ be an open set with positive measure and with diameter no greater than $2 \delta$ which intersects spt $\rho$. There is a unique $\lambda \in \mathbb{R}$ depending on $M$ such that (\ref{EL}) holds on $M$ a.e.
\end{proposition}

\begin{proof}
	{\rm
		We refer to \cite{McC06}, but slightly modify the arguments and provide a more detailed account (see Remark \ref{improvement about local chemical potential}). Define the convex cone $P_{l o c}:=\left\{\sigma \in P_{\infty} \mid \text{spt } \sigma \subset M\right\}$, which is a cone, and let $U=\{\rho \in X \mid \rho\geq 0, U(\rho)< \left.\infty, G(\rho, \rho)<\infty, T_J(\rho)<\infty\right\}$ be the same as in subsection \ref{subsection5.1-variational derivative}. Thanks to Lemma \ref{diff. of energy} and Lemma \ref{diff. at mini.}, we have $\rho \in U, E_J(\rho)$ is $P_{l o c}$-differentiable at $\rho$. Given $\kappa \in W_{l o c}={R}\left(\mathbb{R}^3\right) \cap\{\rho+\sigma \in U \mid \sigma \in P_{l o c}\}$, where the mass constraint is satisfied, thanks to Lemma \ref{properties of Wasser.} ($i$) we have $W^{\infty}(\rho, \kappa) \leq \operatorname{diam} M \leq 2 \delta$, which shows that $\rho$ minimizes $E_J(\kappa)$ on $W_{l o c}$. In particular, $\rho$ is a weak local minimizer restricted on $W_{loc}$ (see Appendix \ref{sectionA-calculus of variations in vector spaces} for the definition). Moreover, since the open set $M$ intersects spt $\rho$, it must carry positive mass under $\rho$, see for example \cite[Section 3.4, Exercise 25]{Fol13}. Therefore, similar to Remark \ref{test function with minimizer bounded}, there is a smaller subset $C \subset M$ of positive measure on which $\rho(x)$ is bounded away from zero and infinity. Let $\mathbf{1}_C$ be the indicator function of set $C$, then by construction we have both $\pm \mathbf{1}_{C} \in P_{l o c}$. For all $\sigma \in P_{loc}$, let $w=\sigma-\frac{\int_{\mathbb{R}^3} \sigma\,dx}{\int_{\mathbb{R}^3} \mathbf{1}_C\,dx}\mathbf{1}_C$, thanks to Remark \ref{nonnegative perturbation}, we know $\rho+tw \in W_{loc}$ when $t$ is small enough. Thanks to Theorem \ref{Multiplier them} and Remark \ref{subset of linear space}, the conditions above imply that there is a unique Lagrange multiplier $\lambda \in R$ such that $\forall \sigma \in P_{l o c}$, we have
		\begin{equation}\label{lag. them. in 1.4}
			\int_{\mathbb{R}^3} E_J^{\prime}(\rho) \sigma\,dx \geq \lambda \int_{\mathbb{R}^3} \sigma\,dx
		\end{equation}
		
		Claim: 
		\begin{enumerate}
			\item[(1)] $
			E_J^{\prime}(\rho)(x)=A^{\prime}(\rho(x))-V_\rho(x)-\frac{J^2}{2 I^2(\rho)} r^2(x-\bar{x}(\rho)) \geq \lambda \quad \text{on } M \text{ a.e.}$
			\item[(2)] $
			E_J^{\prime}(\rho)=\lambda \quad \text{on } M \cap\{x \mid \rho(x)>0\} \text{ a.e.}$
		\end{enumerate}
		
		In fact, if $E_J^{\prime}(\rho)<\lambda$ on a subset $K \subset M$ which has positive measure, similar to Remark \ref{test function with positive measure support}, this subset may be taken slightly smaller so that $\rho$ is bounded on $K$. Then we take $\sigma=\mathbf{1}_K \in P_{loc}$ and would see it contradict (\ref{lag. them. in 1.4}), thus Claim (1) holds true. On the other hand, if $E_J^{\prime}(\rho)>\lambda$ on a subset $K \subset M$ with positive measure and $\rho>0$ in $K$, then similarly $K$ may be taken slightly smaller so that $\rho$ is bounded away from zero and infinity on $K$; in this case $-\mathbf{1}_K \in P_{l o c}$ can make a contradiction to (\ref{lag. them. in 1.4}) again, thus Claim (2) holds true.
		
		By \eqref{A', A and P relation} we know $A'(0)=0$, and $A'(s)>0$ when $s>0$. If $V_\rho(x)+\frac{J^2}{2 I^2(\rho)} r^2(x-\bar{x}(\rho))+\lambda>0$, then by Claim (1) $A^{\prime}(\rho(x)) \geq V_\rho(x)+\frac{J^2}{2 I^2(\rho)} r^2(x-\bar{x}(\rho))+\lambda>0$, which implies $\rho>0$. By Claim (2) we can replace inequality by equation and get $A^{\prime}(\rho(x))=V_\rho(x)+\frac{J^2}{2 I^2(\rho)} r^2(x-\bar{x}(\rho))+\lambda$. If $V_\rho(x)+ \frac{J^2}{2 I^2(\rho)} r^2(x-\bar{x}(\rho))+\lambda \leq 0$, then $\rho=0$, otherwise $\rho>0$ implies again $0< A^{\prime}(\rho(x))=V_\rho(x)+\frac{J^2}{2 I^2(\rho)} r^2(x-\bar{x}(\rho))+\lambda \leq 0$, which leads to a contradiction. Therefore, we have (\ref{EL}), that is $A^{\prime}(\rho(x))=\left[V_\rho(x)+\frac{J^2}{2 I^2(\rho)} r^2(x-\bar{x}(\rho))+\lambda\right]_{+}$ holds for almost all $x \in M$. Since $M$ has positive measure and $\rho$ in $M$ has positive mass, we know there exists a $x_0$ such that  $\rho(x_0)>0$ and thus $A^{\prime}(\rho(x_0))>0$, and (\ref{EL}) holds at $x_0$. In particular, $0<A^{\prime}(\rho(x_0))=V_\rho(x_0)+\frac{J^2}{2 I^2(\rho)} r^2(x_0-\bar{x}(\rho))+\lambda$. Hence, it is clear that $\lambda$ is unique and given by this expression.
	}
\end{proof}

\begin{remark}\label{improvement about local chemical potential}
	Compared to McCann's proof, we have added the following discussions:
	
	1.	When deriving inequality (\ref{lag. them. in 1.4}), McCann simply referenced Theorem \ref{Multiplier them}. However, the scope of Theorem \ref{Multiplier them} does not fully cover the situation here. Therefore, we supplemented the proof with Remark \ref{subset of linear space} and utilized both Theorem \ref{Multiplier them} and Remark \ref{subset of linear space} to reach the conclusion.
	
	2.	After deriving Claim (1) and Claim (2), McCann directly asserted Proposition \ref{local constant multiplier}'s conclusion that  (\ref{EL}) holds on $M$ a.e. In contrast, we added an explanation detailing how this conclusion is obtained.
\end{remark}

\begin{remark}\label{local but not weak local}
	Recall that in Section \ref{section4-Wasserstein metric}, we know $\rho+t w \in {R}\left(\mathbb{R}^3\right)$ may not be $W^{\infty}$-continuous as function of $t$. In other words, whether a local minimizer is also a weak local minimizer remains unclear in general. However, at least we know a minimizer restricted to $W_{loc}$ is a weak local minimizer restricted to $W_{loc}$, so that we can apply Theorem \ref{Multiplier them} and Remark \ref{subset of linear space}. This is why we first consider the local case Proposition \ref{local constant multiplier} (the diameter of $M$ is small) before addressing the next case Proposition \ref{constant multiplier}.
\end{remark}

\begin{proposition}[Componentwise Constant Lagrange Multiplier {\cite[Proposition 5.4]{McC06}}] \label{constant multiplier}
	Let $\rho \in {R}_0\left(\mathbb{R}^3\right)$ minimize $E_J(\kappa)$ among $\kappa \in {R}\left(\mathbb{R}^3\right)$ for which $W^{\infty}(\rho, \kappa)<2 \delta$. $\Omega_i$ is one of the connected components of the $\delta$-neighbourhood of spt $\rho$. Then there is a constant $\lambda_i<0$ such that (\ref{EL}) holds a.e. on $\Omega_i$.
\end{proposition}

\begin{proof}
	{
		\rm
		Given $y \in \Omega_i$, we know the ball $B_\delta(y)$ intersects spt $\rho$, where $B_\delta(y)$ is the open ball defined by $B_\delta(y):=\left\{x \in \mathbb{R}^3\mid | x-y| <\delta\right\}$. Thus Proposition \ref{local constant multiplier} guarantees a unique $\lambda(y)$ such that (\ref{EL}) holds a.e. on $B_\delta(y)$. The claim is that $\lambda(y)$ is independent of $y$. In fact, fix $y_0 \in \Omega_i$, since $B_\delta(y_0)$ is open, it will also be true that a slightly smaller ball $B_{\delta-\epsilon}(y_0)$ intersects spt $\rho$. If $|x-y_0|<\epsilon$, then $M=B_\delta(x) \cap B_\delta(y_0)$ intersects spt $\rho$ since $B_{\delta-\epsilon}(y_0) \subset M$, and then we know $x\in \Omega_i$. $M$ has positive measure, thanks to Proposition \ref{local constant multiplier}, the uniqueness of $\lambda$ corresponding to $M$ forces $\lambda(x)=\lambda(y_0)$. Thus $\lambda(y)$ is locally constant. As a result, the disjoint sets $C=\left\{x \in \Omega_i \mid \lambda(x)=\lambda(y_0)\right\}$ and $D=\left\{x \in \Omega_i \mid \lambda(x) \neq \lambda(y_0)\right\}$ are both open. Since $\Omega_i=C \cup D$ is connected and $y_0 \in C$ implies $C$ is not empty, $C=\Omega_i$. Defining $\lambda_i:=\lambda(y_0)$,(\ref{EL}) must be satisfied a.e. on $\Omega_i$.
		
		Now we show $\lambda<0$: The distance between any point on the boundary of $\Omega_i$ and spt $\rho$ cannot be smaller than $\delta$. Since spt $\rho$ is bounded, $\Omega_i$ has non-empty boundary, and it follows that $\rho(x)=0$ on a set of positive measure in $\Omega_i$. On the other hand, $A^{\prime}(\rho) \rho- A(\rho)=P(\rho)$ implies $A^{\prime}(\rho(x))$ vanishes only if $\rho(x)=0$. $\lambda \geq 0$ in (\ref{EL}) would imply $\rho>$ 0 a.e. on $\Omega_i$, a contradiction.
	}
\end{proof}

\begin{remark}\label{connected components and delta neighborhood}
	Notice we discuss above the connected components of the $\delta$-neighbourhood of spt $\rho$. But the result can be easily converted to the case about the $\delta$-neighbourhood of connected components of spt $\rho$.
\end{remark}

\begin{remark}\label{2 ways to show negative multiplier 1}
	From the above proof, we can see the power of Wasserstein metric, which can extend the Euler-Lagrange equation (\ref{EL}) to a $\delta$-neighborhood of connected component of spt $\rho$, which help us to show the global continuity of $\rho$ and then (\ref{EP'}) holds as we discuss in Theorem \ref{Properties of LEM}. Moreover, it plays a role in the process of proving the Lagrange multiplier is negative. Actually, there is another way to show Lagrange multipliers should be negative, given by Auchmuty and Beals \cite{AB71}, with different conditions. 
\end{remark}

\subsection{Minimizers' existence and nonexistence}\label{subsection5.3-minimizers' existence and nonexistence}
In Section \ref{section4-Wasserstein metric}, we know $\rho+t w \in {R}\left(\mathbb{R}^3\right)$ may not be $W^{\infty}$-continuous as a function of $t$, and whether a local minimizer is also a weak local minimizer remains unclear in general. For comparison, we prove a local minimizer is a weak local minimizer under the topology inherited from a topological vector space. Furthermore, we demonstrate that, in this case, a local minimizer with finite energy does not exist.

Denote $X = L^{\frac{4}{3}}\left( \mathbb{R}^{3} \right)\cap L^{1}\left( \mathbb{R}^{3} \right)$, $\widetilde{U}=\{\rho \in X \mid \rho \geq 0\}$. We first show in the topology inherited from the topological vector space, a local minimizer is a weak local minimizer.

\begin{lemma}\label{local implies weak local}
	If the topology on ${R}\left(\mathbb{R}^3\right)$ is inherited from a topological vector space, and $\rho$ is a local minimizer for $E_J(\rho)$ restricted on ${R}\left(\mathbb{R}^3\right)$, then $\rho$ is a weak local minimizer restricted on ${R}\left(\mathbb{R}^3\right)$, which is defined in Appendix \ref{sectionA-calculus of variations in vector spaces}.
\end{lemma}

\begin{proof}
	{
		\rm
		We denote $\left. N\left( \mathbb{R}^{3} \right) = \left\{ {\sigma \in L^{\frac{4}{3}}\left( \mathbb{R}^{3} \right)} \right|\int_{\mathbb{R}^3}\sigma \,dx= 0 \right\}$. We recall the setup in Appendix \ref{sectionA-calculus of variations in vector spaces}, that $P(\rho) = \left\{ {\sigma \in X} \right|\exists\epsilon(\sigma),~\text{such that}~\rho + t\sigma \in \widetilde{U},~0 \leq t < \epsilon(\sigma)\}$.
		
		Since the topology on $R(\mathbb{R}^3)$ is inherited from a topological vector space, we know there is a topological vector space $Y$ with topology $\mathcal{T}$, such that $R(\mathbb{R}^3)$ is a subset of $Y$, with the collection $\mathcal{T}_{R} = \left\{ R\left( \mathbb{R}^{3} \right) \cap O \middle| O \in \mathcal{T} \right\}$ being the topology on $R(\mathbb{R}^3)$. (See \cite[Chapter 1, Section 16]{Mun00}). 
		
		Since $\rho$ is a local minimizer restricted on $R(\mathbb{R}^3)$, then $\exists O_R\in \mathcal{T}_R$, such that $\rho\in O_R$ and $E_J (\rho)\leq E_J (\kappa)$ for all $\kappa \in O_R$. Notice that $O_R=R(\mathbb{R}^3 )\cap O$ for some $O\in \mathcal{T}$, and thus $\rho \in O$. Given $\sigma \in P(\rho)\cap N(\mathbb{R}^3)$, we know $\rho+t\sigma\in R(\mathbb{R}^3 )\subset Y$. Since $Y$ is a topological vector space, we know when $t$ is small enough, $\rho+t\sigma \in O$. Therefore, $\rho+t\sigma\in R(\mathbb{R}^3 )\cap O=O_R$, and then $E_J (\rho)\leq E_J(\rho+t\sigma)$ for $t$ small enough. Then we know $\rho$ is a weak local minimizer restricted on $R(\mathbb{R}^3)$.  
	}
\end{proof}

\begin{proposition}[No Local Minimizers with Finite Energy in a Vector Space Topology {\cite[Remark 3.7]{McC06}}]\label{no local minimizers with finite energy in a vector space topology}
	For $J>0$, if the topology on ${R}\left(\mathbb{R}^3\right)$ is inherited from a topological vector space, then there is no local minimizer for $E_J$ restricted on ${R}\left(\mathbb{R}^3\right)$, with $E_J(\rho)<\infty$.
\end{proposition}

\begin{remark}
	Note that, unlike in McCann's statement \cite[Remark 3.7]{McC06}, here we first assume the local minimal energy $E_J(\rho)$ is finite, so that we can mimic the arguments in Lemma \ref{diff. of energy} and then make a contradiction.
\end{remark}

\begin{proof}[Proof of Proposition \ref{no local minimizers with finite energy in a vector space topology}]
	{
		\rm
		Suppose $\rho$ is a local minimizer restricted on $R(\mathbb{R}^3)$, with $E_J(\rho)<\infty$. Thanks to Remark \ref{equivalence of finite total energy and finite internal energy}, we know $U(\rho)<\infty$, that is, $A(\rho)$ is integrable. Therefore, similar arguments to Lemma \ref{diff. of energy}, we know $P_\infty (\rho)\subset P(\rho)=\{\sigma \in X \mid \exists \epsilon(\sigma), \text{such that~} \rho+t\sigma \in \widetilde{U}, 0 \leq t< \epsilon(\sigma)\}$, and $E_J(\rho)$ is $P_\infty (\rho)$-differentiable at $\rho$.
		Thanks to Lemma \ref{local implies weak local}, we know $\rho \in{R}\left(\mathbb{R}^3\right)$ is a weak local minimizer restricted on $R(\mathbb{R}^3)$. Let $\sigma \in P_\infty(\rho)$, similar to the discussion in Proposition \ref{local constant multiplier}, we can again derive (\ref{lag. them. in 1.4}) here. Notice there are enough functions $\sigma$ we obtain (\ref{EL})
		$$
		A^{\prime}(\rho(x))=\left[\frac{J^2}{2 I^2(\rho)} r^2(x-\bar{x}(\rho))+V_\rho(x)+\lambda\right]_{+}
		$$
		is satisfied a.e. on $\mathbb{R}^3$ for a fixed $\lambda$. 
		Then we have for almost all $x\in \mathbb{R}^3$,
		$$
		A^{\prime}(\rho(x)) \geq \frac{J^2}{2 I^2(\rho)} r^2(x-\bar{x}(\rho))+\lambda
		$$
		Since $A'$ is strictly increasing, this inequality implies for any $C>0$, $\exists R>0$, such that $\rho(x)>C$ almost everywhere in the region $r(x)>R$ , which contradicts the fact that $\rho \in {R}\left(\mathbb{R}^3\right)$ has mass 1.
	}
\end{proof}


\begin{remark}
	Inspired by McCann's Remark \cite[Remark 3.7]{McC06}, we can also understand physically Proposition \ref{no local minimizers with finite energy in a vector space topology} as the following: suppose there is a small portion of mass in a rotating star which is moved to a very distant location, that is, from $\kappa=\rho(x)+\sigma(x)$ to $\kappa_N=\rho(x)+\sigma(x-N)$ for $N$ large. Then, intuitively, as $N \rightarrow \infty$, the internal energy $U$ and gravitational interaction energy $G$ will not change too much, while $I(\kappa_N) \rightarrow \infty$, which implies the kinetic energy $T_J \rightarrow 0$, leading to the decrease of $E_J$ in the end.
\end{remark}

\begin{remark}\label{R instead R_0}
	In McCann’s discussion \cite[Remark 3.7]{McC06}, the conclusion considers $\rho$ as a local minimizer restricted to $R_0(\mathbb{R}^3)$ rather than $R(\mathbb{R}^3)$. However, as revealed by the proof of Proposition \ref{no local minimizers with finite energy in a vector space topology}, if we focus solely on the restriction to $R_0(\mathbb{R}^3)$, it becomes challenging to demonstrate the existence of sufficiently many functions $\sigma$ to ensure that the  (\ref{EL}) holds. Consequently, it is difficult to establish the non-existence of such $\rho$.
\end{remark}

\begin{remark}
	From the proof above we notice that the issue comes from the fact that even if $\sigma(x)$ remains positive when $x$ is far away from spt $\rho$, $\rho + t\tilde{\sigma}$ is still continuous with respect to $t$, which implies local minimizer $\rho$ is also a weak local minimizer. Fortunately, this situation does not occur if ${R}\left(\mathbb{R}^3\right)$ is topologized via the Wasserstein $L^{\infty}$ metric, because continuity of $\rho + t\tilde{\sigma}$ is not guaranteed under this topology, as discussed in Section \ref{section4-Wasserstein metric}. 
%
\end{remark}

From the proof of Proposition \ref{no local minimizers with finite energy in a vector space topology}, we obtain the nonexistence of local minima with finite energy $E_J(\rho)$. The finiteness assumption is essential for following the arguments in Lemma \ref{diff. of energy} to compute the derivative of $U(\rho)$. However, if the minimal energy $E_J(\rho)$ blows up, i.e., becomes infinite, the arguments in Lemma \ref{diff. of energy} can no longer be utilized.
%

The following Proposition \ref{existence of weak local minima when energy blows up} leads to a conjecture that the statement regarding local minimizer (w.r.t. topology inherited from the topological vector space) having finite energy might not hold. Note that this proposition differs from Lemma \ref{diff. at mini.}, making a comparison interesting.

We first make an additional assumption of $P(\rho)$, which helps us to compare $A(\rho)$ and $A(\lambda\rho)$:
\begin{enumerate}
	\item[\mylabel{F5}{(F5)}] There exists $\lambda \in (0,1)$, such that 
	\begin{equation}\label{formula of F5}
		{\limsup\limits_{\rho\rightarrow\infty}\frac{\int_{\lambda\rho}^{\rho}{P(\tau)\tau^{- 2}d\tau}}{\int_{0}^{\lambda\rho}{P(\tau)\tau^{- 2}d\tau}}} < \infty
	\end{equation}
\end{enumerate}

\begin{remark}
	One can check this condition will be satisfied in the case $P(\rho)=K\rho^\gamma$.
\end{remark}

\begin{remark}\label{large lambda}
	Notice that if we can find a $\lambda$ such that (\ref{formula of F5}) holds, then $\forall \widetilde{\lambda}\in (\lambda,1)$, (\ref{formula of F5}) also holds.
\end{remark}

\begin{proposition}[Existence of Weak Local Minimizer when Energy Blows Up]\label{existence of weak local minima when energy blows up}
	If $P(\rho)$ in addition satisfies \ref{F5}, and $\rho \in R_0 (\mathbb{R}^3)$ satisfies $E_J (\rho)=\infty$. Then $\forall \sigma \in P(\rho)=\{\sigma \in X \mid \exists \epsilon(\sigma), \text{such that~} \rho+t\sigma \in \widetilde{U}, 0 \leq t< \epsilon(\sigma)\}$, $\exists \delta>0$, such that $\forall t\in (0,\delta)$, we have $E_J(\rho+t\sigma)=\infty$. In particular, $\rho$ is a weak local minimizer (see Appendix \ref{sectionA-calculus of variations in vector spaces}) for $E_J (\rho)$ on $R(\mathbb{R}^3)$. 
\end{proposition}

\begin{proof}
	{
		\rm
		If $E_J (\rho)=\infty$, due to Remark \ref{equivalence of finite total energy and finite internal energy}, we know $U(\rho) = {\int_{\mathbb{R}^3}{A(\rho)}\,dx} = \infty$.
		Since $\rho$ has compact support, $A(s)$ is an increasing function of $s$, and $\rho=0$ if and only if $A(\rho)=0$, then we know $\forall N>0$,
		$$
		{\int_{\{{\rho \leq N}\}}{A(\rho)}\,dx} \leq A(N){\int_{\{{\rho \leq N}\}}1\,dx} \leq A(N){\int_{spt~\rho}1\,dx} = A(N)\mu\left( {spt~\rho} \right) < \infty
		$$
		where $\mu$ is Lebesgue measure.
		
		Since we know ${\int_{\mathbb{R}^3}{A(\rho)}\,dx} = {\int_{\{{\rho \leq N}\}}{A(\rho)}\,dx} + {\int_{\{{\rho > N}\}}{A(\rho)}\,dx} = \infty$, we know ${\int_{\{{\rho > N}\}}{A(\rho)}\,dx} = \infty$ for all $N>0$.
		
		On the other hand, suppose $\sigma\in P(\rho)$, then we know $\forall t\in(0,\epsilon(\sigma))$, we have $\rho+t\sigma\geq 0$, then $\sigma \geq - \frac{\rho}{t} \geq - \frac{\rho}{\epsilon(\sigma)}$. If $t\rightarrow0^+$, notice $A(s)$ is a increasing function, we have for all $N$,
		$$U\left( {\rho + t\sigma} \right) = {\int_{\mathbb{R}^3}{A(\rho + t\sigma)}\,dx} \geq {\int_{\mathbb{R}^3}{A\left( {\left( {1 - \frac{t}{\epsilon(\sigma)}} \right)\rho} \right)}\,dx} \geq {\int_{\{{\rho > N}\}}{A\left( {\left( {1 - \frac{t}{\epsilon(\sigma)}} \right)\rho} \right)}\,dx}$$
		Since $t>0$, then $\lambda = 1-\frac{t}{\epsilon(\sigma)} <1$, and ${\lim\limits_{t\rightarrow 0^{+}}\lambda} = 1$.
		
		Due to \ref{F5}, we denote $L \coloneq {\limsup\limits_{\rho\rightarrow\infty}\frac{\int_{\lambda\rho}^{\rho}{P(\tau)\tau^{- 2}d\tau}}{\int_{0}^{\lambda\rho}{P(\tau)\tau^{- 2}d\tau}}} <\infty$, and claim: $\exists \eta\in(0,1)$ and $N_1>0$, such that if $\lambda\in(\eta,1)$ and $\rho>N_1$, we have $A(\lambda \rho)\geq \frac{\lambda}{L+1} A(\rho)$ (for the case $P(\rho)=K\rho^\gamma$ with $\gamma>\frac43$ and hence $A(\rho)=\frac{K}{\gamma-1}\rho^{\gamma}$ (cf. (\ref{AwithP})), the inequality can be verified directly).
		
		In fact, given $\lambda<1$, by the definition of $A$ (\ref{A}), we have 
		$$
		A\left( {\lambda\rho} \right) = \lambda\rho{\int_{0}^{\lambda\rho}{P(\tau)\tau^{- 2}d\tau}} = \lambda\rho{\int_{0}^{\rho}{P(\tau)\tau^{- 2}d\tau}} - \lambda\rho{\int_{\lambda\rho}^{\rho}{P(\tau)\tau^{- 2}d\tau}} = \lambda A(\rho) - \lambda\rho{\int_{\lambda\rho}^{\rho}{P(\tau)\tau^{- 2}d\tau}}
		$$
		Due to Remark \ref{large lambda}, we can find $\eta\in (0,1)$ and $N_1>0$, such that if $\lambda \in (\eta,1)$ and $\rho>N_1$, we have
		$$
		\lambda\rho{\int_{\lambda\rho}^{\rho}{P(\tau)\tau^{- 2}d\tau}} \leq \lambda L\rho{\int_{0}^{\lambda\rho}{P(\tau)\tau^{- 2}d\tau}} = \lambda L\rho{\int_{0}^{\rho}{P(\tau)\tau^{- 2}d\tau}} - \lambda L\rho{\int_{\lambda\rho}^{\rho}{P(\tau)\tau^{- 2}d\tau}}
		$$.
		Then we have
		$$\lambda\rho{\int_{\lambda\rho}^{\rho}{P(\tau)\tau^{- 2}d\tau}} \leq \frac{\lambda L}{L + 1}\rho{\int_{0}^{\rho}{P(\tau)\tau^{- 2}d\tau}} = \frac{\lambda L}{L + 1}A(\rho)$$
		Therefore, by (\ref{A}), we obtain the claim:
		$$
		A\left( {\lambda\rho} \right) = \lambda A(\rho) - \lambda\rho{\int_{\lambda\rho}^{\rho}{P(\tau)\tau^{- 2}d\tau}} \geq \frac{\lambda}{L + 1}A(\rho)
		$$
		
		Due to this claim, we know $\exists \delta >0$, such that when $t\in (0, \delta)$, we have
		$$
		U\left( {\rho + t\sigma} \right) \geq {\int_{\{{\rho > N_1}\}}{A\left( {\left( {1 - \frac{t}{\epsilon(\sigma)}} \right)\rho} \right)}\,dx} \geq \frac{1}{L + 1}\left( {1 - \frac{t}{\epsilon(\sigma)}} \right){\int_{\{{\rho > N_1}\}}{A(\rho)}\,dx} = \infty
		$$
		Thanks to Proposition \ref{finite gravitational interaction energy}, we know $G(\rho+t\sigma, \rho+t\sigma)<\infty$, then we know $E_J(\rho+t\sigma)=\infty$ for $t\in (0,\delta)$.
	}
\end{proof}

\begin{remark}\label{weak local mini. but no topo. needed}
	Notice that, in Proposition \ref{existence of weak local minima when energy blows up}, we do not require the topology to be inherited from a topological vector space. In fact, when defining a weak local minimizer (see Appendix \ref{sectionA-calculus of variations in vector spaces}), we do not utilize any specific properties of any topology. 
	
%
\end{remark}
	\section*{Appendix}\label{Appendix}
	\appendix
	Here we introduce some preliminary knowledge that is used in the paper.
	\section{Calculus of Variations in Vector Spaces}\label{sectionA-calculus of variations in vector spaces}
	Since the functional we deal with is not differentiable in any usual sense, we begin with a less restrictive notion of differentiability.
	
	Let $X$ be a real vector space, $U$ a subset of $X$, and $E$ an extended real-valued function defined on $U$. Given $u \in U$, let $P(u)=\{v \in X \mid \exists \epsilon(v)>0, \text{ such that } u+t v \in U \text{ for } t\in [0, \epsilon(v))\}$. This set is a cone in $X$. If $P$ is any cone in $X$, let $X(P)$ be the linear subspace of $X$ generated by $P$. We say that $E$ is $P$\emph{-differentiable} at $u$ if $P \subset P(u)$ and there is a linear functional $l$ on $X(P)$ such that $\forall v \in P$
	$$
	\lim\limits_{t \rightarrow 0^{+}} t^{-1}\{E(u+t v)-E(u)\}=l(v),
	$$
	The linearity of $l$ is needed in our analysis; its precise domain is not essential here. We denote the ``derivative'' $l$ by $E_u^{\prime}$—with $P$ understood from context.

	We say $E$ has a \emph{weak local minimum} at $u \in U$ if for each $v \in P(u)$, there exists $\{t_n \}$ such that $t_n >0$, ${\lim\limits_{n\rightarrow\infty}t_{n}} = 0$ and $E\left( {u + t_{n}v} \right) \geq E(u)$ for $n$ large enough. And $u$ is called \emph{weak local minimizer}
	
	\begin{remark}
		If $u$ is a weak local minimizer, then one can check
		$$
		\lim _{t \rightarrow 0^{+}} \inf E(u+t v) \geq E(u)
		$$
		In fact, the latter corresponds to the definition of weak local minimizer in \cite[Section 2]{AB71}. However, this limited-based definition does not fully align with our intuitive understanding of a “minimizer”. For instance, consider $E(u) = -u^2$; according to this limit-based definition, $u = 0$ would qualify as a weak local minimizer. Yet, this is not a reasonable conclusion. Our definition, on the other hand, helps exclude such case.
		
	\end{remark}
	
	It is only through our definition of weak local minimizer that we can truly establish the following proposition in \cite[Section 2]{AB71}:
	
	\begin{proposition}[\ {\cite[Section 2]{AB71}}\ ] \label{nonnegative derivative}
		Suppose $E$ has a weak local minimum at $u$ and is $P$-differentiable at $u$. Then $E_u^\prime(v)\geq 0$ for all $v\in P$.
	\end{proposition}
	
	\begin{proof}
		{
			\rm
			The proof can be established using a proof by contradiction, or by referring to Theorem \ref{Multiplier them} below, specifically its proof for $E_u^{\prime}(w) \geq 0$.
		}
	\end{proof}
	
	\begin{remark}
		Generally speaking, if $X$ is a topological vector space, that is, $u+tv$ is continuous w.r.t. $t$, then one can prove that a local minimizer is a weak local minimizer. See, for example, Lemma \ref{local implies weak local}. This is also why we use the term ``weak''.
		
	\end{remark}
	
	To check $u$ is a weak local minimizer, we just need to show $E(u)$ is a minimum over $E(u+t_nv)$ for some ${t_n}$ in any arbitrary direction $v\in P(u)$, rather than for all small $t$ in the direction $v$. Moreover, we do not need to specify the topology or find a neighborhood of $\rho$. See also Remark \ref{weak local mini. but no topo. needed}.
	
	Although it is ``weak'', a weak local minimizer retains the possibility of computing or analyzing derivatives, as we can see in Proposition \ref{nonnegative derivative} and Theorem \ref{Multiplier them}.
	
	Let $W = \left\{ u \in U \mid g(u) = M \right\}$, where $g$ is a linear functional on $X$ and $M$ is a constant. Then similarly, we say $E$ has a \emph{weak local minimum at $u$ restricted to $W$} if $u \in W$ and for each $v\in P(u)$, there exists ${t_n}$ such that $t_n >0$, ${\lim\limits_{n\rightarrow\infty}t_{n}} = 0$, and for $n$ large enough, we have $u+t_{n}v\in W$\footnote{Let $N = \left\{ u \in U \mid g(u) = 0 \right\}$, we can check ``equivalent to $v\in P(u)$ such that $u+tv\in W$ for $t$ small enough'' is ``$v \in P(u)\cap N$'' . However, when considering the extended result in the Remark \ref{subset of linear space}, this criterion can be different and not equivalent, and it requires a more detailed discussion, similar to what is done in the proof of Proposition \ref{local constant multiplier}.}, and $E\left( {u + t_{n}w} \right) \geq E(u)$. Furthermore, we have a result for weak local minimum restricted to $W$:

	\begin{theorem}[(Generalized) Lagrange Multiplier Theorem {\cite[Section 2]{AB71}}] \label{Multiplier them}
		Suppose $E$ is $P$-differentiable at $u \in U$ and that its restriction to $W$ has a weak local minimum at $u$. Suppose also that $P$ is convex and that there is a $u_0 \in P$ such that $-u_0 \in P$ and $g\left(u_0\right) \neq 0$. Then there is a unique  constant $\lambda$ such that $E_u^{\prime}(v) \geq \lambda g(v)$, all $v \in P$.
	\end{theorem}
	
	\begin{proof}
		{
			\rm
			Given $v\in P$, let $w=v-\frac{g(v)}{g(u_0)}u_0$, then $g(w)=0$. Since $P$ is a cone and $\pm u_0\in P$, we know $\frac{g(v)}{g(u_0)}u_0 \in P$. Furthermore, since $P$ is convex and $v\in P$, we know  $\frac{1}{2}v-\frac{1}{2}\frac{g(v)}{g(u_0)}u_0 \in P$. Using again the fact $P$ is a cone, we have $w=v-\frac{g(v)}{g(u_0)}u_0 \in P \subset P(u)$.
			
			Therefore, by the definition of $P(u)$, we know $u+tw\in U$ for small positive $t$, thus $u+tw \in W$. $E$ is $P$-differentiable at $u \in U$, then we know $E_u^{\prime}(w)$ exists. 
			
			Claim: $E_u^{\prime}(w) \geq 0$.
			
			In fact, if $E_u^{\prime}(w) < 0$, by definition of $E_u^{\prime}(w)$ we know exists $\delta>0$, such that if $0<t<\delta$, we have
			$$
			E\left( {u + tw} \right) < E(u)
			$$
			
			On the other hand, $u$ is a weak local minimizer restricted to $W$, then we know $\exists {t_n}$ such that $t_n >0$, ${\lim\limits_{n\rightarrow\infty}t_{n}} = 0$ and $E\left( {u + t_{n}w} \right) \geq E(u)$ for $n$ large enough. It leads to a contradiction. 
			
			Therefore, $E_u^{\prime}(w) \geq 0$. Then we know 
			$$
			0\leq E_u^{\prime}(w)=E_u^{\prime}(v)-\frac{g(v)}{g(u_0)}E_u^{\prime}(u_0)
			$$
			Thus we may take $\lambda=\frac{E_u^{\prime}(u_0)}{g(u_0)}=E_u^{\prime}(\frac{u_0}{g(u_0)})$, and get
			$$
			E_u^{\prime}(v) \geq \lambda g(v)
			$$
			On the other hand, if $\lambda$ does satisfy the above, then $E_u^{\prime}(u_0)\geq \lambda g(u_0)$ and $E_u^{\prime}(-u_0)\geq -\lambda g(u_0)$. Thus $\lambda$ has to be $E_u^{\prime}(\frac{u_0}{g(u_0)})$ and unique.
		}
	\end{proof}
	
	\begin{remark}\label{subset of linear space}
		{
			If we go through the proof above, we can find that $W$ can be a subset of $\left\{ u \in U \middle| g(u) = M \right\}$. At this point, we need to ensure that $u+t_nw$ in the proof above indeed lies within $W$ when $n$ is sufficiently large.
		}
	\end{remark}
	
	\section{Properties of Sobolev Spaces}\label{sectionB-properties of sobolev spaces}
	We review some properties of Sobolev spaces and $L^p$ spaces, complemented by proofs for certain propositions. Some of them will be used frequently in this paper.
	
%
	

	First we notice that Hardy-Littlewood-Sobolev Inequality \cite[Theorem 1.7]{BCD11} can fail when we consider the $L^\infty$ norm of the convolution. To estimate the bound in $L^{\infty}$, we introduce the following proposition:
	
	\begin{proposition}[Bound of Potential {\cite[Proposition 5]{AB71}}] \label{bound of potential}
		Suppose $\rho \in L^{1}\left(\mathbb{R}^{3}\right) \cap L^{p}\left(\mathbb{R}^{3}\right)$. If $1<p \leq \frac{3}{2}$, then $\forall r \in\left(3, \frac{3 p}{3-2 p}\right), V_{\rho} \in L^{r}\left(\mathbb{R}^{3}\right)$, and $\exists 0<b_{r}<1,0<c_{r}<1, C>0$, such that
		\begin{equation} \label{BoVs}
			\left\|V_{\rho}\right\|_{L^{r}} \leq C\left(\|\rho\|_{L^{1}}^{b_{r}}\|\rho\|_{L^{p}}^{1-b_{r}}+\|\rho\|_{L^{1}}^{c_{r}}\|\rho\|_{L^{p}}^{1-c_{r}}\right) 
		\end{equation}
		
		If $p>\frac{3}{2}$, then $V_{\rho}$ is bounded and continuous and satisfies (\ref{BoVs}) with $r=\infty$.
	\end{proposition}
	
	\begin{proof}
		{
			\rm
			One can see the detailed proof for the case $1<p \leq \frac{3}{2}$ in \cite{AB71}. For the case $p>$ $\frac{3}{2}$, we can use the same strategy: let $b_{1}(x)=\frac{1}{|x|} \cdot \mathbf{1}_{\{|x|<1\}}$, where $\mathbf{1}_{A}(x)=\left\{\begin{array}{ll}1, & x \in A \\ 0, & x \notin A\end{array}\right.$ is the indicator function of $A$. $b_{2}(x)=\frac{1}{|x|} \cdot \mathbf{1}_{\{|x| \geq 1\}}$. It is easy to see $b_{1} \in L^{p}\left(\mathbb{R}^{3}\right)$ for $1 \leq p<3$, $b_{2} \in L^{p}\left(\mathbb{R}^{3}\right)$ for $3<p \leq \infty$, and $\frac{1}{|x|}=b_{1}(x)+b_{2}(x)$. Therefore,
			
			$$
			\begin{aligned}
				\left\|V_{\rho}\right\|_{L^{\infty}\left(\mathbb{R}^{3}\right)} & =\left\|\frac{1}{|\cdot|} * \rho\right\|_{L^{\infty}\left(\mathbb{R}^{3}\right)} \\
				& \leq\left\|b_{1} * \rho\right\|_{L^{\infty}\left(\mathbb{R}^{3}\right)}+\left\|b_{2} * \rho\right\|_{L^{\infty}\left(\mathbb{R}^{3}\right)} \\
				& \leq\|\rho\|_{L^{p-\epsilon}} \cdot\left\|b_{1}\right\|_{L^{(p-\epsilon)^{\prime}}}+\|\rho\|_{L^{1+\epsilon}} \cdot\left\|b_{2}\right\|_{L^{(1+\epsilon)^{\prime}}}
			\end{aligned}
			$$
			The last inequality comes from Young's Inequality {\cite[Theorem 4.33 and Exercise 4.30]{Bre11}} with $\frac{1}{p-\epsilon}+\frac{1}{(p-\epsilon)^{\prime}}=1$ and $\frac{1}{1+\epsilon}+$ $\frac{1}{(1+\epsilon)^{\prime}}=1$. Since $p>\frac{3}{2}$, we can choose a small $\epsilon$ such that those norms are finite. And then apply the Interpolation Inequality \cite[Section 4.2]{Bre11} for $\|\rho\|_{L^{p-\epsilon}}$ amd $\|\rho\|_{L^{1+\epsilon}}$ to obtain (\ref{BoVs}). Moreover, we can check $\forall x \in \mathbb{R}^{3}, V_{\rho}(x)$ is well defined and finite. (The idea is similar to that of the following proof of continuity.) Let $\tau_{h} f(x)=f(x+h)$, then
			
			$$
			\begin{aligned}
				\left|V_{\rho}(x+h)-V_{\rho}(x)\right|=&\left|\int_{\mathbb{R}^3} \frac{\rho(x+h-y)-\rho(x-y)}{|y|} \,dy\right| \\
				\leq&\left|\int_{\mathbb{R}^3}\left(\tau_{h} \rho(x-y)-\rho(x-y)\right) b_{1}(y) \,dy\right| \\
				& +\left|\int_{\mathbb{R}^3}\left(\tau_{h} \rho(x-y)-\rho(x-y)\right) b_{2}(y) \,dy\right| \\
				\leq&\left\|\tau_{h} \rho-\rho\right\|_{L^{p}} \cdot\left\|b_{1}\right\|_{L^{p^{\prime}}}+\left\|\tau_{h} \rho-\rho\right\|_{L^{1}} \cdot\left\|b_{2}\right\|_{L^{\infty}}
			\end{aligned}
			$$
			The last inequality comes from Hölder's inequality. Notice $p>\frac{3}{2}$ thus its Hölder conjugate $p^{\prime}<3$ thus $\left\|b_{1}\right\|_{L^{p^{\prime}}}<\infty$. $\left\|\tau_{h} \rho-\rho\right\|_{L^{p}}$ and $\left\|\tau_{h} \rho-\rho\right\|_{L^{1}}$ go to 0 when $h \rightarrow 0$. Therefore, $V_{\rho}$ is continuous and $\left\|V_{\rho}\right\|_{L^{\infty}\left(\mathbb{R}^{3}\right)}$ is finite. Thus $V_{\rho}$ is bounded, i.e. $\exists C>0, \forall x \in \mathbb{R}^3, |V_{\rho}(X)|<C$.
		}
	\end{proof}

	Notice in the proof above we replace ``almost everywhere'' $\left(\left\|V_{\rho}\right\|_{L^{\infty}\left(\mathbb{R}^{3}\right)}<C\right)$ by ``everywhere'' $\left( |V_{\rho}(x)|<C\right)$. The following lemma gives another statement which connects weak derivative and classical derivative, which is a supplement to Morrey Inequality (see \cite[Theorem 9.12]{Bre11} or \cite[Subsection 5.6.2]{Eva10}).
	
	\begin{lemma}[Continuous Differentiable Representative of Potential] \label{continuous differentiable representative of potential}
		If $V \in W^{1, \infty}\left(\mathbb{R}^{3}\right)$, and its weak derivative and itself are continuous. Then $V$ has a representative which is continuously differentiable and the weak derivative coincides with the classical one a.e. in $\mathbb{R}^{3}$.
	\end{lemma}
	
	\begin{proof}
		{
			\rm
			$V \in W^{1, \infty}\left(\mathbb{R}^{3}\right)$. Therefore, $\forall R>0, V$ has a representative which is Lipschitz in $B_{R}(0)$ \cite[Subsection 5.8.2]{Eva10}, and then for almost every $x \in \mathbb{R}^{3}, V$ is differentiable and the weak derivative coincides with the classical one by Rademacher's theorem \cite[Subsection 5.8.3]{Eva10}. To replace ``almost everywhere'' by ``everywhere'', we modify the discussion about Sobolev space in one dimension in \cite[Section 8.2]{Bre11}. Without loss of generality, we assume $j=1$, and consider $\widetilde{f}\left(\widetilde{x_1}, x_{2}, x_{3}\right)=\int_{0}^{x_{1}} \frac{\partial V}{\partial x_{1}}\left(y, x_{2}, x_{3}\right) \,d\widetilde{x_1}$, where $\frac{\partial V}{\partial x_{1}}$ is continuous by assumption, one can show its weak derivative and classical derivative with respect to $x_{1}$ are $\frac{\partial V}{\partial x_{1}}$ both in $\mathbb{R}^{3}$ and in $\mathbb{R}$. Therefore, $g:=\widetilde{f}-V$ is continuous and has 0 as its weak partial derivative. Consider $g_{n}=\rho_{n} * g$, where $\left\{\rho_{n}\right\}$ is a sequence of mollifiers, the classical derivative $\frac{\partial g_{n}}{\partial x_{1}}=\frac{\partial \rho_{n}}{\partial x_{1}} * g$ \cite[Proposition 4.20]{Bre11}, while by the definition of weak derivative we have $\frac{\partial \rho_{n}}{\partial x_{1}} * g(x)=\int_{\mathbb{R}^3} \frac{\partial \rho_{n}}{\partial x_{1}}(x-y) g(y) d y=\int_{\mathbb{R}^3} \rho_{n}(x-y) \frac{\partial g}{\partial x_{1}}(y) d y=0$. Therefore, for all point in $\mathbb{R}^{3}, g_{n}\left(x_{1}, x_{2}, x_{3}\right)=c_{n}\left(x_{2}, x_{3}\right)$. Because $g_{n} \rightarrow g$ uniformly on compact sets of $\mathbb{R}^{3}$ \cite[Proposition 4.21]{Bre11}, in particular converges everywhere, $g\left(x_{1}, x_{2}, x_{3}\right)=$ $c\left(x_{2}, x_{3}\right)$ for some function $c$. Therefore, $\widetilde{f}\left(x_{1}, x_{2}, x_{3}\right)=c\left(x_{2}, x_{3}\right)+V_{\rho}\left(x_{1}, x_{2}, x_{3}\right)$ everywhere in $\mathbb{R}^{3}$. Then we take the classical derivative with respect to $x_{1}$ again to see $V_{\rho}$'s classical derivative is $\frac{\partial V_{\rho}}{\partial x_{1}}$, which is also continuous.
		}
	\end{proof}

	Let $h(x)=-\int_{\mathbb{R}^3} \frac{y_{j} \rho(x-y)}{|y|^{3}} dy$. Similar to the arguments in proof of Proposition \ref{bound of potential}, we know $h$ is in $L^q(\mathbb{R}^3)$ for $\rho \in L^{1}\left(\mathbb{R}^{3}\right) \cap L^p(\mathbb{R}^3)$ under appropriate indices $p$ and $q$. Consequently, through the use of test functions and Fubini's theorem, it follows that $h$ is the weak derivative of $V_{\rho}$, i.e., the following result holds in the sense of distributions:
	$$\frac{\partial V_{\rho}}{\partial x_{j}}(x)=-\int_{\mathbb{R}^3} \frac{y_{j} \rho(x-y)}{|y|^{3}} \,dy.$$
	Therefore, combining the arguments and statements as above, we have:
	
	\begin{proposition}[Differentiability of Potential {\cite[Proposition 7]{AB71}}] \label{diff. of poten.}
		If $\rho \in L^{1}\left(\mathbb{R}^{3}\right) \cap L^{p}\left(\mathbb{R}^{3}\right)$ for some $p>3$, then $V_{\rho} \in$ $W^{1, \infty}\left(\mathbb{R}^{3}\right)$ is continuously differentiable and the weak derivative coincides with the classical one for all $x \in \mathbb{R}^{3}$.
	\end{proposition}
	
	\begin{proof}
		{
			\rm
			Similar to the arguments in proof of Proposition \ref{bound of potential}, we know $\frac{\partial V_{\rho}}{\partial x_{j}} \in L^{\infty}\left(\mathbb{R}^{3}\right) \cap C\left(\mathbb{R}^{3}\right)$, and together with Proposition \ref{bound of potential} we have $V_{\rho} \in$ $W^{1, \infty}\left(\mathbb{R}^{3}\right) \cap C\left(\mathbb{R}^{3}\right)$. Therefore, by Lemma \ref{continuous differentiable representative of potential} we get the results.
		}
	\end{proof}
	
	\begin{remark}
		One can also prove Proposition \ref{diff. of poten.} by direct computation.
	\end{remark}

%

\section*{Acknowledgments}\label{section-acknowledgments}
\addcontentsline{toc}{section}{Acknowledgments}
The author is partially supported by the National Science Foundation grant DMS-2308208. This work was primarily carried out during the author’s Master’s studies at the University of Bonn. The author thanks Juan Velázquez and Dimitri Cobb for their continued advice and support since the author’s time in Bonn, as well as Christof Sparber and Mimi Dai for their comments and support during the author's Ph.D. studies at the University of Illinois Chicago. Thanks also to Lorenzo Pompili, Shao Liu, Xiaopeng Cheng, Bernhard Kepka, Daniel Sánchez Simón del Pino for discussions, and to Théophile Dolmaire and other instructors. The author is grateful to his parents.


\bibliographystyle{abbrv}
\addcontentsline{toc}{section}{References}
\bibliography{references}

\end{document}